\documentclass[aos,preprint]{imsart}

\RequirePackage[OT1]{fontenc}
\RequirePackage{amsthm,amsmath}
\RequirePackage[authoryear]{natbib}
\usepackage[colorlinks,citecolor=blue,urlcolor=blue]{hyperref}

\usepackage{graphicx}
\usepackage{xargs}[2008/03/08]
\usepackage{mathrsfs,amsbsy,amssymb,amsfonts}
\usepackage{multirow}
\usepackage{enumitem}
\usepackage{MnSymbol}
\usepackage{epstopdf}

\usepackage{stackrel}

\pubyear{2019}
\volume{47}
\issue{6}
\firstpage{3533}
\lastpage{3577}
\arxiv{arxiv:1812.01831}
\doi{10.1214/18-AOS1787}

\makeatletter
\newtheorem{thm}{Theorem}
\newtheorem{prop}[thm]{Proposition}

\newtheorem{lem}[thm]{Lemma}

\newcommand{\linlabel}[2]{#2\def\@currentlabel{#2}\label{#1}}
\newcommand{\mylabel}[2]{#2\def\@currentlabel{#2}\label{#1}}
\newcommand{\vertiii}[1]{{\left\vert\kern-0.25ex\left\vert\kern-0.25ex\left\vert #1 
		\right\vert\kern-0.25ex\right\vert\kern-0.25ex\right\vert}}
	
\newcommand{\lincomment}[1]{{#1}}

\DeclareFontFamily{OMX}{MnSymbolE}{}
\DeclareSymbolFont{MnLargeSymbols}{OMX}{MnSymbolE}{m}{n}
\SetSymbolFont{MnLargeSymbols}{bold}{OMX}{MnSymbolE}{b}{n}
\DeclareFontShape{OMX}{MnSymbolE}{m}{n}{
	<-6>  MnSymbolE5
	<6-7>  MnSymbolE6
	<7-8>  MnSymbolE7
	<8-9>  MnSymbolE8
	<9-10> MnSymbolE9
	<10-12> MnSymbolE10
	<12->   MnSymbolE12
}{}
\DeclareFontShape{OMX}{MnSymbolE}{b}{n}{
	<-6>  MnSymbolE-Bold5
	<6-7>  MnSymbolE-Bold6
	<7-8>  MnSymbolE-Bold7
	<8-9>  MnSymbolE-Bold8
	<9-10> MnSymbolE-Bold9
	<10-12> MnSymbolE-Bold10
	<12->   MnSymbolE-Bold12
}{}

\let\llangle\@undefined
\let\rrangle\@undefined
\DeclareMathDelimiter{\llangle}{\mathopen}%
{MnLargeSymbols}{'164}{MnLargeSymbols}{'164}
\DeclareMathDelimiter{\rrangle}{\mathclose}%
{MnLargeSymbols}{'171}{MnLargeSymbols}{'171}
\makeatother

\global\long\def\expect{\mathbb{E}}
\global\long\def\prob{\mathrm{Pr}}

\global\long\def\real{\mathbb{R}}
\global\long\def\Op{O_{P}}
\global\long\def\manifold{\mathcal{M}}

\global\long\def\op{o_{P}}
\global\long\def\oas{o_{a.s.}}

\global\long\def\covarop{\mathcal{C}}

\global\long\def\asympeq{\asymp}

\global\long\def\diffop{\mathrm{d}}
\newcommandx\tangentspace[2][usedefault, addprefix=\global, 1=\manifold]{T_{#2}#1}

\global\long\def\dim{d}

\global\long\def\vec#1{\mathbf{#1}}
\newcommandx\lpnorm[3][usedefault, addprefix=\global, 1=r, 2=]{\|#3\|_{\mathcal{L}^{#1}}^{#2}}
\newcommandx\lp[1][usedefault, addprefix=\global, 1=p]{\mathcal{L}^{#1}}

\global\long\def\ltwo{\mathcal{L}^{2}}

\global\long\def\innerprod#1#2{\langle#1,#2\rangle}
\global\long\def\metric#1#2{\langle#1,#2\rangle}
\global\long\def\Log{\mathrm{Log}}
\global\long\def\Exp{\mathrm{Exp}}
\newcommandx\vfnorm[3][usedefault, addprefix=\global, 1=\mu, 2=]{\|#3\|_{#1}^{#2}}
\newcommandx\vfinnerprod[2][usedefault, addprefix=\global, 1=\mu]{\llangle#2\rrangle_{#1}}
\global\long\def\define{:=}

\global\long\def\tdomain{\mathcal{T}}
\global\long\def\tm{T\manifold}

\newcommandx\opnorm[3][usedefault, addprefix=\global, 1=\mu, 2=]{\vertiii{#3}_{#1}^{#2}}
\newcommandx\fronorm[2][usedefault, addprefix=\global, 1=]{|#2|_{F}^{#1}}
\global\long\def\spd{\mathrm{Sym_{\star}^{+}}(m)}
\global\long\def\borel{\mathscr{B}}
\global\long\def\usphere{\mathbb{S}^{2}}
\global\long\def\usphered{\mathbb{S}^{d}}

\global\long\def\sym{\mathrm{Sym}(m)}
\def\rev#1{#1}
\def\underset#1#2{\stackbin[#1]{}{#2}}
\def\overset#1#2{\stackbin[]{#1}{#2}}

\def\references{\bibliography{fda,manifold,misc}}

\makeatother

\begin{document}

\begin{frontmatter}
\title{Intrinsic Riemannian Functional Data Analysis}
\runtitle{{i}RFDA}

\begin{aug}
\author{\fnms{Zhenhua} \snm{Lin}\thanksref{m1}\ead[label=e1]{stalz@nus.edu.sg}}
\and
\author{\fnms{Fang} \snm{Yao}\thanksref{t2,m2}\ead[label=e2]{fyao@math.pku.edu.cn}}

\thankstext{t2}{Fang Yao's research is partially supported by National Natural Science Foundation of China Grant
	11871080, a Discipline Construction Fund at Peking University and Key Laboratory of Mathematical
	Economics and Quantitative Finance (Peking University), Ministry of Education. Data were provided
	by the Human Connectome Project, WU-Minn Consortium (Principal Investigators: David Van Essen
	and Kamil Ugurbil; 1U54MH091657) funded by the 16 NIH Institutes and Centers that support the
	NIH Blueprint for Neuroscience Research, and by the McDonnell Center for Systems Neuroscience
	at Washington University.}

\runauthor{Z. Lin and F. Yao}

\affiliation{National University of Singapore\thanksmark{m1} and Peking University\thanksmark{m2}}
\address[A]{Z. Lin
	\\ Department of Statistics and Applied Probability\\ National University of Singapore \\ Singapore 117546\\
	\printead{e1}}
\address[C]{F. Yao
	\\ Corresponding author \\ Department of Probability and Statistics \\ Center for Statistical Science \\ Peking University, Beijing, China\\
	\printead{e2}}

\end{aug}

\received{\smonth{10} \syear{2017}}
\revised{\smonth{10} \syear{2018}}

%
\begin{abstract}
	In this work we develop a novel and foundational framework for analyzing 
	general Riemannian functional data, in particular a new development of tensor Hilbert spaces along curves on a manifold. {Such spaces enable us to derive} Karhunen--Lo\`{e}ve expansion for Riemannian random processes. This framework also features {an approach to compare} objects from different tensor Hilbert spaces, which paves the way
	for asymptotic analysis in Riemannian functional data analysis. Built upon  intrinsic
	geometric concepts such as vector field, Levi--Civita connection and
	parallel transport on Riemannian manifolds, the developed framework
	{applies to not only Euclidean submanifolds but also manifolds without a natural ambient space}. {As applications of this framework, we develop intrinsic Riemannian
		functional principal component analysis (iRFPCA) and intrinsic Riemannian
		functional linear regression (iRFLR) that are distinct from their traditional and ambient counterparts}.  We also provide  estimation procedures for iRFPCA and iRFLR,  and investigate their asymptotic properties within the intrinsic geometry. Numerical performance is illustrated by simulated and real examples.
\end{abstract}

%


\end{frontmatter}

\section{Introduction\label{sec:iRFDA-Introduction}}
\rev{Functional data analysis (FDA) advances substantially} in the past two decades, as the  rapid
development of modern technology enables collecting more and more
data continuously over time. There is rich literature spanning more
than seventy years on this topic, including development on functional
principal component analysis such as \citet{Rao1958,Kleffe1973,Dauxois1982,Silverman1996,Yao2005a,Hall2006,Zhang2016},
and advances on functional linear regression such as \citet{Yao2005b,Hall2007c,Yuan2010,Kong2016},
among many others. For a thorough review of the topic, we refer readers to the review article 
\citet{Wang2016} and monographs \citet{Ramsay2005,Ferraty2006,Hsing2015,Kokoszka2017}
for comprehensive treatments on classic functional data analysis. 
Although traditionally functional data take values in a vector space, more data of nonlinear  structure arise and should be properly handled in a nonlinear space. For instance, trajectories of bird migration are  naturally regarded as curves on a sphere which is a nonlinear Riemannian manifold, rather than the three-dimensional vector space $\real^3$. Another \mbox{example} is the dynamics of brain  functional connectivity. The functional connectivity at a time point is represented by a symmetric positive-definite matrix (SPD). Then the dynamics shall be modelled as a curve in the space of SPDs that is endowed with either the affine-invariant metric \citep{Moakher2005} or the Log-Euclidean metric \citep{Arsigny2007} to avoid the ``swelling'' effect \citep{Arsigny2007}. Both metrics turn SPD into a nonlinear Riemannian manifold.  In this paper, we refer this type of functional data as  \emph{Riemannian functional data}, which are functions
taking values on a Riemannian manifold and modelled by \emph{Riemannian random processes}, that is, we treat Riemannian
trajectories as realizations of a Riemannian random process.

Analysis of Riemannian functional data is not only challenged by the infinite dimensionality and compactness of covariance operator from functional data, but also obstructed by the \emph{nonlinearity} of the range of functions, since manifolds are generally
not vector spaces and render many techniques relying on linear
structure ineffective \rev{or  inapplicable. For instance},
if the sample mean curve is computed for bird migration trajectories as if they were sampled from the ambient space $\real^{3}$,
this na\"ive sample mean in general does not fall on the sphere of earth. 
\rev{For manifolds of tree-structured data studied in \cite{Wang2007}, as they are naturally not Euclidean submanifolds which refer to Riemannian submanifolds of a Euclidean space in this paper, the na\"ive sample mean can not even be defined}
from ambient spaces, and thus a proper treatment of
manifold structure is necessary. While the literature for Euclidean functional data is abundant, works
involving nonlinear manifold structure are scarce. 
\citet{Chen2012}
and \citet{Lin2017a} respectively investigate representation and
regression for functional data living in a low-dimensional nonlinear  
manifold that is embedded in an infinite-dimensional space, while
\citet{Lila2017} focuses principal component analysis on functional
data whose domain is a two-dimensional manifold. None of these deal with functional data that take values on a nonlinear manifold, while \citet{Dai2017} is the only endeavor in this direction for Euclidean submanifolds.

As functional principal component analysis (FPCA) is an essential
tool for FDA, it is of im-portance and
interest to develop this notion for Riemannian functional data.
Since manifolds are in general not vector spaces, classic
covariance functions/operators do not exist naturally for a Riemannian
random process.
A strategy that is often adopted, for example, \cite{Shi2009} and \citet{Cornea2017},
to overcome the lack of vectorial structure is to \rev{map data on the manifold into tangent spaces via Riemannian logarithm map defined in Section \ref{subsec:Random-Elements-on-THS}.} As tangent spaces at different points are different vector
spaces, in order to handle observations from different tangent spaces,
some existing works assume
a Euclidean ambient space for the manifold and identify tangent vectors as Euclidean vectors. \rev{This strategy is adopted by \citet{Dai2017} on Riemannian functional data such as compositional data modelled on the unit sphere for the first time. Specifically, they assume that functional data are sampled from a time-varying geodesic submanifold, where at a given time point, the functions take values on a geodesic submanifold of a common manifold. Such a common manifold is further assumed to be a Euclidean submanifold that allows to identify all tangent spaces as hyperplanes in a common Euclidean space (endowed with the usual Euclidean inner product). Then, with the aid of Riemannian logarithm map, \citet{Dai2017} are able to transform Riemannian functional data into Euclidean ones while accounting for the intrinsic curvature of the underlying manifold.}  

\rev{To avoid confusion, we distinguish two different perspectives to deal with Riemannian manifolds. One is to work with the manifold under consideration without assuming an ambient space surrounding it or an isometric embedding into a Euclidean space. This perspective is regarded as \emph{completely intrinsic}, or simply \emph{intrinsic}. Although generally difficult to work with, it can fully respect all geometric structure of the manifold. The other one, referred to as \emph{ambient} here, assumes that the manifold under consideration is isometrically embedded in a Euclidean ambient space, so that geometric objects such as tangent vectors can be processed within the ambient space. For example, from this point of view, the local polynomial regression for SPD proposed by \cite{Yuan2012} is intrinsic, while  the aforementioned work by \citet{Dai2017} takes the ambient perspective.}

\rev{Although it is possible to account for some of geometric structure in the ambient perspective, for example, the curved nature of manifold via Riemannian logarithm map, several issues arise due to manipulation of geometric objects such as tangent vectors in the ambient space.} First, the essential dependence on an ambient space restricts
potential applications. It is not
immediately applicable to manifolds that are not a Euclidean submanifold or do not have a natural isometric embedding into a Euclidean
space, for example, the Riemannian manifold of $p\times p$ ($p\geq 2$) SPD matrices \rev{endowed with the affine-invariant metric \citep{Moakher2005} which is not compatible with the $p(p+1)/2$-dimensional Euclidean metric}. 
Second, although an ambient space provides
a common stage for tangent vectors at different points, operation
on tangent vectors from this ambient perspective can potentially violate the
intrinsic geometry of the manifold. To illustrate this, consider
comparison of two tangent vectors at different points (this comparison
is needed in the asymptotic analysis of Section \ref{subsec:Asymptotics-fpca}; see also Section \ref{subsec:euclidean-submanifold}).
From the ambient perspective, taking the difference
of tangent vectors requires moving a tangent vector parallelly
\emph{within the ambient space} to the base point of the other tangent
vector. However, the resultant tangent vector after movement in the ambient space is
generally not a tangent vector for the base point of the other tangent
vector; see the left panel of Figure \ref{fig:parallel-transport}
for a geometric illustration. In another word, the ambient difference of
two tangent vectors at different points is not an intrinsic geometric
object on the manifold, and the departure from intrinsic geometry can
potentially affect the statistical efficacy and/or efficiency. Lastly, since manifolds
might be embedded into more than one ambient space, the interpretation
of statistical results  crucially
depends on the ambient space and could be misleading if one does not choose the ambient space appropriately.

\begin{figure}[t]
	\begin{minipage}[t]{0.49\columnwidth}%
		\begin{center}
			\includegraphics[scale=0.65]{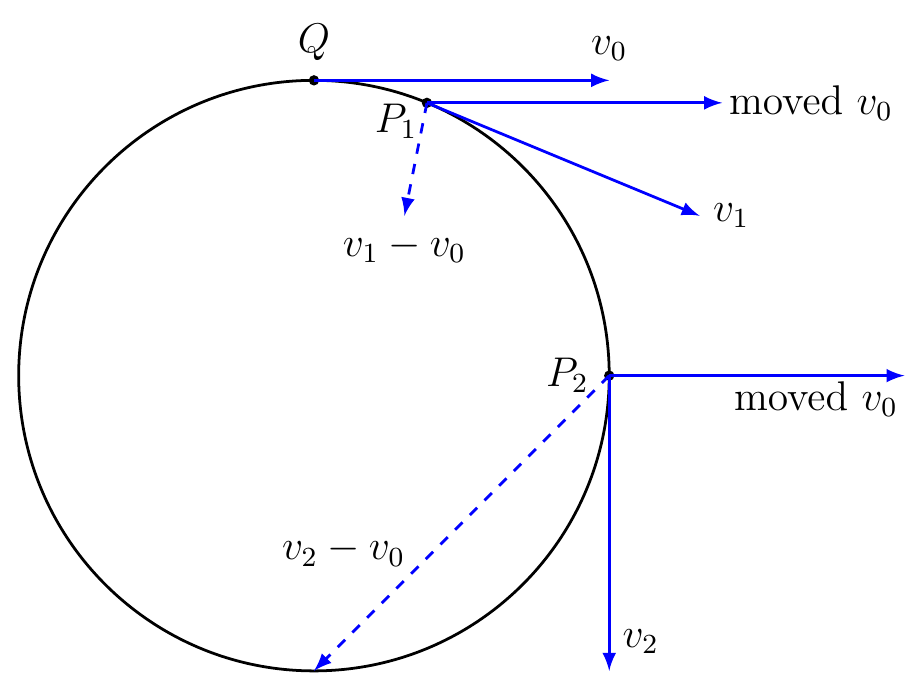}
			\par\end{center}%
	\end{minipage}%
	\begin{minipage}[t]{0.49\columnwidth}%
		\begin{center}
			\includegraphics[scale=0.65]{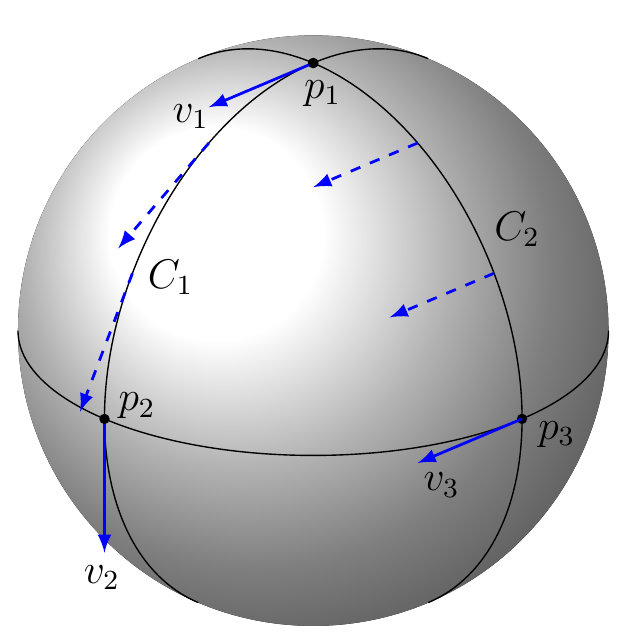}
			\par\end{center}%
	\end{minipage}
	
	\caption{Left panel: illustration of ambient movement of tangent vectors. The
		tangent vector $v_{0}$ at the point $Q$ of a unit circle embedded
		in a Euclidean plane is moved to the point $P_1$ and $P_2$ within the ambient
		space. $v_{1}$ (resp. $v_2$) is a tangent vector at $P_1$ (resp. $P_2$). The differences $v_{1}-v_{0}$ and $v_2-v_0$
		are not tangent to the circle at $P_1$ and $P_2$, respectively. If $v_0$, $v_1$ and $v_2$ have the same length, then the intrinsic parallel transport of $v_0$ to $P_k$ shall coincide with $v_k$, and $\mathcal{P}v_0-v_k=0$, where $k=1,2$ and $\mathcal{P}$ represents the parallel transport on the unit circle with the canonical metric tensor. Thus, $\|\mathcal{P}v_0-v_k\|_{\real^2}=0$. However, $\|v_0-v_k\|_{\real^2}>0$, and this nonzero value completely results from the departure of the Euclidean geometry from the unit circle geometry. The ambient discrepancy $\|v_0-v_1\|_{\real^2}$ is small as $P_1$ is close to $P$, while  $\|v_0-v_2\|_{\real^2}$ is large since  $P_2$  is far away from $Q$. Right panel: illustration of
		parallel transport. A tangent vector $v_{1}$ at the point $p_{1}$
		on the unit sphere is parallelly transported to the point $p_{2}$
		and $p_{3}$ along curves $C_{1}$ and $C_{2}$, respectively. During
		parallel transportation, the transported tangent vector always stays
		within the tangent spaces along the curve. }
	\label{fig:parallel-transport}
\end{figure}

In the paper, we develop a completely intrinsic framework that provides a 
foundational theory for general Riemannian functional data that paves the way for the development of 
intrinsic Riemannian functional principal component analysis and intrinsic
Riemannian functional linear regression, among other potential
applications. The key building block is a new concept
of \emph{tensor Hilbert space} along a curve on the manifold, which
is described in Section \ref{sec:Tensor-Hilbert-Space}. 
On one hand, our 
approach experiences dramatically elevated technical challenges relative
to the ambient counterparts. For example, without an ambient space,
it is nontrivial to perceive and handle tangent vectors. On the
other hand, the advantages of the intrinsic perspective are at least
three-fold, in contrast to ambient approaches. First, our results immediately apply to many important
Riemannian manifolds that are not naturally a Euclidean submanifold but commonly seen in statistical analysis and machine
learning, such as the aforementioned  SPD manifolds  
and Grassmannian manifolds. Second, our
framework features a novel intrinsic proposal for coherent comparison
of objects from different tensor Hilbert spaces on the manifold, and
hence makes the asymptotic analysis sensible. Third,
results produced by our approach are invariant to embeddings and ambient
spaces, and can be interpreted independently, which avoids potential misleading interpretation in practice.

As important applications of the proposed framework, we develop intrinsic Riemannian functional principal component
analysis (iRFPCA) and intrinsic Riemannian functional linear regression
(iRFLR). Specifically, estimation procedures of intrinsic 
eigenstructure are provided and their asymptotic properties are investigated within the intrinsic geometry. For iRFLR, we study a Riemannian functional linear
regression model, where a scalar response intrinsically and linearly
depends on a Riemannian functional predictor through a \emph{Riemannian
	slope function}, a concept that is formulated in Section \ref{sec:Intrinsic-Regression},
along  with the concept of linearity in the context of Riemannian
functional data. We present an FPCA-based estimator and a Tikhonov estimator
for the Riemannian slope function and explore their asymptotic properties, where
the proposed framework of tensor Hilbert space again plays an essential role.

The rest of the paper is structured as follows. The 
foundational framework for intrinsic Riemannian functional data analysis is laid
in Section \ref{sec:Tensor-Hilbert-Space}. Intrinsic Riemannian functional
principal component analysis is presented in Section \ref{sec:Intrinsic-Representation},
while intrinsic Riemannian functional regression is studied in Section
\ref{sec:Intrinsic-Regression}. In Section \ref{sec:Numerical-Evidence}, numerical performance is illustrated through simulations, and an application to Human Connectome Project analyzing functional connectivity and behavioral data is provided.

\section{Tensor Hilbert space and Riemannian random process\label{sec:Tensor-Hilbert-Space}}

In this section, we first define the concept of tensor Hilbert space
and discuss its properties, including a mechanism to deal with two different
tensor Hilbert spaces at the same time. Then, random elements on tensor
Hilbert space are investigated, with the proposed  intrinsic Karhunen--Lo\`{e}ve expansion
for the random \mbox{elements}. Finally, practical computation with respect
to an orthonormal frame is given. Throughout this section, we assume
a $d$-dimensional, connected and geodesically complete Riemannian
manifold $\manifold$ equipped with a Riemannian metric $\innerprod{\cdot}{\cdot}$,
which defines a scalar product $\innerprod{\cdot}{\cdot}_{p}$ for
the tangent space $\tangentspace p$ at each point $p\in\manifold$.
This metric also induces a distance function $d_{\manifold}$ on $\manifold$.
A \mbox{preliminary} for Riemannian manifolds can be found in the appendix.
For a comprehensive treatment on Riemannian manifolds, we recommend
the introductory text by \citet{Lee1997} and also \citet{Lang1995}.

\subsection{Tensor Hilbert spaces along curves\label{subsec:Tensor-Hilbert-Spaces-along-cruves}}

Let $\mu$ be a measurable curve on a manifold  $\manifold$ and parameterized by a compact domain
$\tdomain\subset\real$ equipped with a finite measure $\upsilon$. A vector field
$V$ along $\mu$ is a map from $\tdomain$ to the tangent bundle
$T\manifold$ such that $V(t)\in\tangentspace{\mu(t)}$ for all $t\in\tdomain$.
It is seen that the collection of vector fields $V$ along $\mu$
is a vector space, where the vector addition between two vector fields $V_1$ and $V_2$ is a vector field $U$ such that $U(t)=V_1(t)+V_2(t)$ for all $t\in\tdomain$, and the scalar multiplication between a real number $a$ and a vector field $V$ is a vector field $U$ such that $U(t)=aV(t)$ for all $t\in\tdomain$. Let $\mathscr{T}(\mu)$ be the collection of (equivalence
classes of) measurable vector fields $V$ along $\mu$ such that
$\vfnorm V\define\{\int_{\tdomain}\metric{V(t)}{V(t)}_{\mu(t)}\diffop\upsilon(t)\}^{1/2}<\infty$
with identification between $V$ and $U$ in $\mathscr{T}(\mu)$ (or equivalently, $V$ and $U$ are in the same equivalence class) 
when $\upsilon(\{t\in\tdomain:V(t)\neq U(t)\})=0$. Then $\mathscr{T}(\mu)$ is turned
into an inner product space by the inner product $\vfinnerprod[\mu]{V,U}\define\int_{\tdomain}\metric{V(t)}{U(t)}_{\mu(t)}\diffop\upsilon(t)$,
with the induced norm given by $\vfnorm[\mu]{\cdot}$. Moreover, we
have that
\begin{thm}
	\label{prop:hilbertian-vfc}For a measurable curve $\mu$ on $\manifold$, $\mathscr{T}(\mu)$ is a separable Hilbert
	space.
\end{thm}
We call the space $\mathscr{T}(\mu)$ the \emph{tensor Hilbert space} along $\mu$, as tangent vectors are a special type of tensor
and the above Hilbertian structure can be defined for tensor fields
along $\mu$. The above theorem is of paramount importance, in the sense that it suggests $\mathscr{T}(\mu)$ to serve as a cornerstone for Riemannian functional data analysis for two reasons. First, as shown in Section \ref{subsec:Random-Elements-on-THS}, \rev{via Riemannian logarithm maps, a Riemannian random process may be transformed into a tangent-vector-valued random process (called log-process in Section \ref{subsec:Random-Elements-on-THS}) that can be regarded as a random element in a tensor Hilbert space. Second, the rigorous theory of functional data analysis formulated in \cite{Hsing2015} by random elements in separable Hilbert spaces fully applies to the log-process.} 

One distinct feature of the tensor Hilbert space is
that, different curves that are even parameterized by the same domain
give rise to different tensor Hilbert spaces. In practice, one often
needs to deal with two different tensor Hilbert spaces at the same
time. For example, in the next subsection we will see that under some
conditions, a Riemannian random process $X$ can be conceived as a
random element on the tensor Hilbert space $\mathscr{T}(\mu)$ along
the intrinsic mean curve $\mu$. However, the mean curve is often
unknown and estimated from a random sample of $X$. Since the sample
mean curve $\hat{\mu}$ generally does not agree with the population
one, quantities of interest such as covariance operator and their
sample versions are defined on two different tensor Hilbert spaces $\mathscr{T}(\mu)$ and 
$\mathscr{T}(\hat{\mu})$, respectively. For
statistical analysis, one needs to compare the sample quantities
with their population counterparts and hence involves objects such as covariance operators
from two different tensor Hilbert spaces. 

In order to intrinsically quantify the discrepancy between objects
of the same kind from different tensor Hilbert spaces, we utilize
the Levi--Civita connection (p. 18, \citealt{Lee1997})  associated with the Riemannian manifold
$\manifold$. The Levi--Civita connection is uniquely
determined by the Riemannian structure. It is the only torsion-free
connection compatible with the Riemannian metric. Associated
with this connection is a unique parallel transport operator $\mathcal{P}_{p,q}$
that smoothly transports tangent vectors at $p$ along a curve to
$q$ and preserves the inner product. We shall emphasize that the parallel transportation is performed intrinsically. For instance, tangent vectors being transported always stay tangent to the manifold during  transportation, which is illustrated by the right panel of Figure \ref{fig:parallel-transport}. 
Although transport operator $\mathcal{P}_{p,q}$ depends on the curve
connecting $p$ and $q$, there exists a canonical choice of the curve
connecting two points, which is the minimizing geodesic between $p$
and $q$ (under some conditions, almost surely the minimizing geodesic is unique between two points randomly sampled from the manifold). The smoothness of parallel transport also implies that if $p$ and $q$ are not far apart, then the initial tangent vector and the transported one stays close (in the space of tangent bundle endowed with the Sasaki metric  \citep{Sasaki1958}). This feature is desirable for our purpose, as when sample mean $\hat{\mu}(t)$ approaches to $\mu(t)$, one expects a tangent vector at $\hat{\mu}(t)$ converges to its transported version at $\mu(t)$. Owing to these nice properties of parallel transport, it becomes an ideal tool to construct a mechanism of comparing objects from different tensor Hilbert spaces as follows.

Suppose $f$ and $h$ are two measurable curves on $\manifold$ defined on $\tdomain$. 
\rev{Let $\gamma_t(\cdot):=\gamma(t,\cdot)$ be a family of smooth curves that is parameterized by the interval $[0,1]$ (the way of parameterization here does not matter) and connects $f(t)$ to $h(t)$, 
	that is, $\gamma_{t}(0)=f(t)$
	and $\gamma_{t}(1)=h(t)$, such that $\gamma(\cdot,s)$ is measurable for all $s\in[0,1]$.} Suppose $v\in\tangentspace{f(t)}$ and let $V$
be a smooth vector field along $\gamma_{t}$ such that $\nabla_{\dot{\gamma}}V=0$
and $V(0)=v$, where $\nabla$ denotes the Levi--Civita connection of the manifold $\manifold$. The theory of Riemannian manifolds shows that such a vector field $V$ uniquely exists. This gives rise to the parallel transporter $\mathcal{P}_{f(t),h(t)}:\tangentspace{f(t)}\rightarrow\tangentspace{h(t)}$ along $\gamma_t$, 
defined by $\mathcal{P}_{f(t),h(t)}(v)=V(1)$. In other
words, $\mathcal{P}_{f(t),h(t)}$ parallelly transports $v$ to $V(1)\in\tangentspace{h(t)}$ along the curve $\gamma_t$.
As the parallel transporter determined by the Levi--Civita connection, $\mathcal{P}$ preserves the
inner product of tangent vectors along transportation, that is, $\metric uv_{f(t)}=\metric{\mathcal{P}_{f(t),h(t)}u}{\mathcal{P}_{f(t),h(t)}v}_{h(t)}$ for $u,v\in \tangentspace{f(t)}$.
Then we can define the parallel transport
of vector fields from $\mathscr{T}(f)$ to $\mathscr{T}(h)$, denoted
by $\Gamma_{f,h}$, $(\Gamma_{f,h}U)(t)=\mathcal{P}_{f(t),h(t)}(U(t))$
for all $U\in\mathscr{T}(f)$ and $t\in\tdomain$. One immediately
sees that $\Gamma_{f,h}$ is a linear operator on tensor Hilbert space.
Its adjoint, denoted by $\Gamma_{f,h}^{\ast}$, is a map from $\mathscr{T}(h)$
to $\mathscr{T}(f)$ and is given by $\vfinnerprod[f]{U,\Gamma_{f,h}^{\ast}V}=\vfinnerprod[h]{\Gamma_{f,h}U,V}$
for $U\in\mathscr{T}(f)$ and $V\in\mathscr{T}(h)$. Let $\mathscr{C}(f)$
denote the set of all Hilbert--Schmidt operators on $\mathscr{T}(f)$,
which is a Hilbert space with the Hilbert--Schmidt norm $\opnorm[f]{\cdot}$.
We observe that the operator $\Gamma_{f,h}$ also gives rise to a
mapping $\Phi_{f,h}$ from $\mathscr{C}(f)$ to $\mathscr{C}(h)$,
defined by $(\Phi_{f,h}\mathcal{A})V=\Gamma_{f,h}(\mathcal{A}(\Gamma_{f,h}^{\ast}V))$
for $\mathcal{A}\in\mathscr{C}(f)$ and $V\in\mathscr{T}(h)$. The
operator $\Phi_{f,h}$ is called the parallel transporter of Hilbert--Schmidt
operators. Below are some important properties of $\Gamma_{f,h}$
and $\Phi_{f,h}$, where $(x_{1}\otimes x_{2})x_{3}\define\vfinnerprod[f]{x_{1},x_{3}}x_{2}$
for $x_{1},x_{2},x_{3}\in\mathscr{T}(f)$.
\begin{prop}
	\label{prop:property-Gamma-operator}Suppose $U\in\mathscr{T}(f)$,
	$V\in\mathscr{T}(h)$, $\mathcal{A}\in\mathscr{C}(f)$ and $\mathcal{B}\in\mathscr{C}(h)$ for two measurable curves $f$ and $h$ on $\manifold$. \rev{Let $\Gamma_{f,h}$ and $\Phi_{f,h}$ be parallel transporters along a family of smooth curves $\gamma_t$ defined above, such that $\gamma(t,\cdot)$ is smooth and $\gamma(,s)$ is measurable. Then the following statements regarding $\Gamma_{f,h}$ and $\Phi_{f,h}$ hold.}
	\begin{enumerate}
		\setlength\itemsep{1.5mm}
		\item \label{prop:property-Gamma-operator-enu-1}The operator $\Gamma_{f,h}$
		is a unitary transformation from $\mathscr{T}(f)$ to $\mathscr{T}(h)$.
		\item \label{prop:property-Gamma-operator-enu-2}$\Gamma_{f,h}^{\ast}=\Gamma_{h,f}$.
		Also, $\vfnorm[h]{\Gamma_{f,h}U-V}=\vfnorm[f]{U-\Gamma_{h,f}V}$.
		\item \label{prop:property-Gamma-operator-enu-3}$\Gamma_{f,h}(\mathcal{A}U)=(\Phi_{f,h}\mathcal{A})(\Gamma_{f,h}U)$.
		\item \label{prop:property-Gamma-operator-enu-4}If $\mathcal{A}$ is invertible,
		then $\Phi_{f,h}\mathcal{A}^{-1}=(\Phi_{f,h}\mathcal{A})^{-1}$.
		\item \label{prop:property-Gamma-operator-enu-5}$\Phi_{f,h}\sum_{k}c_{k}\varphi_{k}\otimes\varphi_{k}=\sum_{k}c_{k}(\Gamma_{f,h}\varphi_{k})\otimes(\Gamma_{f,h}\varphi_{k})$,
		where $c_{k}$ are scalar constants, and $\varphi_{k}\in\mathscr{T}(f)$.
		\item \label{prop:property-Gamma-operator-enu-6}$\opnorm[h]{\Phi_{f,h}\mathcal{A}-\mathcal{B}}=\opnorm[f]{\mathcal{A}-\Phi_{h,f}\mathcal{B}}$.
	\end{enumerate}
\end{prop}

We define $U\ominus_{\Gamma}V\define\Gamma_{f,h}U-V$ for $U\in\mathscr{T}(f)$
and $V\in\mathscr{T}(h)$, and $\mathcal{A}\ominus_{\Phi}\mathcal{B}\define\Phi_{f,h}\mathcal{A}-\mathcal{B}$
for operators $\mathcal{A}$ and $\mathcal{B}$. To quantify the discrepancy
between an element $U$ in $\mathscr{T}(f)$ and another one $V$ in $\mathscr{T}(h)$,
we can use the quantity $\vfnorm[h]{U\ominus_{\Gamma}V}$. Similarly,
we adopt $\opnorm[h]{\mathcal{A}\ominus_{\Phi}\mathcal{B}}$ as
discrepancy measure for two covariance operators $\mathcal{A}$ and
$\mathcal{B}$. These quantities are intrinsic as they are built on intrinsic geometric concepts. In light of Proposition \ref{prop:property-Gamma-operator}, they are
symmetric under the parallel transport, that is, transporting $\mathcal{A}$
to $\mathcal{B}$ yields the same discrepancy measure as transporting
$\mathcal{B}$ to $\mathcal{A}$. We also note that, when $\manifold=\real^{\dim}$,
the difference operators $\ominus_{\Gamma}$ and $\ominus_{\Phi}$
reduce to the regular vector and operator difference, that is, $U\ominus_{\Gamma}V$
becomes $U-V$, while $\mathcal{A}\ominus_{\Phi}\mathcal{B}$ becomes $\mathcal{A}-\mathcal{B}$.
Therefore, $\ominus_{\Gamma}$ and $\ominus_{\Phi}$ can be viewed
as generalization of the regular vector and operator difference to
a Riemannian setting. \rev{One shall note that $\Gamma$ and $\Phi$ depend on the choice of the family of curves $\gamma$, a canonical choice of which is discussed in Section \ref{subsec:Asymptotics-fpca}.}

\subsection{Random elements on tensor Hilbert spaces\label{subsec:Random-Elements-on-THS}}

Let $X$ be a Riemannian random process. In order to introduce the
concept of intrinsic mean function for $X$, we define a family of functions indexed by $t$: 
\begin{equation}
F(p,t)=\expect d_{\manifold}^{2}(X(t),p),\quad\quad p\in\manifold,t\in\tdomain.\label{eq:frechet-variances}
\end{equation}
For a fixed $t$, if there exists a unique $q\in\manifold$ that minimizes
$F(p,t)$ over all $p\in\manifold$, then $q$ is called the intrinsic
mean (also called Fr\'{e}chet mean) at $t$, denoted by $\mu(t)$,
that is,
\[
\mu(t)=\underset{p\in\manifold}{\arg\min}\,F(p,t).
\]
As required for intrinsic analysis, we assume the following
condition.

\begin{description}[labelwidth=1.2cm,leftmargin=1.4cm,align=left]
	\item[\mylabel{cond:exist-mean-function}{A.1}] The intrinsic mean function $\mu$ exists.
\end{description}
\rev{We refer readers to \citet{Bhattacharya2003} and \cite{Afsari2011} for conditions under which the intrinsic mean of a random variable on a general manifold exists and is
	unique}. For example, according to Cartan--Hadamard theorem, if the manifold
is simply connected and complete with nonpositive sectional curvature,
then intrinsic mean function always exists and is unique as long as  for all $t\in\tdomain$, $F(p,t)<\infty$ for some $p\in\manifold$.

Since $\manifold$ is geodesically complete, by Hopf--Rinow theorem (p. 108, \citealt{Lee1997}),
its exponential map $\Exp_{p}$ at each $p$ is defined on the entire
$\tangentspace p$. As $\Exp_{p}$ might not be injective, in order
to define its inverse, we restrict $\Exp_{p}$ to a subset
of the tangent space $\tangentspace p$. {Let $\mathrm{Cut}(p)$ denote the set of all tangent vectors $v\in\tangentspace p$ such that the geodesic $\gamma(t)=\Exp_{p}(tv)$ fails to be minimizing for $t\in[0,1+\epsilon)$ for each $\epsilon>0$.} 
Now, we define
$\Exp_{p}$ only on $\mathscr{D}_{p}\define\tangentspace p\backslash\mathrm{Cut}(p)$.
The image of $\Exp_{p}$, denoted by $\mathrm{Im}(\Exp_{p})$, consists
of points $q$ in $\manifold$, such that $q=\Exp_{p}v$ for some
$v\in\mathscr{D}_{p}$. \rev{In this case, the inverse of $\Exp_{p}$ exists and is
	called Riemannian logarithm map, which is denoted by $\Log_{p}$ and maps $q$ to $v$.} We shall make the
following assumption.

\begin{description}[labelwidth=1.2cm,leftmargin=1.4cm,align=left]
	\item[\mylabel{cond:cut-locus}{A.2}] $\prob\{\forall t\in\tdomain: X(t)\in\mathrm{Im}(\Exp_{\mu(t)})\}=1$.
\end{description}Then, $\Log_{\mu(t)}X(t)$ is almost surely defined for all $t\in\tdomain$.
The condition is superfluous if $\Exp_{\mu(t)}$ is injective for
all $t$, like the manifold of $m\times m$ SPDs endowed with the affine-invariant metric. 

In the sequel we shall assume  $X$ satisfies conditions
\ref{cond:exist-mean-function} and \ref{cond:cut-locus}. An important
observation is that the log-process $\{\Log_{\mu(t)}X(t)\}_{t\in\tdomain}$
(denoted by $\Log_{\mu}X$ for short) is a random vector field
along $\mu$. If we assume continuity for the sample paths of $X$,
then the process $\Log_{\mu}X$ is measurable with respect to the
product $\sigma$-algebra $\borel(\tdomain)\times\mathscr{E}$ and
the Borel algebra $\mathscr{B}(\tm)$, where $\mathscr{E}$ is the $\sigma$-algebra of the probability space. Furthermore, if $\expect\vfnorm[][2]{\Log_{\mu}X}<\infty$,
then according to Theorem 7.4.2 of \citet{Hsing2015}, $\Log_{\mu}X$
can be viewed as a tensor Hilbert space $\mathscr{T}(\mu)$ valued
random element. Observing that $\expect\Log_{\mu}X=0$ according to
Theorem 2.1 of \citet{Bhattacharya2003}, the intrinsic covariance
operator for $\Log_{\mu}X$ is given by $\covarop=\expect(\Log_{\mu}X\otimes\Log_{\mu}X)$.
This operator is nuclear and self-adjoint. It then admits the following
eigendecomposition (Theorem 7.2.6, \citealp{Hsing2015})
\begin{equation}
\covarop=\sum_{k=1}^{\infty}\lambda_{k}\phi_{k}\otimes\phi_{k}\label{eq:eigen-decomp-coop}
\end{equation}
with eigenvalues $\lambda_{1}\geq\lambda_{2}\cdots\ge0$ and orthonormal
eigenelements $\phi_{k}$ that form a complete orthonormal system for $\mathscr{T}(\mu)$. Also, with probability one, the log-process
of $X$ has the following Karhunen--Lo\`{e}ve expansion
\begin{equation}\label{eq:intrinsic-V}
\Log_{\mu}X=\sum_{k=1}^{\infty}\xi_{k}\phi_{k}
\end{equation}
with $\xi_{k}\define\vfinnerprod{X,\phi_{k}}$ being uncorrelated
and centered random variables. Therefore, we obtain the intrinsic
Riemannian Karhunen--Lo\`{e}ve (iRKL) expansion for $X$ given by
\begin{equation}
X(t)=\Exp_{\mu(t)}\sum_{k=1}^{\infty}\xi_{k}\phi_{k}(t).\label{eq:intrinsic-KL-expansion}
\end{equation}
The elements $\phi_{k}$ are called intrinsic Riemannian functional
principal component (iRFPC), while the variables $\xi_{k}$ are called
intrinsic iRFPC scores. This result is summarized in the following
theorem whose proof is already contained in the above derivation and
hence omitted. \rev{We shall note that the continuity assumption on sample paths can be weakened to piece-wise continuity.}
\begin{thm}
	[Intrinsic Karhunen--Lo\`{e}ve Representation]Assume that $X$ satisfies
	assumptions \ref{cond:exist-mean-function} and \ref{cond:cut-locus}.
	If sample paths of $X$ are continuous and $\expect\vfnorm[][2]{\Log_{\mu}X}<\infty$,
	then the intrinsic covariance operator $\covarop=\expect(\Log_{\mu} X\otimes \Log_{\mu} X)$
	of $\Log_{\mu}X$ admits the decomposition (\ref{eq:eigen-decomp-coop}), and
	the random process $X$ admits the representation (\ref{eq:intrinsic-KL-expansion}).
\end{thm}


In practice, the series at (\ref{eq:intrinsic-KL-expansion}) is truncated
at some positive integer $K$, resulting in a truncated intrinsic
Riemannian Karhunen--Lo\`{e}ve expansion of $X$, given by $X_{K}=\Exp_{\mu}W_{K}$
with $W_{K}=\sum_{k=1}^{K}\xi_{k}\phi_{k}$. The quality of the approximation
of $X_{K}$ for $X$ is quantified by $\int_{\tdomain}d_{\manifold}^{2}(X(t),X_{K}(t))\diffop\upsilon(t)$,
and can be shown by a method similar to  \citet{Dai2017} that if the manifold has non-negative
sectional curvature everywhere, then $\int_{\tdomain}d_{\manifold}^{2}(X(t),X_{K}(t))\diffop\upsilon(t)\leq\vfnorm[][2]{\Log_{\mu}X-W_{K}}$.
For manifolds with negative sectional curvatures, such inequality
in general does not hold. However, for Riemannian random process $X$
that almost surely lies in a compact subset of $\manifold$, the residual
$\int_{\tdomain}d_{\manifold}^{2}(X(t),X_{K}(t))\diffop\upsilon(t)$
can be still bounded by $\vfnorm[][2]{\Log_{\mu}X-W_{K}}$ up to a
scaling constant.
\begin{prop}
	\label{prop:truncation-residual}Assume that conditions \ref{cond:exist-mean-function}
	and \ref{cond:cut-locus} hold, and the sectional curvature
	of $\manifold$ is bounded from below by $\kappa\in\real$. Let $\mathcal K$ be a subset of $\manifold$. If  $\kappa\geq0$,  we let $\mathcal{K}=\manifold$, and if $\kappa<0$, we  assume that $\mathcal{K}$ is compact. Then, 
	for some constant $C>0$, $d_{\manifold}(P,Q)\leq\sqrt{C}|\Log_{O}P-\Log_{O}Q|$
	for all $O,P,Q\in\mathcal{K}$. Consequently, if $X\in\mathcal{K}$
	almost surely, then $\int_{\tdomain}d_{\manifold}^{2}(X(t),X_{K}(t))\diffop\upsilon(t)\leq C\vfnorm[][2]{\Log_{\mu}X-W_{K}}$.
\end{prop}

\subsection{Computation in orthonormal frames\label{subsec:Computation-in-Orthonormal}}

In practical computation, one might want to work with specific orthonormal
bases for tangent spaces. A choice of orthonormal basis for each
tangent space constitutes an orthonormal frame on the manifold. In
this section, we study the representation of the intrinsic Riemannian
Karhunen--Lo\`{e}ve expansion under an orthonormal frame and  formulas for change of orthonormal frames.

Let $\vec E=(E_{1},\ldots,E_{\dim})$ be a continuous orthonormal
frame, that is, each $E_{j}$ is a vector field of $\manifold$
such that $\metric{E_{j}(p)}{E_{j}(p)}_{p}=1$ and $\metric{E_{j}(p)}{E_{k}(p)}_{p}=0$
for $j\neq k$ and all $p\in\manifold$. At each point $p$, $\{E_{1}(p),\ldots,E_{\dim}(p)\}$
form an orthonormal basis for $\tangentspace p$. The coordinate of
$\Log_{\mu(t)}X(t)$ with respect to $\{E_{1}(\mu(t)),\ldots,E_{\dim}(\mu(t))\}$
is denoted by $Z_{\vec E}(t)$, with the subscript $\vec E$ indicating
its dependence on the frame. The resulting process $Z_{\vec E}$ is
called the \emph{$\vec{E}$-coordinate process} of $X$. Note that $Z_{\vec E}$
is a regular $\real^{\dim}$ valued random process defined on $\tdomain$,
and classic theory in \citet{Hsing2015} applies to $Z_{\vec E}$.
For example, its $\ltwo$ norm is defined by $\lpnorm[2]{Z_{\vec E}}=\{\expect\int_{\tdomain}|Z_{\vec E}(t)|^{2}\diffop t\}^{1/2}$,
where $|\cdot|$ denotes the canonical norm on $\real^{\dim}$. One
can show that $\lpnorm[2][2]{Z_{\vec E}}=\expect\vfnorm[][2]{\Log_{\mu}X}$.
Therefore, if $\expect\vfnorm[][2]{\Log_{\mu}X}<\infty$, then the
covariance function exists and is $d\times d$ matrix-valued, quantified by $\covarop_{\vec E}(s,t)=\expect\{Z_{\vec E}(s)Z_{\vec E}(t)^{T}\}$
\citep{Kelly1960,Balakrishn1960}, noting that $\expect Z_{\vec E}(t)=0$ as $\expect\Log_{\mu(t)}X(t)=0$ for all $t\in\tdomain$.
Also, the vector-valued Mercer's theorem implies the eigendecomposition
\begin{equation}
\covarop_{\vec E}(s,t)=\sum_{k=1}^{\infty}\lambda_{k}\phi_{\vec E,k}(s)\phi_{\vec E,k}^{T}(t),\label{eq:coordinate-eigen-decomp}
\end{equation}
with eigenvalues $\lambda_{1}\geq\lambda_{2}\geq\cdots$ and corresponding
eigenfunctions $\phi_{\vec E,k}$.  Here, the subscript $\vec E$ in $\phi_{\vec E,k}$ is to emphasize
the dependence on the chosen frame. One can see that $\phi_{\vec E,k}$
is a coordinate representation of $\phi_{k}$, that is, $\phi_{k}=\phi_{\vec E,k}^{T}\vec E$.

The coordinate process $Z_{\vec E}$ admits the vector-valued Karhunen--Lo\`{e}ve
expansion
\begin{equation}
Z_{\vec E}(t)=\sum_{k=1}^{\infty}\xi_{k}\phi_{\vec E,k}(t)\label{eq:coordinate-KL-representation}
\end{equation}
under the assumption of mean square continuity of $Z_{\vec E}$, according
to Theorem 7.3.5 of \citet{Hsing2015}, where $\xi_{k}=\int_{\tdomain}Z_{\vec E}^{T}(t)\phi_{\vec E,k}(t)\diffop\upsilon(t)$.
While the covariance function and eigenfunctions of $Z_{\vec E}$
depend on frames, $\lambda_{k}$ and $\xi_{k}$ in (\ref{eq:intrinsic-KL-expansion})
and (\ref{eq:coordinate-KL-representation}) are not,
which justifies
the absence of $\vec E$ in their subscripts and the use of the same notation
for eigenvalues and iRFPC scores in (\ref{eq:eigen-decomp-coop}),
(\ref{eq:intrinsic-KL-expansion}), (\ref{eq:coordinate-eigen-decomp})
and (\ref{eq:coordinate-KL-representation}). This follows from the
formulas for change of frames that we shall develop below.

Suppose $\vec A=(A_{1},\ldots,A_{\dim})$ is another orthonormal frame.
Change from $\vec E(p)=\{E_{1}(p),\ldots,E_{\dim}(p)\}$ to $\vec A(p)=\{A_{1}(p),\ldots,A_{\dim}(p)\}$
can be characterized by a unitary matrix $\vec O_{p}$. For example,
$\vec A(t)=\vec O_{\mu(t)}^T\vec E(t)$ and hence $Z_{\vec A}(t)=\vec O_{\mu(t)}Z_{\vec E}(t)$
for all $t$. Then the covariance function of $Z_{\vec A}$ is given
by
\begin{alignat}{1}
\covarop_{\vec A}(s,t) & =\expect\{Z_{\vec A}(s)Z_{\vec A}^T(t)\}=\expect\{\vec O_{\mu(s)}Z_{\vec E}(s)Z_{\vec E}^T(t)\vec O_{\mu(t)}^{T}\}=\vec O_{\mu(s)}\covarop_{\vec E}(s,t)\vec O_{\mu(t)}^{T},\label{eq:relation-covariance-functions}
\end{alignat}
and consequently,
\[
\covarop_{\vec A}(s,t)=\sum_{k=1}^{\infty}\lambda_{k}\{\vec O_{\mu(s)}\phi_{\vec E,k}(s)\}\{\vec O_{\mu(t)}\phi_{\vec E,k}(t)\}^{T}.
\]
From the above calculation, we immediately see that $\lambda_{k}$
are also eigenvalues of $\covarop_{\vec A}$. Moreover, the eigenfunction
associated with $\lambda_{k}$ for $\covarop_{\vec A}$ is given by
\begin{equation}
\phi_{\vec A,k}(t)=\vec O_{\mu(t)}\phi_{\vec E,k}(t).\label{eq:relation-eigenfunctions}
\end{equation}
Also, the variable $\xi_{k}$ in (\ref{eq:intrinsic-KL-expansion})
and (\ref{eq:coordinate-KL-representation}) is the functional
principal component score for $Z_{\vec A}$ associated with $\phi_{\vec A,k}$,
as seen by $\int_{\tdomain}Z_{\vec A}^{T}(t)\phi_{\vec A,k}(t)\diffop\upsilon(t)=\int_{\tdomain}Z_{\vec E}^{T}(t)\vec O_{\mu(t)}^{T}\vec O_{\mu(t)}\phi_{\vec E,k}(t)\diffop\upsilon(t)=\int_{\tdomain}Z_{\vec E}^{T}(t)\phi_{\vec E,k}(t)\diffop\upsilon(t)$.
The following proposition summarizes the above results.
\begin{prop}[Invariance Principle]
	\label{prop:fpca-frame}Let $X$ be a $\manifold$-valued random process
	satisfying conditions \ref{cond:exist-mean-function} and \ref{cond:cut-locus}.
	Suppose $\vec{E}$ and $\vec{A}$ are measurable orthonormal frames that are continuous on
	a neighborhood of the image of $\mu$, and $Z_{\vec E}$ denotes the $\vec E$-coordinate
	log-process of $X$. Assume $\vec O_{p}$ is the unitary matrix continuously varying with $p$ such that $\vec A(p)=\vec O_{p}^T\vec E(p)$
	for $p\in\manifold$.
	\begin{enumerate}
		\setlength\itemsep{1.5mm}
		\item \label{enu:prop:fpca-frame-1}The $\lp[r]$-norm of $Z_{\vec E}$ for $r>0$,
		defined by $\lpnorm[r]{Z_{\vec E}}=\{\expect\int_{\tdomain}|Z_{\vec E}(t)|^{r}\diffop\upsilon(t)\}^{1/r}$,
		is independent of the choice of frames. In particular, $\|Z_{\vec E}\|_{\lp[2]}^{2}=\expect\vfnorm[][2]{\Log_{\mu}X}$
		for all orthonormal frames $\vec E$.
		\item \label{enu:prop:fpca-frame-2}If $\expect\vfnorm[][2]{\Log_{\mu}X}<\infty$,
		then the covariance function of $Z_{\vec E}$ exists for all $\vec E$
		and admits decomposition of (\ref{eq:coordinate-eigen-decomp}). Also,
		(\ref{eq:eigen-decomp-coop}) and (\ref{eq:coordinate-eigen-decomp})
		are related by $\phi_{k}(t)=\phi_{\vec E,k}^{T}(t)\vec E(\mu(t))$
		for all $t$, and the eigenvalues $\lambda_{k}$ coincide. Furthermore,
		the eigenvalues of $\covarop_{\vec E}$ and the principal component
		scores of Karhunen--Lo\`{e}ve expansion of $Z_{\vec E}$ do not depend
		on $\vec E$. 
		\item \label{enu:prop:fpca-frame-3}The covariance functions $\covarop_{\vec A}$
		and $\covarop_{\vec E}$ of respectively $Z_{\vec A}$ and $Z_{\vec E}$,
		if they exist, are related by (\ref{eq:relation-covariance-functions}).
		Furthermore, their eigendecomposions are related by (\ref{eq:relation-eigenfunctions})
		and $Z_{\vec A}(t)=\vec O_{\mu(t)}Z_{\vec E}(t)$ for all $t\in\tdomain$.
		\item \label{enu:prop:fpca-frame-4}If $\expect\vfnorm[][2]{\Log_{\mu}X}<\infty$ and
		sample paths of $X$ are continuous, then the scores $\xi_{k}$
		(\ref{eq:coordinate-KL-representation}) coincide with the iRFPC scores in
		(\ref{eq:intrinsic-KL-expansion}).
	\end{enumerate}
\end{prop}
We conclude this subsection by emphasizing that the concept of covariance function of the log-process depends on the frame $\vec{E}$, while  the covariance operator, eigenvalues, eigenelements and iRFPC scores do not. In particular, the scores $\xi_k$, which are often the input for  further statistical analysis such as regression and classification,  are invariant to the choice of coordinate frames. An important consequence of the invariance principle is that, these scores can be safely computed in any convenient coordinate frame without altering the subsequent analysis.

\subsection{Connection to the special case of Euclidean submanifolds}\label{subsec:euclidean-submanifold}
\lincomment{Our framework applies to general manifolds that include Euclidean submanifolds as  special examples to which the methodology of  \cite{Dai2017} also applies. When the underlying manifold is a $d$-dimensional submanifold of the Euclidean space $\real^{d_0}$ with $d<d_0$, we recall that the tangent space at each point is identified as a $d$-dimensional linear subspace of $\real^{d_0}$.} For such Euclidean manifolds, \cite{Dai2017} treat the log-process of $X$ as a $\real^{d_0}$-valued random process, and derive the representation for the log-process (equation (5) in their paper) within the 
ambient Euclidean space. This is distinctly different from our
intrinsic representation \eqref{eq:intrinsic-V} based on the theory of tensor Hilbert spaces, despite their similar appearance. For instance, equation. (5) in their work can only be defined for Euclidean  submanifolds, while ours is applicable to general Riemannian manifolds. Similarly, the covariance function 
defined in \citet{Dai2017}, denoted by $C^{{DM}}(s,t)$, is associated with the  ambient log-process $V(t)\in \real^{d_0}$, that is, $C^{{DM}}(s,t)=\expect V(s)^TV(t)$. 
Such an ambient covariance function can only be defined for Euclidean submanifolds but not general manifolds. 

\lincomment{Nevertheless, there are connections between the ambient method of \cite{Dai2017} and our framework when  $\manifold$ is a Euclidean submanifold.  For instance, the mean curve is intrinsically  defined in the same way in both works. For the covariance structure, although  
	our covariance function  $\covarop_{\vec E}$ is a $d\times d$ matrix-valued function while $C^{{DM}}(s,t)$ is a $d_0\times d_0$ matrix-valued function, they both represent the intrinsic covariance operator when $\manifold$ is a Euclidean submanifold. To see so, first, we observe that the ambient log-process $V(t)$ as defined in \cite{Dai2017} at the time $t$, although is ambiently $d_0$-dimensional, lives in a $d$-dimensional linear subspace of $\real^{d_0}$. Second, the orthonormal basis $\vec E(t)$ for the tangent space at $\mu(t)$ can be realized by a $d_0\times d$ full-rank matrix $\vec G_t$ by concatenating vectors $E_1(\mu(t)),\ldots,E_d(\mu(t))$. Then $U(t)=\vec G_t^TV(t)$ is the $\vec E$-coordinate process of $X$. This implies that $\covarop_{\vec E}(s,t)=\vec G_s^TC^{{DM}}(s,t)\vec G_t$. On the other hand, since $V(t)=\vec G_tU(t)$, one has $C^{{DM}}(s,t)=\vec G_t\covarop_{\vec E}(s,t)\vec G_t^T$. Thus, $\covarop_{\vec E}$ and $C^{{DM}}$ determine each other and represent the same object. In light of this observation and the invariance principle  stated in Proposition \ref{prop:fpca-frame}, when $\manifold$ is a Euclidean submanifold, $C^{{DM}}$ can be viewed as the ambient representation of the intrinsic covariance operator $\covarop$, while $\covarop_{\vec E}$ is the coordinate representation of $\covarop$ with respect to the frame $\vec E$.  Similarly, the eigenfunctions $\phi_k^{DM}$ of $C^{DM}$ are the ambient representation of the eigenelements $\phi_k$ of $\covarop$. 
	The above reasoning also applies to sample mean functions and sample covariance structure. Specifically, when $\manifold$ is a Euclidean submanifold, our estimator for the mean function discussed in Section \ref{sec:Intrinsic-Representation} is  identical to the one in \cite{Dai2017}, while the estimators for the covariance function and eigenfunctions proposed in \cite{Dai2017} are the ambient representation of our estimators stated in Section \ref{sec:Intrinsic-Representation}.}

\lincomment{ 
	However, when quantifying the discrepancy between the population covariance structure  and its estimator, \citet{Dai2017} adopt the Euclidean difference as a measure. For instance, they use $\hat{\phi}^{DM}_k-\phi^{DM}_k$ to represent the discrepancy between the sample eigenfunctions and the population eigenfunctions, where $\hat{\phi}^{DM}_k$ is the sample version of $\phi_k^{DM}$. When $\hat{\mu}(t)$, the sample version of $\mu(t)$, is not equal to $\mu(t)$, $\hat{\phi}^{DM}_k(t)$ and ${\phi}^{DM}_k(t)$ belong to different tangent spaces. In such case, the Euclidean difference $\hat{\phi}^{DM}_k-\phi^{DM}_k$ is a Euclidean vector that does not belong to the tangent space at either $\hat{\mu}(t)$ or $\mu(t)$, as illustrated in the left panel of Figure \ref{fig:parallel-transport}. In other words, the Euclidean difference of ambient eigenfunctions does not obey the geometry of the manifold, hence might not properly measure the intrinsic discrepancy. In particular, the measure $\|\hat{\phi}^{DM}_k-\phi^{DM}_k\|_{\real^{d_0}}$ might  completely result from the departure of the ambient Euclidean geometry from the manifold, rather than the intrinsic discrepancy between the sample and population eigenfunctions, as demonstrated in the left panel of Figure \ref{fig:parallel-transport}. Similar reasoning applies to $\hat{C}^{DM}-C^{DM}$. In contrast,  we base on Proposition \ref{prop:property-Gamma-operator} to propose an intrinsic measure to characterize the intrinsic discrepancy between a population quantity and its estimator in Section \ref{subsec:Asymptotics-fpca}.}


\section{Intrinsic Riemannian functional principal component analysis\label{sec:Intrinsic-Representation}}

\subsection{Model and estimation\label{subsec:Estimation-fpca}}

Suppose $X$ admits the intrinsic Riemannian Karhunen--Lo\`{e}ve expansion
(\ref{eq:intrinsic-KL-expansion}), and $X_{1},\ldots,X_{n}$ are
a random sample of $X$. In the sequel, we assume that trajectories
$X_{i}$ are fully observed.  \lincomment{In the case that data are densely observed,
	each trajectory can be individually  interpolated 
	by using regression techniques for manifold valued data, such as \citet{Steinke2010},
	\citet{Cornea2017} and \citet{Petersen2017}. This way the densely observed data could be represented by their interpolated surrogates, and thus treated as if they were fully observed curves.} When data are sparse,
delicate information pooling of observations across different subjects
is required. The development of such methods is substantial and beyond
the scope of this paper.

In order to estimate the mean function $\mu$, we define the finite-sample version of $F$ in \eqref{eq:frechet-variances} by 
\[
F_{n}(p,t)=\frac{1}{n}\sum_{i=1}^{n}d_{\manifold}^{2}(X_{i}(t),p).
\]
Then, an estimator for $\mu$ is given by 
\[
\hat{\mu}(t)=\underset{p\in\manifold}{\arg\min}F_{n}(p,t).
\]
The computation of $\hat{\mu}$ depends on the Riemannian structure
of the manifold. Readers are referred to \citet{Cheng2016} and \citet{Salehian2015}
for practical algorithms. For a subset $A$ of $\manifold$,
$A^{\epsilon}$ denotes the set $\bigcup_{p\in A}B(p;\epsilon)$,
where $B(p;\epsilon)$ is the ball with center $p$ and radius $\epsilon$
in $\manifold$. We use $\mathrm{Im}^{-\epsilon}(\Exp_{\mu(t)})$
to denote the set $\manifold\backslash\{\manifold\backslash\mathrm{Im}(\Exp_{\mu(t)})\}^{\epsilon}$.
In order to define $\Log_{\hat{\mu}}X_{i}$, at least with a dominant
probability for a large sample, we shall assume a slightly stronger
condition than \ref{cond:cut-locus}:

\begin{description}[labelwidth=1.2cm,leftmargin=1.4cm,align=left]
	\item[\mylabel{cond:cut-locus-prime}{A.2$'$}] There is some constant $\epsilon_0>0$ such that $\prob\{\forall t\in\tdomain: X(t)\in\mathrm{Im}^{-\epsilon_0}(\Exp_{\mu(t)})\}=1$.
\end{description}Then, combining the fact  $\sup_{t}|\hat{\mu}(t)-\mu(t)|=o_{a.s.}(1)$ that we will show later,
we conclude that for a large sample, almost surely, $\mathrm{Im}^{-\epsilon}(\Exp_{\mu(t)})\subset\mathrm{Im}(\Exp_{\hat{\mu}(t)})$
for all $t\in\tdomain$. Therefore, under this condition, $\Log_{\hat{\mu}(t)}X_{i}(t)$
is well-defined almost surely for a large sample. 

The intrinsic Riemannian covariance operator is estimated by its finite-sample
version
\[
\hat{\covarop}=\frac{1}{n}\sum_{i=1}^{n}(\Log_{\hat{\mu}}X_{i})\otimes(\Log_{\hat{\mu}}X_{i}).
\]
This sample intrinsic Riemannian covariance operator also admits an
intrinsic eigendecomposion $\hat{\covarop}=\sum_{k=1}^{\infty}\hat{\lambda}_{k}\hat{\phi}_{k}\otimes\hat{\phi}_{k}$
for $\hat{\lambda}_{1}\geq\hat{\lambda}_{2}\geq\cdots\geq0$. Therefore,
the estimates for the eigenvalues $\lambda_{k}$ are given by $\hat{\lambda}_{k}$,
while the estimates for $\phi_{k}$ are given by $\hat{\phi}_{k}$.
These estimates can also be conveniently obtained under a frame, due to the invariance principle  stated in Proposition \ref{prop:fpca-frame}. Let $\vec E$
be a chosen orthonormal frame and $\hat{\covarop}_{\vec E}$ be the
sample covariance function based on $\hat{Z}_{\vec E,1},\ldots,\hat{Z}_{\vec E,n}$,
where $\hat{Z}_{\vec E,i}$ is the coordinate process of  $\Log_{\hat{\mu}(t)}X_{i}(t)$
under the frame $\vec E$ with respect to $\hat{\mu}$. We can then
obtain the eigendecomposition $\hat{\covarop}_{\vec E}(s,t)=\sum_{k=1}^{\infty}\hat{\lambda}_{k}\hat{\phi}_{\vec E,k}(s)\hat{\phi}_{\vec E,k}(t)^{T}$, which yields $\hat{\phi}_{k}(t)=\hat{\phi}_{\vec E,k}^{T}(t)\vec E(t)$ for
$t\in\tdomain$. Finally, the truncated
process for $X_{i}$ is estimated by
\begin{equation}
\hat{X}_{i}^{(K)}=\Exp_{\hat{\mu}}\sum_{k=1}^{K}\hat{\xi}_{ik}\hat{\phi}_{k},\label{eq:intrinsic-truncated-KL}
\end{equation}
where $\hat{\xi}_{ik}=\vfinnerprod[\hat{\mu}]{\Log_{\hat{\mu}}X_{i},\hat{\phi}_{k}}$
are estimated iRFPC scores. \rev{The above truncated iRKL expansion can be regarded as generalization of the representation (10) in \cite{Dai2017} from Euclidean submanifolds to general Riemannian manifolds.}


\subsection{Asymptotic properties\label{subsec:Asymptotics-fpca}}

To quantify the difference between $\hat{\mu}$ and $\mu$, it is natural to use the square geodesic distance
$d_{\manifold}(\hat{\mu}(t),\mu(t))$ as a measure of discrepancy. For
the asymptotic properties of $\hat{\mu}$, we need the following regularity
conditions.

\begin{description}[labelwidth=1.2cm,leftmargin=1.4cm,align=left]
	\setlength\itemsep{1.5mm}
	\item[\mylabel{cond:manifold-property}{B.1}] 
	
	The manifold $\manifold$ is connected and  complete. In addition, the exponential map $\Exp_p:\tangentspace p\rightarrow\manifold$ is surjective at every point $p\in\manifold$. %
	
	\item[\mylabel{cond:sample-path-continuity}{B.2}] 
	
	The sample paths of $X$ are continuous.
	
	\item[\mylabel{cond:uniform-second-moment}{B.3}]
	
	$F$ is finite. Also, for all compact subsets $\mathcal{K}\subset\manifold$,
	$\sup_{t\in\tdomain}\sup_{p\in\mathcal{K}}\expect d_{\manifold}^2(p,X(t))<\infty$. 
	
	\item[\mylabel{cond:boundedness-mean}{B.4}] 
	
	The image $\mathcal{U}$ of the mean function $\mu$ is bounded, that is, the diameter is finite, 
	$\mathrm{diam}(\mathcal{U})<\infty$.
	
	\item[\mylabel{cond:separability}{B.5}] 
	
	For all $\epsilon>0$, $\inf_{t\in\tdomain}\inf_{p:d_{\manifold}(p,\mu(t))\geq\epsilon}F(p,t)-F(\mu(t),t)>0$.
\end{description}
To state the next condition, let $V_{t}(p)=\Log_{p}X(t)$. The calculus of manifolds suggests that
$V_{t}(p)=-d_{\manifold}(p,X(t))\,\mathrm{grad}_{p}d_{\manifold}(p,X(t))=\mathrm{grad}_{p}(-d_{\manifold}^{2}(p,X(t))/2)$, where $\mathrm{grad}_p$ denotes the gradient operator at $p$.
For each $t\in\tdomain$, let $H_{t}$  denote the Hessian of the real function  $d_{\manifold}^{2}(\,\cdot\,,X(t))/2$,
that is, for vector fields $U$ and $W$ on $\manifold$,
\[
\metric{H_{t}U}W(p)=\metric{-\nabla_{U}V_{t}}W(p)=\mathrm{Hess}_{p}\left(\frac{1}{2}d_{\manifold}^{2}(p,X(t))\right)(U,W).
\]
\begin{description}[labelwidth=1.2cm,leftmargin=1.4cm,align=left]
	\setlength\itemsep{1.5mm}
	\item[\mylabel{cond:isolation-minima}{B.6}] 
	
	$\inf_{t\in\tdomain}\{\lambda_{\min}(\expect H_{t})\}>0$, where $\lambda_{\min}(\cdot)$ denotes the smallest eigenvalue of an operator or matrix.
	
	\item[\mylabel{cond:Lipschitz-moments}{B.7}] 
	
	$\expect L(X)^{2}<\infty$ and $L(\mu)<\infty$, where $L(f)\define\sup_{s\neq t}d_{\manifold}(f(s),f(t))/|s-t|$
	for a real function $f$ on $\manifold$.
	
\end{description}
The assumption \ref{cond:manifold-property} regarding the property
of manifolds is met in general, for example, the $\dim$-dimensional unit
sphere $\mathbb{S}^{\dim}$, SPD manifolds, etc. By the Hopf--Rinow theorem, the condition also implies that $\manifold$ is geodesically complete. Conditions similar to  \ref{cond:sample-path-continuity},
\ref{cond:separability}, \ref{cond:isolation-minima} and \ref{cond:Lipschitz-moments} are made
in \citet{Dai2017}. The condition \ref{cond:boundedness-mean}
is a weak requirement for the mean function and is automatically satisfied
if the manifold is compact, while \ref{cond:uniform-second-moment}
is analogous to standard moment conditions in the literature of Euclidean
functional data analysis and becomes superfluous when $\manifold$
is compact. If $\manifold=\real$,
then \ref{cond:uniform-second-moment} is equivalent to $\sup_{t\in\tdomain}\expect|X(t)|^{2}<\infty$, a condition commonly made in the literature
of functional data analysis, for example, \citet{Hall2006}. 
The following theorem summarizes asymptotic properties of $\hat{\mu}$ for \emph{general} Riemannian manifolds. \rev{Its proof can be obtained by mimicking the techniques in \cite{Dai2017}, with additional care of the case that $\manifold$ is noncompact. For completeness, we provide its proof in the Appendix. As noted by \cite{Dai2017}, the root-$n$ rate can not be improved in general.}
\begin{thm}
	\label{thm:property-mean-function} Under  conditions
	\ref{cond:exist-mean-function}, \ref{cond:cut-locus-prime} {and \ref{cond:manifold-property}}-\ref{cond:boundedness-mean}, the following holds. 
	\begin{enumerate}
		\setlength\itemsep{1.5mm}
		\item \label{enu:thm:mean-curve-1} $\mu$ is uniformly continuous, and $\hat{\mu}$ is uniformly continuous
		with probability one.
		\item \label{enu:thm:mean-curve-2}$\hat{\mu}$ is a uniformly strong consistent estimator for $\mu$,
		that is, $\sup_{t}d_{\manifold}(\hat{\mu}(t),\mu(t))=\oas(1)$.
		\item \label{enu:thm:mean-curve-3}If additional conditions  \ref{cond:separability}-\ref{cond:Lipschitz-moments}
		are assumed, then $\sqrt{n}\,\Log_{\mu}\hat{\mu}$ converges in distribution to a Gaussian measure on the tensor Hilbert space $\mathscr{T}(\mu)$.
		\item \label{enu:thm:mean-curve-4}If additional conditions \ref{cond:separability}-\ref{cond:Lipschitz-moments}
		are assumed, then  $\sup_{t\in\tdomain}d_{\manifold}^2(\hat{\mu}(t),\mu(t))=\Op(n^{-1})$ and $\int_{\tdomain}d_{\manifold}^2(\hat{\mu}(t),\mu(t))\diffop \upsilon(t)=\Op(n^{-1})$.
	\end{enumerate}
\end{thm}

For asymptotic analysis of the estimated eigenstructure, as discussed in Section \ref{subsec:Tensor-Hilbert-Spaces-along-cruves} and \ref{subsec:euclidean-submanifold},
$\hat{\covarop}-\covarop$ and $\hat{\phi}_{k}-\phi_{k}$ are not
defined for intrinsic Riemannian functional data analysis, since they are objects
originated from different tensor Hilbert spaces $\mathscr{T}(\mu)$ and $\mathscr{T}(\hat{\mu})$
induced by \emph{different curves} $\mu$ and $\hat{\mu}$. Therefore,
we shall adopt the intrinsic measure of discrepancy developed in Section
\ref{subsec:Tensor-Hilbert-Spaces-along-cruves}, namely, $\hat{\covarop}\ominus_{\Phi}\covarop$
and $\hat{\phi}_k\ominus_{\Gamma}\phi_k$, where the dependence of $\Phi$ and $\Gamma$ on $\hat{\mu}$ and $\mu$ is suppressed for simplicity. \rev{As mentioned at the end of Section \ref{subsec:Tensor-Hilbert-Spaces-along-cruves}, both $\Gamma$ and $\Phi$ depend on the choice of the family of curves $\gamma$. Here, a canonical choice for each $\gamma(t,\cdot)$ is the minimizing geodesic between $\mu(t)$ and $\hat{\mu}(t)$. The existence and uniqueness of such geodesics can be deduced from  assumptions \ref{cond:exist-mean-function} and \ref{cond:cut-locus-prime}. Also, the continuity of $\mu(t)$ and $\hat{\mu}(t)$ implies the continuity of $\gamma(\cdot,\cdot)$ and hence the measurability of $\gamma(\cdot,s)$ for each $s\in[0,1]$.} By Proposition \ref{prop:property-Gamma-operator},
one sees that $\Phi\hat{\covarop}=n^{-1}\sum_{i=1}^n(\Gamma\hat{V}_{i}\otimes\Gamma\hat{V}_{i})$,
recalling that $\hat{V}_{i}=\Log_{\hat{\mu}}X_{i}$ is a vector field along $\hat{\mu}$. It can also be seen that $(\hat{\lambda}_{k},\Gamma\hat{\phi}_{k})$ are eigenpairs
of $\Phi\hat{\covarop}$. These identities match our intuition that
the transported sample covariance operator ought to be an operator
derived from transported sample vector fields, and that the eigenfunctions
of the transported operator are identical to the transported
eigenfunctions. 

To state the asymptotic properties for the  eigenstructure, we define
\[
\eta_{k}=\min_{1\leq j\leq k}(\lambda_{j}-\lambda_{j+1}),\qquad J=\inf\{j\geq1:\lambda_{j}-\lambda_{j+1}\leq2\opnorm[\mu]{\hat{\covarop}\ominus_{\Phi}\covarop}\},
\]
\[
\hat{\eta}_{j}=\min_{1\leq j\leq k}(\hat{\lambda}_{j}-\hat{\lambda}_{j+1}),\qquad\hat{J}=\inf\{j\geq1:\hat{\lambda}_{j}-\hat{\lambda}_{j+1}\leq2\opnorm[\mu]{\hat{\covarop}\ominus_{\Phi}\covarop}\}.
\]
\begin{thm}
	\label{thm:explict-bound-eigenfunction}Assume that every eigenvalue
	$\lambda_{k}$ has multiplicity one, and {conditions
		\ref{cond:exist-mean-function}, }\ref{cond:cut-locus-prime}{{}
		and \ref{cond:manifold-property}}-\ref{cond:Lipschitz-moments} hold. Suppose tangent vectors are parallel transported along minimizing geodesics for defining the parallel transporters $\Gamma$ and $\Phi$. If
	$\expect\vfnorm[][4]{\Log_{\mu}X}<\infty$,
	then $\opnorm[\mu][2]{\hat{\covarop}\ominus_{\Phi}\covarop}=\Op(n^{-1})$.
	Furthermore, $\sup_{k\geq1}|\hat{\lambda}_{k}-\lambda_{k}|\leq\opnorm[\mu]{\hat{\covarop}\ominus_{\Phi}\covarop}$
	and for all $1\leq k\leq J-1$,
	\begin{equation}
	\vfnorm[][2]{\hat{\phi}_{k}\ominus_{\Gamma}\phi_{k}}\leq8\opnorm[\mu][2]{\hat{\covarop}\ominus_{\Phi}\covarop}/\eta_{k}^{2}.\label{eq:bound-eigenfunction-1}
	\end{equation}
	If $(J,\eta_{j})$ is replaced by $(\hat{J},\hat{\eta}_{j})$, then
	(\ref{eq:bound-eigenfunction-1}) holds with probability 1.
\end{thm}
In this theorem, \eqref{eq:bound-eigenfunction-1}  generalizes Lemma 4.3 of \citet{Bosq2000} to the Riemannian setting. Note that the intrinsic rate for $\hat{\covarop}$ is optimal. Also, from \eqref{eq:bound-eigenfunction-1} one can deduce the optimal rate $\vfnorm[][2]{\hat{\phi}_{k}\ominus_{\Gamma}\phi_{k}}=\Op(n^{-1})$ for a fixed $k$. We stress that these results apply to not only Euclidean submanifolds, but also \emph{general} Riemannian manifolds.

\section{Intrinsic Riemannian functional linear regression\label{sec:Intrinsic-Regression}}

\subsection{Regression model and estimation}

Classical functional linear regression for Euclidean functional
data is well studied in the literature, that is, the model relating a scalar response $Y$ and a functional
predictor $X$ by $Y=\alpha+\int_{\tdomain}X(t)\beta(t)\diffop\upsilon(t)+\varepsilon$,
where $\alpha$ is the intercept, $\beta$ is the slope function and
$\varepsilon$ represents measurement errors, for example, \citet{Cardot2003},
\citet{Cardot2007}, \citet{Hall2007c} and \citet{Yuan2010}, among others. However,
for Riemannian functional data, both $X(t)$ and $\beta(t)$ take
values in a manifold and hence the product $X(t)\beta(t)$ is not
well defined. Rewriting the model as $Y=\alpha+\vfinnerprod[\ltwo]{X,\beta}+\varepsilon$,
where $\vfinnerprod[\ltwo]{\cdot,\cdot}$ is the canonical inner product
of the $\ltwo$ square integrable functions, we propose to replace  $\vfinnerprod[\ltwo]{\cdot,\cdot}$
by the inner product on the tensor Hilbert space $\mathscr{T}(\mu)$,
and define the following Riemannian functional linear regression model: 
\begin{equation}
Y=\alpha+\vfinnerprod[\mu]{\Log_{\mu}X,\Log_{\mu}\beta}+\varepsilon,\label{eq:mFLR}
\end{equation}
where we require conditions \ref{cond:exist-mean-function} and \ref{cond:cut-locus}.
Note that $\beta$ is a manifold valued function defined on $\tdomain$, namely the \emph{Riemannian slope function} of the model \eqref{eq:mFLR}, and this model is linear in terms of $\Log_{\mu(t)}\beta(t)$.
We stress that the model (\ref{eq:mFLR}) is intrinsic to the Riemannian
structures of the manifold.

According to Theorem 2.1 of \citet{Bhattacharya2003}, the process
$\Log_{\mu(t)}X(t)$ is centered at its mean function, that is, $\expect\Log_{\mu(t)}X(t)\equiv0$,
which also implies that $\alpha=\expect Y$. Since the focus is the
Riemannian slope function $\beta$, in the sequel, without loss of generality, we assume
that $Y$ is centered and hence $\alpha=0$. In practice, we may use sample mean of $Y$ for centering. With the same reasoning as \citet{Cardot2003},
one can see that $\Log_{\mu}\beta$ and hence $\beta$ are identifiable
if \ref{cond:exist-mean-function}, \ref{cond:cut-locus} and the
following condition hold. 

\begin{description}[labelwidth=1.2cm,leftmargin=1.4cm,align=left]
	\item[\mylabel{cond:flr-identifiability}{C.1}] 
	
	The random pair $(X,Y)$ satisfies $\sum_{k=1}^{\infty}\lambda_{k}^{-2}[\expect\{Y\vfinnerprod[\mu]{\Log_{\mu}X,\phi_{k}}\}]^{2}<\infty$,
	where $(\lambda_{k},\phi_{k})$ are eigenpairs of the covariance operator
	$\covarop$ of $\Log_{\mu}X$.
\end{description}

To estimate the Riemannian slope function $\beta$, by analogy to
the FPCA approach in traditional functional linear regression \citep{Hall2007c},
we represent $\beta$ by $\beta=\Exp_{\mu}\sum_{k}b_{k}\phi_{k}$.
Then $b_{k}=\lambda_{k}^{-1}a_{k}$ with $a_{k}=\vfinnerprod{\chi,\phi_{k}}$
and $\chi(t)=\expect\{Y\Log_{\mu(t)}X(t)\}$. The empirical iRFPCA estimate
of $\beta$ is then given by 
\begin{equation}
\hat{\beta}=\Exp_{\hat{\mu}}\sum_{k=1}^{K}\hat{b}_{k}\hat{\phi}_{k},\label{eq:PCR-estimator}
\end{equation}
where $\hat{b}_{k}=\hat{\lambda}_{k}^{-1}\hat{a}_{k}$, $\hat{a}_{k}=\vfinnerprod[\hat{\mu}]{\hat{\chi},\hat{\phi}_{k}}$,
$\hat{\chi}(t)=n^{-1}\sum_{i=1}^{n}(Y_{i}-\bar{Y})\Log_{\hat{\mu}(t)}X_{i}(t)$
and $\bar{Y}=n^{-1}\sum_{i=1}^{n}Y_{i}$. We can also obtain a Tikhonov
estimator $\tilde{\beta}$ as follows. For $\rho>0$, define $\hat{\covarop}^{+}=(\hat{\covarop}+\rho I_{\hat{\mu}})^{-1}$,
where $I_{\hat{\mu}}$ is the identity operator on $\mathscr{T}(\hat{\mu})$.
The Tikhonov estimator is given by 
\begin{equation}
\tilde{\beta}=\Exp_{\hat{\mu}}(\hat{\covarop}^{+}\hat{\chi}).\label{eq:Tikhonov-estimator}
\end{equation}

In practice, it is convenient to conduct computation with respect to an
orthonormal frame $\vec E$. For each $X_{i}$, the $\vec E$-coordinate
along the curve $\hat{\mu}$ is denoted by $\hat{Z}_{\vec E,i}$. Note
that by Theorem 2.1 of \citet{Bhattacharya2003}, $n^{-1}\sum_{i=1}^{n}\Log_{\hat{\mu}}X_{i}=0$
implies that $n^{-1}\sum_{i=1}^{n}\hat{Z}_{\vec E,i}(t)\equiv0$.
Thus, the empirical covariance function $\hat{\covarop}_{\vec E}$
is computed by $\hat{\covarop}_{\vec E}(s,t)=n^{-1}\sum_{i=1}^{n}\hat{Z}_{\vec E,i}(s)\hat{Z}^{T}_{\vec E,i}(t)$.
Then, $\hat{\chi}_{\vec E}=n^{-1}\sum_{i=1}^{n}(Y_{i}-\bar{Y})\hat{Z}_{\vec E,i}$
and $\hat{a}_{k}=\int_{\tdomain}\hat{\chi}^T_{\vec E,i}(t)\hat{\phi}_{\vec E,k}(t)\diffop\upsilon(t)$,
recalling that $\hat{\phi}_{\vec E,k}$ are eigenfunctions of $\hat{\covarop}_{\vec E}$.
Finally, the $\vec E$-coordinate of $\Log_{\hat{\mu}}\hat{\beta}$
is given by $\sum_{k=1}^{K}\hat{b}_{k}\hat{\phi}_{\vec E,k}$. For
the Tikhonov estimator, we first observe that $\hat{\covarop}^{+}$
shares the same eigenfunctions with $\hat{\covarop}$, while the eigenvalues
$\hat{\lambda}_{k}^{+}$ of $\hat{\covarop}^{+}$ is related to $\hat{\lambda}_{k}$
by $\hat{\lambda}_{k}^{+}=1/(\hat{\lambda}_{k}+\rho)$. Therefore,
$\hat{\covarop}^{+}\hat{\chi}=\sum_{k=1}^{\infty}\hat{\lambda}_{k}^{+}\vfinnerprod[\hat{\mu}]{\hat{\phi}_{k},\hat{\chi}}\hat{\phi}_{k}$
and its $\vec E$-coordinate is given by $\sum_{k=1}^{\infty}\hat{\lambda}_{k}^{+}\hat{a}_{k}\hat{\phi}_{\vec E,k}$. We emphasize that both estimators $\hat{\beta}$ and $\tilde{\beta}$ are intrinsic and hence independent of frames, while their coordinate representations depend on the choice of frames. \lincomment{In addition, as a special case,  when $\manifold$ is a Euclidean submanifold, an argument similar to that in Section \ref{subsec:euclidean-submanifold} can show that, if one treats $X$ as an ambient random process and adopts the FPCA and Tikhonov regularization approaches \citep{Hall2007c} to estimate the slope function $\beta$ in the ambient space,  then  the  estimates are the ambient representation of our estimates $\Log_{\hat\mu} \hat\beta$ and $\Log_{\hat\mu}\tilde{\beta}$ in \eqref{eq:PCR-estimator} and \eqref{eq:Tikhonov-estimator}, respectively.}

\subsection{Asymptotic properties}

In order to derive convergence of the iRFPCA estimator and the
Tikhonov estimator, we shall assume the sectional curvature of the
manifold is bounded from below by $\kappa$ to exclude pathological
cases. 
%
%
The compact support condition on $X$ in the case
$\kappa<0$ might be relaxed to weaker assumptions on the tail decay
of the distribution of $\Log_{\mu}X$. Such weaker conditions do
not provide more insight for our derivation, but 
complicate the proofs significantly, which is not pursued further. 

\begin{description}[labelwidth=1.2cm,leftmargin=1.4cm,align=left]
	\item[\mylabel{cond:X-moment-flr}{C.2}] 
	
	If $\kappa<0$, $X$ is assumed to lie in a compact subset $\mathcal{K}$
	almost surely. 
	Moreover, errors $\varepsilon_{i}$ are identically distributed
	with zero mean and variance not exceeding a constant $C>0$.
\end{description}
The follow conditions concern the spacing and the decay rate of eigenvalues
$\lambda_{k}$ of the covariance operator, as well as the strength
of the signal $b_{k}$. They are standard in the literature of functional
linear regression, for example, \citet{Hall2007c}.
\begin{description}[labelwidth=1.2cm,leftmargin=1.4cm,align=left]
	\setlength\itemsep{1.5mm}
	\item[\mylabel{cond:eigen-decay}{C.3}]  
	
	For $k\geq1$, $\lambda_{k}-\lambda_{k+1}\geq Ck^{-\alpha-1}$.
	
	\item[\mylabel{cond:beta-decay}{C.4}] 
	
	$|b_{k}|\leq Ck^{-\varrho}$, $\alpha>1$ and $(\alpha+1)/2<\varrho$.
	
\end{description}
Let $\mathcal{F}(C,\alpha,\varrho)$ be the collection of distributions
$f$ of $(X,Y)$ satisfying conditions \ref{cond:X-moment-flr}-\ref{cond:beta-decay}.
The following theorem establish the convergence rate of the iRFPCA
estimator $\hat{\beta}$ for the class of models in $\mathcal{F}(C,\alpha,\varrho)$. 
\begin{thm}
	\label{thm:fpcr-estimator}Assume that conditions \ref{cond:exist-mean-function},
	\ref{cond:cut-locus-prime}, \ref{cond:manifold-property}-{\ref{cond:Lipschitz-moments}}
	and \ref{cond:flr-identifiability}-\ref{cond:beta-decay} hold. If
	$K\asympeq n^{1/(4\alpha+2\varrho+2)}$, then 
	\[
	\lim_{c\rightarrow\infty}\underset{n\rightarrow\infty}{\lim\sup}\sup_{f\in\mathcal{F}}\prob_{f}\left\{ \int_{\tdomain}d_{\manifold}^{2}(\hat{\beta}(t),\beta(t))\diffop\upsilon(t)>cn^{-\frac{2\varrho-1}{2\varrho+4\alpha+2}}\right\} =0.
	\]
\end{thm}

For the Tikhonov estimator $\tilde{\beta}$, we have a similar result.
Instead of conditions \ref{cond:eigen-decay}-\ref{cond:beta-decay},
we make the following assumptions, which again are standard in the functional data 
literature.

\begin{description}[labelwidth=1.2cm,leftmargin=1.4cm,align=left]
	\setlength\itemsep{1.5mm}
	\item[\mylabel{cond:eigen-decay-tik}{C.5}] 
	
	$k^{-\alpha}\leq C\lambda_{k}$.
	
	\item[\mylabel{cond:beta-decay-tik}{C.6}] 
	
	$|b_{k}|\leq Ck^{-\varrho}$, $\alpha>1$ and $\alpha-1/2<\varrho$.
\end{description}
Let $\mathcal{G}(C,\alpha,\varrho)$ be the class of distributions
of $(X,Y)$ that satisfy \ref{cond:X-moment-flr} and \ref{cond:eigen-decay-tik}-\ref{cond:beta-decay-tik}.
The following result shows that the convergence rate of $\tilde{\beta}$
is at least $n^{-(2\varrho-\alpha)/(\alpha+2\varrho)}$ in terms of integrated
square geodesic distance.
\begin{thm}
	\label{thm:tikhonov-estimator-flr}Assume that conditions \ref{cond:exist-mean-function},
	\ref{cond:cut-locus-prime}, \ref{cond:manifold-property}-{\ref{cond:Lipschitz-moments},
		\ref{cond:flr-identifiability}}-\ref{cond:X-moment-flr} and \ref{cond:eigen-decay-tik}-\ref{cond:beta-decay-tik}
	hold. If $\rho\asympeq n^{-\alpha/(\alpha+2\varrho)}$, then
	\[
	\lim_{c\rightarrow\infty}\underset{n\rightarrow\infty}{\lim\sup}\sup_{f\in\mathcal{G}}\prob_{f}\left\{ \int_{\tdomain}d_{\manifold}^{2}(\tilde{\beta}(t),\beta(t))\diffop\upsilon(t)>cn^{-\frac{2\varrho-\alpha}{2\varrho+\alpha}}\right\} =0.
	\]
\end{thm}

It is important to point out that the theory in \citet{Hall2007c}
is formulated for Euclidean functional data and hence does not apply to
Riemannian functional data. In particular, their proof machinery 
depends on the linear structure of the sample mean function $n^{-1}\sum_{i=1}^{n}X_{i}$ for Euclidean functional data. However, the intrinsic empirical mean generally
does not admit an analytic expression, which hinges derivation of the 
optimal convergence rate. We leave the refinement on minimax rates of iRFPCA and Tikhonov estimators
to future research. \rev{Note that model \eqref{eq:mFLR} can be extended to include a finite and fixed number of scalar predictors with slight modification, and the asymptotic properties of $\hat{\beta}$ and $\tilde{\beta}$ remain unchanged.}

\section{Numerical examples\label{sec:Numerical-Evidence}}

\subsection{Simulation studies}

\lincomment{We consider two manifolds that are frequently encountered in practice\footnote{Implementation of iRFPCA on these manifolds can be found on Github: https://github.com/linulysses/iRFDA.}. The first
	one is the unit sphere $\usphered$ which is a compact nonlinear Riemannian
	submanifold of $\real^{d+1}$ for a positive integer $d$. The sphere can be used to model compositional data, as exhibited in \citet{Dai2017} which also provides details of the geometry of $\usphered$. 
	Here we consider the case of $d=2$. The sphere $\usphere$ consists of points $(x,y,z)\in\real^{3}$
	satisfying $x^{2}+y^{2}+z^{2}=1$.} Since the intrinsic Riemannian geometry of $\usphere$ is the
same as the one inherited from its ambient space (referred to as ambient
geometry hereafter), according to the discussion in Section \ref{subsec:euclidean-submanifold}, the ambient approach to FPCA and functional linear regression yields the same results as our intrinsic approach.

\lincomment{The other manifold considered is the space of $m\times m$ symmetric positive
	definite matrices, denoted by $\spd$. The space $\spd$ includes nonsingular  covariance matrices which naturally arise from the study of DTI  data \citep{Dryden2009} and functional connectivity \citep{Friston2011}.
	Transitionally,
	it is treated as a convex cone of the linear space $\sym$
	of symmetric $m\times m$ matrices. However, as discussed in \citet{Arsigny2007},
	this geometry suffers from the undesired swelling effect. To rectify
	the issue, several intrinsic geometric structures have been proposed, including the affine-invariant metric \citep{Moakher2005} 
	which is also known as the trace metric (Chapter XII, \citealt{Lang1999}). 
	This metric is invariant to affine transformation (i.e., change of coordinates) on $S\in \spd$ and thus suitable for covariance matrices; see \cite{Arsigny2007} as well as \cite{Fletcher2007} for more details. Moreover, the affine-invariant metric has a negative sectional curvature, and thus the Fr\'echet mean is unique if it exists. 
	In our simulation, we consider $m=3$.} We emphasize that the affine-invariant geometry of $\spd$ is
different from the geometry inherited from the linear space
$\sym$. Thus, the ambient RFPCA of \citet{Dai2017} might yield inferior performance on this manifold.

We simulate data as follows. First, the time
domain is set to $\tdomain=[0,1]$. The mean curves for $\usphere$ and $\spd$ are, respectively, $\mu(t)=(\sin\varphi(t)\cos\theta(t),\linebreak[0]\sin\varphi(t)\sin\theta(t),\cos\varphi(t))$
with $\theta(t)=2t^{2}+4t+1/2$ and $\varphi(t)=(t^{3}+3t^{2}+t+1)/2$,
and  $\mu(t)=(t^{0.4}, 0.5t, 0.1t^{1.5}; 0.5t, t^{0.5}, 0.5t; 
0.1t^{1.5},  0.5t$,  $t^{0.6})$ that is a $3 \times 3$ matrix.
The Riemannian random processes are produced in accordance to $X=\Exp\left(\sum_{k=1}^{20}\sqrt{\lambda_{k}}\xi_{k}\phi_{k}\right)$, where  $\xi_{k}\overset{i.i.d.}{\sim}\mathrm{Uniform}(-\pi/4,\pi/4)$ for $\usphere$ and $\xi_{k}\overset{i.i.d.}{\sim}N(0,1)$ for $\spd$. We set iRFPCs
$\phi_{k}(t)=(A\psi_{k}(t))^{T}\vec E(t)$, where $\vec E(t)=(E_{1}(\mu(t)),\ldots,E_{d}(\mu(t)))$
is an orthonormal frame over the path $\mu$, $\psi_{k}(t)=(\psi_{k,1}(t),\ldots,\psi_{k,d}(t))^{T}$
with $\psi_{k,j}$ being orthonormal Fourier basis functions on $\tdomain$,
and $A$ is an orthonormal matrix that is randomly generated but fixed
throughout all simulation replicates. We take $\lambda_{k}=2k^{-1.2}$ for all manifolds. Each
curve $X(t)$ is observed at $M=101$ regular design points $t=0,0.01,\ldots,1$.
The slope function is $\beta=\sum_{k=1}^{K}c_{k}\phi_{k}$
with $c_{k}=3k^{-2}/2$. \rev{Two different types of distribution for $\varepsilon$ in \eqref{eq:mFLR} are considered, namely, normal and Student's $t$ distribution with degree of freedom $\mathrm{df}=2.1$. Note that the latter is a heavy-tailed distribution, with a smaller $\mathrm{df}$ suggesting a heavier tail and $\mathrm{df}>2$ ensuring the existence of variance.} In addition, the noise $\varepsilon$ is scaled to make the signal-to-noise ratio equal to $2$. Three
different training sample sizes are considered, namely, $50$, $150$
and $500$, while the sample size for test data is $5,000$. Each simulation
setup is repeated independently $100$ times.

\lincomment{First, we illustrate the difference between the intrinsic measure and the ambient counterpart for the discrepancy of two random objects residing on different tangent spaces, through the examples of the sphere manifold $\usphere$ and the first two iRFPCs. Recall that the metric of $\usphere$ agrees with its ambient Euclidean geometry, so that both iRFPCA and RFPCA essentially yield the same estimates for iRFPCs. We propose to use the intrinsic root mean integrated squared error (iRMISE) $\{\expect\vfnorm[\mu][2]{\hat{\phi}_{k}\ominus_{\Gamma}\phi_{k}}\}^{1/2}$ to characterize the difference between $\phi_k$ and its estimator $\hat{\phi}_k$, while \cite{Dai2017} adopt the ambient RMISE (aRMISE) $\{\expect\|\hat\phi_k-\phi_k\|_{\real^{d_0}}^2\}^{1/2}$, as discussed in Section \ref{subsec:euclidean-submanifold}. The  numerical results of iRMISE and aRMISE for $\hat\phi_1$ and $\hat\phi_2$, as well as the RMISE for $\hat\mu$, are showed in Table \ref{tab:impact-curvature}. We see that, when  $n$ is small and hence $\hat\mu$ is not sufficiently close to $\mu$, the difference between iRMISE and aRMISE is visible, while such difference decreases as the sample size grows and  $\hat{\mu}$ converges to $\mu$. In particular,  aRMISE is always larger than iRMISE since aRMISE  contains an additional ambient component that is not intrinsic to the manifold, as illustrated on the left panel of Figure \ref{fig:parallel-transport}.}

We now use iRMISE to assess the performance of iRFPCA by comparing to the ambient
counterpart RFPCA proposed by \citet{Dai2017}. 
Table \ref{tab:eigenfunctions}
presents the results for the top $5$ eigenelements. The first observation
is that iRFPCA and RFPCA yield the same results on the manifold $\usphere$, which numerically verifies our discussion in Section \ref{subsec:euclidean-submanifold}.
We notice that in \citet{Dai2017} the quality of estimation of principal components is not evaluated, likely due to the lack of a proper tool to do so. In contrast, our framework of tensor Hilbert space provides an intrinsic gauge (e.g., iRMISE) to naturally compare two vector fields along different curves. For the
case of $\spd$ which is not a Euclidean submanifold, the iRFPCA produces more accurate estimation than RFPCA. In
particular, as sample size grows, the estimation error for iRFPCA
decreases quickly, while the error of RFPCA persists. This coincides with our intuition that when the geometry induced from
the ambient space is not the same as the intrinsic geometry, the ambient RFPCA
incurs loss of statistical efficiency, or even worse, inconsistent
estimation. In summary, the results $\spd$ numerically
demonstrate that the RFPCA proposed by \citet{Dai2017} does not apply
to manifolds that do not have an ambient space or whose intrinsic
geometry differs from its ambient geometry, while our iRFPCA are applicable to such Riemannian manifolds.

\begin{table}
	\caption{\label{tab:impact-curvature}The root mean integrated squared error (RMISE) of the estimation of the mean function, and the intrinsic RMISE (iRMISE) and the ambient RMISE (aRMISE) of the  estimation for the first two eigenfunctions in the case of $\usphere$ manifold. The Monte Carlo standard error based on 100 simulation
		runs is given in parentheses.}
	\vspace{0.05in}
	\resizebox{\textwidth}{!}{
		\begin{tabular}{|c|c|c|c|c|c|c|}
			\hline 
			& \multicolumn{2}{c|}{$n=50$} & \multicolumn{2}{c|}{$n=150$} & \multicolumn{2}{c|}{$n=500$}\tabularnewline
			\hline 
			\hline 
			$\mu$ & \multicolumn{2}{c|}{0.244 (0.056)} & \multicolumn{2}{c|}{0.135 (0.029)} & \multicolumn{2}{c|}{0.085 (0.019)}\tabularnewline
			\hline 
			\hline 
			& iRMISE & aRMISE & iRMISE & aRMISE & iRMISE & aRMISE\tabularnewline
			\hline 
			$\phi_{1}$ & 0.279 (0.073) & 0.331 (0.078) & 0.147 (0.037) & 0.180 (0.042) & 0.086 (0.022) & 0.106 (0.027)\tabularnewline
			\hline 
			$\phi_{2}$ & 0.478 (0.133) & 0.514 (0128) & 0.264 (0.064) & 0.287 (0.061) & 0.147 (0.044) & 0.167 (0.042)\tabularnewline
			\hline 
	\end{tabular}}
\end{table}

\begin{table}
	\caption{\label{tab:eigenfunctions}Intrinsic root integrated mean squared error (iRMISE) of estimation for eigenelements.
		The first column denotes the manifolds, where $\protect\usphere$
		is the unit sphere 
		and $\protect\spd$ is the space of $m\times m$ symmetric
		positive-definite matrices endowed with the affine-invariant metric. In the
		second column, $\phi_{1},\ldots,\phi_{5}$ are the top five intrinsic
		Riemannian functional principal components. Columns 3-5 are (iRMISE) of the iRFPCA estimators for $\phi_{1},\ldots,\phi_{5}$
		with different sample sizes, while columns 5-8 are iRMISE for the RFPCA
		estimators. The Monte Carlo standard error based on 100 simulation
		runs is given in parentheses.}
	\vspace{0.05in}
	\resizebox{\textwidth}{!}{
		\begin{tabular}{|c|c|c|c|c|c|c|c|}
			\hline 
			\multirow{2}{*}{Manifold} & \multirow{2}{*}{FPC} & \multicolumn{3}{c|}{iRFPCA} & \multicolumn{3}{c|}{RFPCA}\tabularnewline
			\cline{3-8} 
			&  & $n=50$ & $n=150$ & $n=500$ & $n=50$ & $n=150$ & $n=500$\tabularnewline
			\hline 
			\multirow{5}{*}{$\usphere$} & $\phi_{1}$ & $0.279\,(.073)$ & $0.147\,(.037)$ & $0.086\,(.022)$ & $0.279\,(.073)$ & $0.147\,(.037)$ & $0.086\,(.022)$\tabularnewline
			\cline{2-8} 
			& $\phi_{2}$ & $0.475\,(.133)$ & $0.264\,(.064)$ & $0.147\,(.044)$ & $0.475\,(.133)$ & $0.264\,(.064)$ & $0.147\,(.044)$\tabularnewline
			\cline{2-8} 
			& $\phi_{3}$ & $0.647\,(.153)$ & $0.389\,(.120)$ & $0.206\,(.054)$ & $0.647\,(.153)$ & $0.389\,(.120)$ & $0.206\,(.054)$\tabularnewline
			\cline{2-8} 
			& $\phi_{4}$ & $0.818\,(.232)$ & $0.502\,(.167)$ & $0.261\,(.065)$ & $0.818\,(.232)$ & $0.502\,(.167)$ & $0.261\,(.065)$\tabularnewline
			\cline{2-8} 
			& $\phi_{5}$ & $0.981\,(.223)$ & $0.586\,(.192)$ & $0.329\,(.083)$ & $0.981\,(.223)$ & $0.586\,(.192)$ & $0.329\,(.083)$\tabularnewline
			\hline 
			\multirow{5}{*}{$\spd$} & $\phi_{1}$ & $0.291\,(.105)$ & $0.155\,(.046)$ & $0.085\,(.025)$ & $0.707\,(.031)$ & $0.692\,(.021)$ & $0.690\,(.014)$\tabularnewline
			\cline{2-8} 
			& $\phi_{2}$ & $0.523\,(.203)$ & $0.283\,(.087)$ & $0.143\,(.040)$ & $0.700\,(.095)$ & $0.838\,(.113)$ & $0.684\,(.055)$\tabularnewline
			\cline{2-8} 
			& $\phi_{3}$ & $0.734\,(.255)$ & $0.418\,(.163)$ & $0.206\,(.067)$ & $0.908\,(.116)$ & $0.904\,(.106)$ & $0.981\,(.039)$\tabularnewline
			\cline{2-8} 
			& $\phi_{4}$ & $0.869\,(.251)$ & $0.566\,(.243)$ & $0.288\,(.086)$ & $0.919\,(.115)$ & $1.015\,(.113)$ & $0.800\,(.185)$\tabularnewline
			\cline{2-8} 
			& $\phi_{5}$ & $1.007\,(.231)$ & $0.699\,(.281)$ & $0.378\,(.156)$ & $0.977\,(.100)$ & $1.041\,(.140)$ & $1.029\,(.058)$\tabularnewline
			\hline 
	\end{tabular}}
\end{table}

\begin{table}[h]
	\caption{\label{tab:FLR}Estimation quality of slope function $\beta$ and prediction of $y$
		on test datasets. The second column indicates the distribution of noise, while the third column indicates the manifolds, where
		$\protect\usphere$ is the unit sphere
		and $\protect\spd$ is the space of $m\times m$ symmetric positive-definite matrices endowed
		with the affine-invariant metric. Columns 4-6 are performance of the iRFLR
		on estimating the slope curve $\beta$ and predicting the response
		on new instances of predictors, while columns 7-9 are performance
		of the RFLR method. Estimation quality of slope curve is quantified
		by intrinsic root mean integrated squared errors (iRMISE), while the performance
		of prediction on independent test data is measured by root mean
		squared errors (RMSE). The Monte Carlo standard error based on 100
		simulation runs is given in parentheses.}
	\vspace{0.05in}
	\centering
	\resizebox{\textwidth}{!}{
		\begin{tabular}{|c|c|c|c|c|c|c|c|c|}
			\hline 
			\multicolumn{3}{|c|}{}  & \multicolumn{3}{c|}{iRFLR} & \multicolumn{3}{c|}{RFLR}\tabularnewline
			\cline{4-9} 
			\multicolumn{3}{|c|}{} & $n=50$ & $n=150$ & $n=500$ & $n=50$ & $n=150$ & $n=500$\tabularnewline
			\hline 
			\multirow{4}{*}{\rotatebox[origin=c]{90}{Estimation}} & \multirow{2}{*}{\rotatebox[origin=c]{90}{normal}} & $\usphere$ & $0.507\,(0.684)$ & $0.164\,(0.262)$ & $0.052\,(0.045)$ & $0.507\,(0.684)$ & $0.164\,(0.262)$ & $0.052\,(0.045)$\tabularnewline
			\cline{3-9} 
			& & SPD & $1.116\,(2.725)$ & $0.311\,(0.362)$ & $0.100\,(0.138)$ & $2.091\,(0.402)$ & $1.992\,(0.218)$ & $1.889\,(0.126)$\tabularnewline
			\cline{2-9}  & \multirow{2}{*}{\rotatebox[origin=c]{90}{$t(2.1)$}} & $\usphere$ & $0.575\,(0.768)$ & $0.183\,(0.274)$ & $0.053\,(0.050)$ & $0.575\,(0.768)$  & $0.183\,(0.274)$ & $0.053\,(0.050)$ \tabularnewline
			\cline{3-9} 
			& & SPD & $1.189\,(2.657)$ & $0.348\,(0.349)$  & $0.108\,(0.141)$ & $2.181\,(0.439)$  & $1.942\,(0.209)$ & $1.909\,(0.163)$ 
			\tabularnewline
			\hline 
			\multirow{4}{*}{\rotatebox[origin=c]{90}{Prediction}} & \multirow{2}{*}{\rotatebox[origin=c]{90}{normal}} & $\usphere$ & $0.221\,(0.070)$ & $0.135\,(0.046)$ & $0.083\,(0.019)$ & $0.221\,(0.070)$ & $0.135\,(0.046)$ & $0.083\,(0.019)$\tabularnewline
			\cline{3-9} 
			& & SPD & $0.496\,(0.184)$ & $0.284\,(0.092)$ & $0.165\,(0.062)$ & $0.515\,(0.167)$ & $0.328\,(0.083)$ & $0.248\,(0.047)$\tabularnewline
			\cline{2-9}  & \multirow{2}{*}{\rotatebox[origin=c]{90}{$t(2.1)$}} & $\usphere$ & $0.251\,(0.069)$ & $0.142\,(0.042)$ & $0.088\,(0.020)$ & $0.251\,(0.069)$ & $0.142\,(0.042)$ & $0.088\,(0.020)$ \tabularnewline
			\cline{3-9} 
			& & SPD & $0.532\,(0.189)$ & $0.298\,(0.097)$ & $0.172\,(0.066)$ & $0.589\,(0.185)$ & $0.360\,(0.105)$  & $0.268\,(0.051)$
			\tabularnewline
			\hline 
	\end{tabular}}
\end{table}

For functional linear regression, we adopt iRMISE to quantify the quality
of the estimator $\hat{\beta}$ for slope function $\beta$, and assess
the prediction performance by
prediction RMSE on independent test dataset. For comparison, we also
fit the functional linear model using the principal components produced
by RFPCA \citep{Dai2017}, and hence we refer to this competing method
as RFLR. For both methods, the tuning parameter which is the number
of principal components included for $\hat{\beta}$, is selected by
using an independent validation data of the same size of the training
data to ensure fair comparison between two methods. The simulation
results are presented in Table \ref{tab:FLR}. As expected, we observe
that on $\usphere$ both methods produce the same results. For the SPD manifold, in terms of estimation, we see that iRFLR yields
far better estimators than RFLR does. Particularly, we again observe
that, the quality of RFLR estimators does not improve significantly
when sample size increases, in contrast to estimators based on
the proposed iRFLR. For prediction, iRFLR outperforms
RFLR by a significant margin. Interestingly, comparing to estimation
of slope function where the RFLR estimator is much inferior to the
iRFLR one, the prediction performance by RFLR is relatively closer to that by iRFLR. We attribute this to the smoothness
effect brought by the integration in model \eqref{eq:mFLR}.
Nevertheless, although the integration cancels out certain discrepancy
between the intrinsic and the ambient geometry, the loss
of efficiency is inevitable for the RFLR method that is bound to the
ambient spaces. \rev{In addition, we observe that the performance of both methods for Gaussian noise  is slightly better than that in the case of heavy-tailed noise. }

\subsection{Data application}
We apply the proposed iRFPCA and iRFLR to analyze the relationship between
functional connectivity and behavioral data from the HCP 900 subjects
release \citep{Essen2013}. Although neural effects on language \citep{Binder1997},
emotion \citep{Phana2002} and fine motor skills \citep{Dayan2011}
have been extensively studied in the literature, scarce is the exploration
on human behaviors that do not seem related to neural activities,
such as endurance. Nevertheless, a recent research by \citet{Raichlen2016}
suggests that endurance can be related to functional connectivity.
Our goal is to study the endurance performance of subjects based on
their functional connectivity.

The data consists of $n=330$ subjects who are healthy young  adults, in which each subject is asked to walk for two minutes and the distance in feet
is recorded. Also, each subject participates in a motor task, where
participants are asked to act according to presented visual cues,
such as tapping their fingers, squeezing their toes or moving their
tongue. During the task, the brain of each subject is scanned and
the neural activities are recorded at 284 equispaced time points. After preprocessing,
the average BOLD (blood-oxygen-level dependent) signals at 68 different
brain regions are obtained. The details of experiment and data acquisition
can be found in the reference manual of WU-Minn HCP 900 Subjects Data
Release that is available on the website of human connectome project.

Our study focuses on $m=6$ regions that are related to the primary motor cortex,
including precentral gyrus, Broca's area, etc. At each design time
point $t$, the functional connectivity of the $i$th subject is represented by
the covariance matrix $S_i(t)$ of BOLD signals from regions of interest (ROI). \lincomment{To practically compute $S_i(t)$, let $V_{it}$ be an $m$-dimensional column vector that represents the BOLD signals at time $t$ from the $m$ ROIs of the $i$th subject. We then adopt a local sliding window approach \citep{Park2017} to compute $S_i(t)$  by $$S_i(t)=\frac{1}{2h+1}\sum_{j=t-h}^{t+h}(V_{ij}-\bar{V}_{it})(V_{ij}-\bar{V}_{it})^T\quad\text{with}\quad \bar{V}_{it}=\frac{1}{2h+1}\sum_{j=t-h}^{t+h}V_{ij},$$ where   $h$ is a positive integer that represents the length of the sliding window to compute $S_i(t)$ for $t=h+1,h,\ldots,284-h$. Without loss of generality, we reparameterize each $S_i(\cdot)$ from the domain $[h+1,284-h]$ to $[1,284-2h]$. In practice, the window length $h$ is required to be sufficiently large to make  $S_t$ nonsingular, but not too large in order to avoid significant bias. We found that the analysis below is not sensitive to the choice of $h$  as long as it is in a reasonable range, for example, $h\in [m,3m]$. Thus, we simply set $h=2m$ throughout our study.}

The functional predictors are then the functions $S_i(\cdot)$ whose values are covariance matrices,
or SPDs, and hence is regarded as
a curve on $\spd$. The heat maps of the first two iRFPCs along 8 time points are depicted in the top two rows of Figure \ref{fig:phi-beta}, which shows the modes of variation in the functional connectivity with contrasts between early and middle/later times in both iRFPCs.
The scalar response is the distance that subjects walked in
two minutes. A scatterplot of the first two principal component scores labeled by the response variable for 10 randomly selected subjects is given in Figure \ref{fig:y-vs-2pcs} respectively for iRFPCA and RFPCA, to visualize the difference of scores produced by these methods. Gender, age, gait speed and strength are included as baseline covariates, selected by the forward-stepwise selection method (Section 3.3, \citealt{Hastie2009}). Among these covariates, gender and age are in accordance with the common sense about endurance, while gait speed and muscle strength could be influential since endurance is measured by the distance walked in two minutes. Our primary interest is to assess the significance of the functional predictor when effect of the baseline covariates is controlled.

To fit the intrinsic functional linear model, we adopt the cross-validation procedure
to select the number of components to be included in representing
the Riemannian functional predictor and the Riemannian slope function
$\beta$.  For assessment, we conduct 100 runs of 10-fold cross-validation,
where in each run we permute the data independently. In each run,
the model is fitted on 90\% data and the MSE for predicting the walking
distance is computed on the other 10\% data for both iRFLR and RFLR
methods. The fitted intrinsic Riemannian slope function $\mbox{Log}_{\hat{\mu}} \hat{\beta}$ displayed in the bottom panel of Figure \ref{fig:phi-beta} shows the pattern of weight changes. 
The MSE for iRFLR is reduced by around 9.7\%, compared to that for RFLR. 
Moreover, the $R^{2}$ for iRFLR is
$0.338$, with a p-value $0.012$ based on a permutation
test of 1000 permutations, which is significant at level 5\%. In  contrast, the $R^2$ for RFLR drops to $0.296$ and the p-value is $0.317$ that does not spell significance at all. 

%

\begin{figure}[h]
	\includegraphics[scale=0.6]{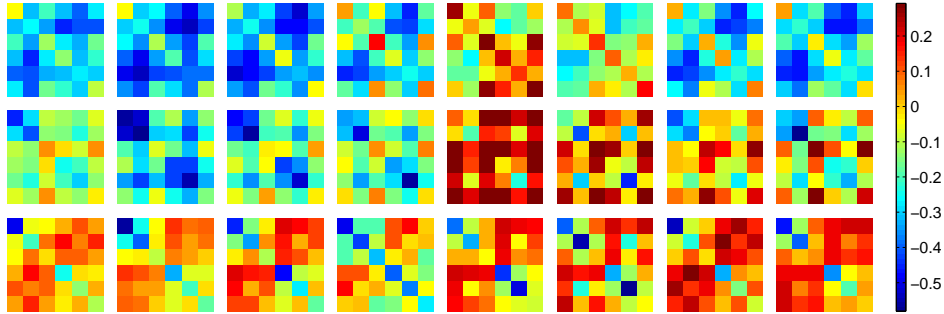}
	\caption{The $\hat{\phi}_1$, $\hat{\phi}_2$ and $\Log_{\hat{\mu}}\hat{\beta}$ at time points $t=1+24k,\,\,k=1,\ldots,8$.}
	\label{fig:phi-beta}
\end{figure}


\begin{figure}[h]
	\includegraphics[scale=0.6]{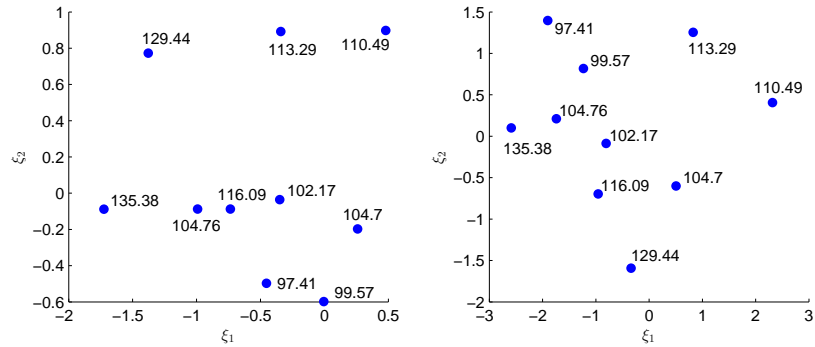}
	\caption{Scatterplot of the first two Riemannian functional principal component scores labeled by the response variable endurance in iRFPCA (left panel) and RFPCA (right panel) for 10 randomly selected subjects.}
	\label{fig:y-vs-2pcs}
\end{figure}

\appendix

\section{Background on Riemannian Manifold}

We introduce geometric concepts related to Riemannian manifolds from an intrinsic perspective  without referring to any ambient
space.

A smooth manifold is a differentiable manifold with all transition
maps being $C^{\infty}$ differentiable. For each point $p$ on the
manifold $\manifold$, there is a linear space $\tangentspace p$
of tangent vectors which are derivations. A derivation is a linear
map that sends a differentiable function on $\manifold$ into $\real$
and satisfies the Leibniz property. For example, if $D_{v}$ is the
derivation associate with the tangent vector $v$ at $p$, then $D_{v}(fg)=(D_{v}f)\cdot g(p)+f(p)\cdot D_{v}(g)$
for any $f,g\in A(\manifold)$, where $A(\manifold)$ is a collection
of real-valued differentiable functions on $\manifold$. For submanifolds
of a Euclidean space $\real^{\dim_0}$ for some $d_0>0$, tangent vectors are often perceived
as vectors in $\real^{\dim_0}$ that are tangent to the submanifold
surface. If one interprets a Euclidean tangent vector as a directional
derivative along the vector direction, then Euclidean tangent vectors
coincide with our definition of tangent vectors on a general manifold.
The linear space $T_{p}\manifold$ is called the tangent space at
$p$. The disjoint union of tangent spaces at each point constitutes
the tangent bundle, which is also equipped with a smooth manifold
structure induced by $\manifold$. The tangent bundle of $\manifold$
is conventionally denoted by $T\manifold$. A (smooth) vector field
$V$ is a map from $\manifold$ to $T\manifold$ such that $V(p)\in\tangentspace p$
for each $p\in\manifold$. It is also called a smooth section of $T\manifold$.
Noting that a tangent vector is a tensor of type $(0,1)$, a vector
field can be viewed as a kind of tensor field, which assigns a tensor
to each point on $\manifold$. A vector field along a curve $\gamma:I\rightarrow\manifold$
on $\manifold$ is a map $V$ from an interval $I\subset\real$ to
$T\manifold$ such that $V(t)\in\tangentspace{\gamma(t)}$. For a
smooth function from a manifold $\manifold$ and to another manifold
$\mathcal{N}$, the differential $\diffop\varphi_{p}$ of $f$ at
$p\in\manifold$ is a linear map from $\tangentspace p$ to $\tangentspace[\mathcal{N}]{\varphi(p)}$,
such that $\diffop\varphi_{p}(v)(f)=D_{v}(f\circ\varphi)$ for all
$f\in A(\manifold)$ and $v\in\tangentspace{p}$. 

An affine connection $\nabla$ on $\manifold$ is a bilinear mapping
that sends a pair of smooth vector fields $(U,V)$ to another smooth
vector field $\nabla_{U}V$ and satisfies $\nabla_{fU}V=f\nabla_{U}V$
and the Leibniz rule in the second argument, that is, $\nabla_{U}(fV)=\diffop f(U)V+f\nabla_{U}V$
for smooth real-valued functions $f$ on $\manifold$. An interesting
fact is that $(\nabla_{U}V)(p)$ is determined by $U(p)$ and a neighborhood
of $V(p)$. A vector field $V$ is parallel if $\nabla_{U}V=0$ for
all vector fields $U$. When a connection is interpreted as the covariant
derivative of a vector field with respect to another one, a parallel
vector field is intuitively a constant vector field in some sense,
as vanishing derivative indicates constant. Similarly, a vector field
$U$ along a curve $\gamma$ is parallel if $\nabla_{\dot{\gamma}(t)}U=0$
for all $t\in I$, where $\dot{\gamma}(t)$ denotes the differential
of $\gamma$ at $t$. Given a tangent vector $u$ at $\gamma(0)$,
there exists a unique vector field $U$ along $\gamma$ such that $\nabla_{\dot{\gamma}(t)}U=0$
and $U(\gamma(0))=u$. In this case, $U(\gamma(t))$ is the parallel
transport of $u$ along $\gamma$ from $\gamma(0)$ to $\gamma(t)$.
In particular, if $\nabla_{\dot{\gamma}}\dot{\gamma}=0$, then $\gamma$
is called a geodesic of the connection $\nabla$. Also, for any $u\in\tangentspace{\gamma(0)}$,
there exists a unique geodesic $\gamma$ such that $\nabla_{\dot{\gamma}}\dot{\gamma}=0$
and $\dot{\gamma}(0)=u$. Then the exponential map at $p=\gamma(0)$,
denoted by $\Exp_{p}$, is defined by $\Exp_{p}(u)=\gamma(1)$.

A Riemannian manifold is a smooth manifold with a metric tensor $\metric{\cdot}{\cdot}$,
such that for each $p\in\manifold$, the tensor $\metric{\cdot}{\cdot}_{p}$
defines an inner product on $\tangentspace p\times\tangentspace p$.
For Riemannian manifold, the fundamental theorem of Riemannian geometry
says that there is a unique connection that is 1) compatible with
the metric tensor, that is, for every pair of vector fields $U$ and $V$,
and every vector $v\in\tangentspace p$, $v(\metric UV)=\metric{\nabla_{v}U}V+\metric U{\nabla_{v}V}$;
2) torsion-free, that is, for any vector fields $U$ and $V$, $\nabla_{U}V-\nabla_{V}U=[U,V]$,
where $[U,V]$ denotes the Lie bracket of a pair of vector fields
$(U,V)$, and is defined by $[U,V](f)=U(V(f))-V(U(f))$ for all $f\in A(\manifold)$
(by definition, $[U,V]$ is also a vector field). This connection
is called Levi--Civita connection, which is uniquely determined by the
Riemannian metric tensor. The Riemannian curvature tensor $R$ is
a map that sends a triple of smooth vector fields $(U,V,W)$ to another
vector field $R(U,V)W=\nabla_{U}\nabla_{V}W-\nabla_{V}\nabla_{U}W-\nabla_{[U,V]}W$.
Then the sectional curvature is defined for a 2-plane spanned by two
linearly independent tangent vector at the same point $p$, given
by 
\[
K(u,v)=\frac{\metric{R(u,v)v}u_{p}}{\metric uu_{p}\metric vv_{p}-\metric uv_{p}^{2}}\in\real.
\]
The metric tensor also induces a distance function $d_{\manifold}$
on $\manifold$ that turns a connected Riemannian manifold $\manifold$
into a metric space. The function $d_{\manifold}$ is defined in the
following way. For a continuously differentiable curve $\gamma:[a,b]\rightarrow\manifold$,
the length of $\gamma$ is given by $L(\gamma)=\int_{a}^{b}\sqrt{\metric{\dot{\gamma}(t)}{\dot{\gamma}(t)}}\diffop t$.
Then $d_{\manifold}(p,q)$ is the infimum of $L(\gamma)$ over all
continuously differentiable curves joining $p$ and $q$. For a connected and complete Riemannian, given two points on the manifold, there is a minimizing geodesic connecting these two points.

\section{Implementation for $\spd$}
Given $\spd$-valued functional data $X_1,\ldots,X_n$, below we briefly outline the numerical steps to perform iRFPCA. The computation details for $\usphered$ can be found in \cite{Dai2017}.

\begin{description}[labelwidth=1.5cm,leftmargin=1.7cm,align=left]
	\item[Step 1.] Compute the sample Fr\'echet mean $\hat{\mu}$. As there is no analytic solution, the  recursive algorithm developed by  \cite{Cheng2016} can be used.
	\item[Step 2.] Select an orthonormal frame $\vec{E}=(E_1,\dots,E_d)$ along $\hat\mu$. For $\spd$, at each $S\in\spd$, the tangent space $\tangentspace[\spd]S$ is isomorphic to $\sym$. This space has a canonical linearly independent basis $e_1,\ldots,e_d$ with $d=m(m+1)/2$, defined in the following way. For an integer $k\in[1,d]$, let $N_1$ be the largest integer such that $N_1(N_1+1)/2\leq k$. Let $N_2=k-N_1(N_1-1)/2$. Then $e_k$ is defined as the $m\times m$ matrix that has 1 at $(N_1,N_2)$, 1 at $(N_2,N_1)$ and 0 elsewhere. Because the inner product on the space  $\tangentspace[\spd]{\hat{\mu}(t)}$ is given by $$\mathrm{tr}({\hat{\mu}(t)}^{-1/2}U{\hat{\mu}(t)}V{\hat{\mu}(t)}^{-1/2})$$ for $U,V\in \tangentspace[\spd] S$, in general  this basis is not orthonormal in $\tangentspace[\spd] {\hat{\mu}(t)}$.  To obtain an orthonormal basis of $\tangentspace[\spd]{\hat\mu(t)}$ for any given $t$, we can apply the Gram-Schmidt procedure on the basis $e_1,\ldots,e_d$. The orthonormal bases obtained in this way smoothly vary with $t$ and hence form an orthonormal frame of $\spd$ along $\hat\mu$. 
	\item[Step 3.] Compute the $\vec{E}$-coordinate representation $\hat{Z}_{\vec E,i}$ of each $\Log_{\hat\mu}X_i$. For $\spd$, the logarithm map at a generic $S\in\spd$ is given by $\Log_S(Q)=S^{1/2}\log(S^{-1/2}QS^{-1/2})S^{1/2}$ for $Q\in\spd$, where $\log$ denotes the matrix logarithm function. Therefore, $$\Log_{\hat\mu(t)}X_i(t)={\hat{\mu}(t)}^{1/2}\log({\hat{\mu}(t)}^{-1/2}X_i(t){\hat{\mu}(t)}^{-1/2}){\hat{\mu}(t)}^{1/2}.$$ Using the orthonormal basis $E_1(t),\ldots,E_d(t)$ obtained in the previous step, one can compute the coefficient $\hat Z_{\vec E,i}(t)$ representation of $\Log_{\hat\mu(t)}X_i(t)$ for any given $t$.
	\item[Step 4.] Compute the first $K$ eigenvalues $\hat\lambda_1,\dots,\hat{\lambda}_K$ and eigenfunctions $\hat{\phi}_{\vec E,1},\dots,\hat{\phi}_{\vec E,K}$ of the empirical covariance function $\hat{\covarop}_{\vec E}(s,t)=n^{-1}\times\linebreak[0]\sum_{i=1}^n \hat Z_{\vec E,i}(s)\hat Z_{\vec E,i}^T(t)$. This step is generic and does not involve the manifold structure. For $d=1$, the classic univariate FPCA method such as \cite{Hsing2015} can be employed to derive the eigenvalues and eigenfunctions of $\hat{\covarop}_{\vec E}$. When $d>1$, each observed coefficient function $\hat{Z}_{\vec E,i}(t)$ is vector-valued. FPCA for vector-valued functional data can be performed by the methods developed in \cite{Happ2018} or \cite{Wang2008}.
	\item[Step 5.] Compute the scores $\hat{\xi}_{ik}=\int\hat{Z}_{\vec E,i}^T(t)\hat\phi_{\vec E,k}(t)\diffop t$. Finally, compute the approximations of $X_i$ by the first $K$ principal components using $$\hat{X}_i^K(t)=\Exp_{\hat{\mu}(t)}\sum_{k=1}^K \hat\xi_{ik}\hat\phi_{\vec E,k}^T(t)\vec E(t),$$ where for $\spd$, the exponential map at a generic $S$ is given by $$\Exp_S(U)=S^{1/2}\exp(S^{-1/2}US^{-1/2})S^{1/2}$$ for $U\in\tangentspace[\spd]S$, where $\exp$ denotes the matrix exponential function.
\end{description}

\section{Proofs of Main Theorems}
\begin{proof}[Proof of Theorem \ref{prop:hilbertian-vfc}]
	We first show that $\mathscr{T}(\mu)$ is a Hilbert space. It is
	sufficient to prove that the inner product space $\mathscr{T}(\mu)$
	is complete. Suppose $\{V_{n}\}$ is a Cauchy sequence in $\mathscr{T}(\mu)$.
	We will later show that there exists a subsequence $\{V_{n_{k}}\}$
	such that 
	\begin{equation}
	\sum_{k=1}^{\infty}|V_{n_{k+1}}(t)-V_{n_{k}}(t)|<\infty,\qquad\upsilon\text{-}a.s.\label{eq:pf-hsvfc-1}
	\end{equation}
	Since $\tangentspace{\mu(t)}$ is complete, the limit $V(t)=\lim_{k\rightarrow\infty}V_{n_{k}}(t)$
	is $\upsilon$-a.s. well defined and in $\tangentspace{\mu(t)}$.
	Fix any $\epsilon>0$ and choose $N$ such that $n,m\geq M$ implies
	$\vfnorm[\mu]{V_{n}-V_{m}}\leq\epsilon$. Fatou's lemma applying to the function $|V(t)-V_{m}(t)|$ implies that
	if $m\geq N$, then $\vfnorm[\mu][2]{V-V_{m}}\leq\underset{k\rightarrow\infty}{\lim\inf}\vfnorm[\mu][2]{V_{n_{k}}-V_{m}}\leq\epsilon^{2}$.
	This shows that $V-V_{m}\in\mathscr{T}(\mu)$.
	Since $V=(V-V_{m})+V_{m}$, we see that $V\in\mathscr{T}(\mu)$. The
	arbitrariness of $\epsilon$ implies that $\lim_{m\rightarrow\infty}\vfnorm[\mu]{V-V_{m}}=0$.
	Because $\vfnorm[\mu]{V-V_{n}}\leq\vfnorm[\mu]{V-V_{m}}+\vfnorm[\mu]{V_{m}-V_{n}}\leq2\epsilon$,
	we conclude that $V_{n}$ converges to $V$ in $\mathscr{T}(\mu)$.
	
	It remains to show (\ref{eq:pf-hsvfc-1}). To do so, we choose $\{n_{k}\}$
	so that $\vfnorm[\mu]{V_{n_{k}}-V_{n_{k+1}}}\leq2^{-k}$. This is
	possible since $V_{n}$ is a Cauchy sequence. Let $U\in\mathscr{T}(\mu)$.
	By Cauchy--Schwarz inequality, $\int_{\tdomain}|U(t)|\cdot|V_{n_{k}}(t)-V_{n_{k+1}}(t)|\diffop\upsilon(t)\leq\vfnorm[\mu]U\vfnorm[\mu]{V_{n_{k}}-V_{n_{k+1}}}\leq2^{-k}\vfnorm[\mu]U$.
	Thus, 
	$\sum_{k}\int_{\tdomain}|U(t)|\cdot|V_{n_{k}}(t)-V_{n_{k+1}}(t)|\diffop\upsilon(t)\leq\vfnorm[\mu]U<\infty.$
	Then (\ref{eq:pf-hsvfc-1}) follows, because otherwise, if the series
	diverges on a set $A$ with $\upsilon(A)>0$, then a choice of $U$
	such that $|U(t)|>0$ for $t\in A$ contradicts the above inequality.
	
	Now let $\vec{E}$ be a measurable orthonormal frame. For every element
	$U\in\mathscr{T}(\mu)$, the coordinate representation of $U$ with
	respect to $\vec E$ is denoted by $U_{\vec E}$. One can see that
	$U_{\vec E}$ is an element in the Hilbert space $\ltwo(\tdomain,\real^{\dim})$
	of square integrable $\real^{\dim}$-valued measurable functions with
	norm $\|f\|_{\ltwo}=\{\int_{\tdomain}|f(t)|^{2}\diffop\upsilon(t)\}^{1/2}$
	for $f\in\ltwo(\tdomain,\real^{\dim})$. If we define the map $\Upsilon:\mathscr{T}(\mu)\rightarrow\ltwo(\tdomain,\real^{\dim})$
	by $\Upsilon(U)=U_{\vec E}$, we can immediately see that $\Upsilon$
	is a linear map. It is also surjective, because for any $f\in\ltwo(\tdomain,\real^{\dim})$,
	the vector field $U$ along $\mu$ given by $U_{f}(t)=f(t)\vec E(\mu(t))$
	for $t\in\tdomain$ is an element in $\mathscr{T}(\mu)$, since $\vfnorm{U_{f}}=\|f\|_{\ltwo}$.
	It can be also verified that $\Upsilon$ preserves the inner product.
	Therefore, it is a Hilbertian isomorphism. Since $\ltwo(\tdomain,\real^{\dim})$
	is separable, the isomorphism between $\ltwo(\tdomain,\real^{\dim})$
	and $\mathscr{T}(\mu)$ implies that $\mathscr{T}(\mu)$ is also separable.
\end{proof}
\begin{proof}[Proof of Proposition \ref{prop:property-Gamma-operator}] The regularity conditions on $f$, $h$ and $\gamma$ ensure that $\Gamma$ and $\Phi$ are measurable. 
	Part \ref{prop:property-Gamma-operator-enu-1}, \ref{prop:property-Gamma-operator-enu-2}
	and \ref{prop:property-Gamma-operator-enu-6} can be deduced from
	the fact that $\mathcal{P}_{f,h}$ is a unitary operator between two
	finite-dimensional real Hilbert spaces and its inverse is $\mathcal{P}_{h,f}$.
	To reduce notational burden, we shall suppress the subscripts $f,h$
	from $\Gamma_{f,h}$ and $\Phi_{f,h}$ below. For Part \ref{prop:property-Gamma-operator-enu-3},
	\begin{align*}
	(\Phi\mathcal{A})(\Gamma U) & =\Gamma(\mathcal{A}\Gamma^{\ast}\Gamma U))=\Gamma(\mathcal{A}U).
	\end{align*}
	To prove Part \ref{prop:property-Gamma-operator-enu-4}, assume $V\in\mathscr{T}(g)$.
	Then, noting that $\Gamma(\Gamma^{\ast}V)=V$ and $\Gamma^{\ast}(\Gamma U)=U$,
	we have
	\begin{align*}
	(\Phi\mathcal{A})((\Phi\mathcal{A}^{-1})V) & =(\Phi\mathcal{A})(\Gamma(\mathcal{A}^{-1}\Gamma^{\ast}V)) =\Gamma(\mathcal{A}\Gamma^{\ast}(\Gamma(\mathcal{A}^{-1}\Gamma^{\ast}V))) \\ & =\Gamma(\mathcal{A}\mathcal{A}^{-1}\Gamma^{\ast}V)=\Gamma(\Gamma^{\ast}V)=V,
	\end{align*}
	and
	\begin{align*}
	(\Phi\mathcal{A}^{-1})(\Phi\mathcal{A}V) & =(\Phi\mathcal{A}^{-1})(\Gamma(\mathcal{A}\Gamma^{\ast}V)) =\Gamma(\mathcal{A}^{-1}\Gamma^{\ast}(\Gamma(\mathcal{A}\Gamma^{\ast}V))) \\ & =\Gamma(\mathcal{A}^{-1}\mathcal{A}\Gamma^{\ast}V)=\Gamma(\Gamma^{\ast}V)=V.
	\end{align*}
	Part \ref{prop:property-Gamma-operator-enu-5} is seen by the following
	calculation: for $V\in\mathscr{T}(g)$,
	\begin{align*}
	(\Phi_{f,g}\sum c_{k}\varphi_{k}\otimes\varphi_{k})V & =\Gamma(\sum c_{k}\vfinnerprod[f]{\varphi_{k},\Gamma^{\ast}V}\varphi_{k})=\sum c_{k}\vfinnerprod[f]{\phi_{k},\Gamma^{\ast}V}\Gamma\varphi_{k}\\
	& =\sum c_{k}\vfinnerprod[g]{\Gamma\varphi_{k},V}\Gamma\varphi_{k}=\left(\sum c_{k}\Gamma\varphi_{k}\otimes\Gamma\varphi_{k}\right)V.
	\end{align*}
\end{proof}
\begin{proof}[Proof of Proposition \ref{prop:truncation-residual}]
	The case $\kappa\geq0$ is already given by \citet{Dai2017} with
	$C=1$. Suppose $\kappa<0$. The second statement follows from the
	first one if we let $O=\mu(t)$, $P=X(t)$ and $Q=X_{K}(t)$ for any
	fixed $t$ and note that $C$ is independent of $t$. 
	
	For the first statement, the inequality is clearly true if $P=O$,
	$Q=O$ or $P=Q$. Now suppose $O$, $P$ and $Q$ are distinct points
	on $\manifold$. The minimizing geodesic curves between these points
	form a geodesic triangle on $\manifold$. By Toponogov's theorem (the
	hinge version), $d_{\manifold}(P,Q)\leq d_{\mathbb{M}_{\kappa}}(P^{\prime},Q^{\prime})$,
	where $\mathbb{M}_{\kappa}$ is the model space with constant sectional
	curvature $\kappa$. For $\kappa<0$, it is taken as the hyperboloid
	with curvature $\kappa$. Let $a=d_{\manifold}(O,P)$, $b=d_{\manifold}(O,Q)$
	and $c=d_{\manifold}(P,Q)$. The interior angle of geodesics connecting
	$O$ to $P$ and $O$ to $Q$ is denoted by $\gamma$. Denote $\delta=\sqrt{-\kappa}$,
	the law of cosine on $\mathbb{M}_{\kappa}$ gives 
	\begin{align*}
	\cosh(\delta c) =& \{\cosh(\delta a)\cosh(\delta b)-\sinh(\delta a)\sinh(\delta b)\}+\\
	& \{\sinh(\delta a)\sinh(\delta b)(1-\cos\gamma)\}\\
	\equiv & I_{1}+I_{2}.
	\end{align*}
	By definition that $\cosh(x)=(e^{x}+e^{-x})/2$ and $\sinh(x)=(e^{x}-e^{-x})/2$,
	$I_{1}=(e^{\delta a-\delta b}+e^{\delta b-\delta a})/2$. By Taylor
	expansion $e^{x}=\sum x^{k}/k!$, we have 
	\begin{equation}
	I_{1}=\sum_{k=0}^{\infty}\frac{(\delta a-\delta b)^{2k}}{(2k)!}\leq e^{\delta^{2}(a-b)^{2}}\leq e^{\delta^{2}h^{2}},\label{eq:pf-prop-truncation-residual-1}
	\end{equation}
	where $h=|\Log_{O}P-\Log_{O}Q|$ with $|\cdot|$ denotes the norm of tangent vectors. The last inequality is due to the
	fact that $a=|\Log_{O}P|$, $b=|\Log_{O}Q|$. 
	
	For the second term $I_{2}$, we first have $\sinh(\delta a)\sinh(\delta b)\leq\delta^{2}abe^{\delta^{2}a^{2}+\delta^{2}b^{2}}$.
	Also, the Euclidean law of cosine implies that $h^{2}=a^{2}+b^{2}-2ab\cos\gamma$
	and $2ab(1-\cos\gamma)=h^{2}-(a-b)^{2}\leq h^{2}$. Therefore, 
	\[
	I_{2}\leq\delta^{2}h^{2}e^{\delta^{2}a^{2}+\delta^{2}b^{2}}/2\leq\delta^{2}h^{2}e^{\delta^{2}h^{2}+2\delta^{2}ab\cos\gamma}/2\leq\delta^{2}h^{2}e^{\delta^{2}h^{2}}e^{2D^{2}\delta^{2}}/2,
	\]
	where $D$ is the diameter of $\mathcal{K}$. Combining with (\ref{eq:pf-prop-truncation-residual-1}),
	\begin{align*}
	I_{1}+I_{2} & \leq e^{\delta^{2}h^{2}}+\delta^{2}h^{2}e^{\delta^{2}h^{2}}e^{2D^{2}\delta^{2}}/2\leq e^{\delta^{2}h^{2}}(1+A\delta^{2}h^{2})/2 \\
	& \leq e^{\delta^{2}h^{2}+A\delta^{2}h^{2}}=e^{(1+A)\delta^{2}h^{2}}=e^{Bh^{2}},
	\end{align*}
	where $A=e^{2D^{2}\delta^{2}}$ and $B=(1+A)\delta^{2}$ constants
	uniform for $t$. Therefore 
	\[
	\delta c=\cosh^{-1}(I_{1}+I_{2})=\log(I_{1}+I_{2}+\sqrt{(I_{1}+I_{2})^{2}-1})\leq\log(e^{Bh^{2}}+\sqrt{e^{2Bh^{2}}-1}).
	\]
	It can be shown that $\sqrt{e^{2Bx}-1}\leq e^{Bx}\sqrt{2Bx}$. Thus,
	\begin{align*}
	\delta c & \leq\log(e^{Bh^{2}}(1+\sqrt{2B}h))\leq\log(e^{Bh^{2}+\sqrt{2B}h})\\
	& \leq Bh^{2}+\sqrt{2B}h\leq2BDh+\sqrt{2B}h=\sqrt{C}\delta h,
	\end{align*}
	where $C=\{(2BD+\sqrt{2B})/\delta\}^{2}$, or in other words, $d_{\manifold}(P,Q)\leq\sqrt{C}|\Log_{O}P-\Log_{O}Q|$.
\end{proof}
\begin{proof}[Proof of Proposition \ref{prop:fpca-frame}]
	Part \ref{enu:prop:fpca-frame-1} follows from a simple calculation.
	To lighten notations, let $\mathbf{f}^{T}\vec E$ denote $\mathbf{f}^{T}(\cdot)\vec E(\mu(\cdot))$
	for a $\real^{\dim}$ valued function defined on $\tdomain$. Suppose
	$\phi_{\vec E,k}$ is the coordinate of $\phi_{k}$ under $\vec E$.
	Because $$(\covarop_{\vec E}\phi_{\vec E,k})^{T}\vec E=\expect\metric{Z_{\vec E}}{\phi_{\vec E,k}}Z_{\vec E}\vec E=\expect\vfinnerprod{\Log_{\mu}X,\phi_{k}}\Log_{\mu}X=\lambda_{k}\phi_{k}=\lambda_{k}\phi_{\vec E,k}^{T}\vec E,$$
	one concludes that $\covarop_{\vec E}\phi_{\vec E,k}=\lambda_{k}\phi_{\vec E,k}$
	and hence $\phi_{\vec E,k}$ is an eigenfunction of $\covarop_{\vec E}$
	corresponding to the eigenvalue $\lambda_{k}$. Other results in Part
	\ref{enu:prop:fpca-frame-2} and \ref{enu:prop:fpca-frame-3} have
	been derived in Section \ref{sec:Intrinsic-Representation}. The continuity
	of $X$ and $\vec E$, in conjunction with $\expect\vfnorm[][2]{\Log_{\mu}X}<\infty$,
	implies that $Z_{\vec E}$ is a mean square continuous random process
	and the joint measurability of $X$ passes to $Z_{\vec E}$. Then
	$Z_{\vec E}$ can be regarded a random element of the Hilbert space
	$\ltwo(\tdomain,\mathscr{B}(\tdomain),\upsilon)$ that is isomorphic
	to $\mathscr{T}(\mu)$. Also, the isomorphism maps $Z_{\vec E}$
	to $X$ for each $\omega$ in the sample space. Then, Part \ref{enu:prop:fpca-frame-4}
	follows from Theorem 7.4.3 of \citet{Hsing2015}.
\end{proof}
\begin{proof}[Proof of Theorem \ref{thm:property-mean-function}]
	The strong consistency stated in Part \ref{enu:thm:mean-curve-2} is an immediate consequence of Lemma \ref{lem:neighbor-tube-mu-hat}. For Part \ref{enu:thm:mean-curve-1}, 
	to prove continuity of $\mu$, fix $t\in\tdomain$. Let $\mathcal{K}\supset\mathcal{U}$
	be compact. By \ref{cond:uniform-second-moment}, $c\define\sup_{p\in\mathcal{K}}\sup_{s\in\tdomain}\expect d_{\manifold}^{2}(p,X(s))<\infty$.
	Thus,
	\begin{align*}
	|F(\mu(t),s)-F(\mu(s),s)| & \leq|F(\mu(t),t)-F(\mu(s),s)|+|F(\mu(t),s)-F(\mu(t),t)|\\
	& \leq\sup_{p\in\mathcal{K}}|F(p,t)-F(p,s)|+2c\expect d_{\manifold}(X(s),X(t)) \\
	& \leq4c\expect d_{\manifold}(X(s),X(t)).
	\end{align*}
	The {continuity assumption of sample paths} implies
	$\expect d_{\manifold}(X(s),X(t))\rightarrow0$ as $s\rightarrow t$.
	Then by {condition \ref{cond:separability}}, $d_{\manifold}(\mu(t),\mu(s))\rightarrow0$
	as $s\rightarrow t$, and the the continuity of $\mu$ follows. The
	uniform continuity follows from the compactness of $\tdomain$. Given Lemma
	\ref{lem:sup-sup-F-n-F} and \ref{lem:neighbor-tube-mu-hat}, the
	a.s. continuity of $\hat{\mu}$ can be derived in a similar way. The first statement of Part \ref{enu:thm:mean-curve-4} is a corollary of Part \ref{enu:thm:mean-curve-3}, while the second statement follows from the first one and the compactness of $\tdomain$. It remains to show Part \ref{enu:thm:mean-curve-3} in order to conclude the proof, as follows.
	
	Let $V_{t,i}(p)=\Log_{p}X_{i}(t)$ and $\gamma_{t,p}$
	be the minimizing geodesic from $\mu(t)$ to $p\in\manifold$ at unit
	time. The first-order Taylor series expansion at $\mu(t)$ yields
	\begin{alignat}{1}
	\mathcal{P}_{\hat{\mu}(t),\mu(t)}\sum_{i=1}^{n}V_{t,i}(\hat{\mu}(t)) & =\sum_{i=1}^{n}V_{t,i}(\mu(t))+\sum_{i=1}^{n}\nabla_{\gamma_{t,\hat{\mu}(t)}^{\prime}(0)}V_{t,i}(\mu(t))+\Delta_{t}(\hat{\mu}(t))\gamma_{t,\hat{\mu}(t)}^{\prime}(0)\nonumber \\
	& =\sum_{i=1}^{n}V_{t,i}(\mu(t))-\sum_{i=1}^{n}H_{t}(\mu(t))\gamma_{t,\hat{\mu}(t)}^{\prime}(0)+\Delta_{t}(\hat{\mu}(t))\gamma_{t,\hat{\mu}(t)}^{\prime}(0),\label{eq:mean-pf-0}
	\end{alignat}
	where an expression for $\Delta_t$ is provided in the proof of Lemma \ref{lem:delta-operator}. 
	
	
	Since $\sum_{i=1}^{n}V_{t,i}(\hat{\mu}(t))=\sum_{i=1}^{n}\Log_{\hat{\mu}(t)}X_{i}(t)=0$,
	we deduce from (\ref{eq:mean-pf-0}) that
	\begin{equation}
	\frac{1}{n}\sum_{i=1}^{n}\Log_{\mu(t)}X_{i}(t)-\left(\frac{1}{n}\sum_{i=1}^{n}H_{t,i}(\mu(t))-\frac{1}{n}\Delta_{t}(\hat{\mu}(t))\right)\Log_{\mu(t)}\hat{\mu}(t)=0.\label{eq:mean-pf-6}
	\end{equation}
	By LLN, $\frac{1}{n}\sum_{i=1}^{n}H_{t,i}(\mu(t))\rightarrow\expect H_{t}(\mu(t))$
	in probability, while $\expect H_{t}(\mu(t))$ is invertible for all
	$t$ by condition \ref{cond:isolation-minima}. In light of Lemma \ref{lem:delta-operator},
	this result suggests that with probability tending to one, for all $t\in\tdomain$, $\frac{1}{n}\sum_{i=1}^{n}H_{t,i}(\mu(t))-\frac{1}{n}\Delta_{t}(\hat{\mu}(t))$
	is invertible, and also 
	\[
	\left(\frac{1}{n}\sum_{i=1}^{n}H_{t,i}(\mu(t))-\frac{1}{n}\Delta_{t}(\hat{\mu}(t))\right)^{-1}=\{\expect H_{t}(\mu(t))\}^{-1}+\op(1),
	\]
	and according to (\ref{eq:mean-pf-6}), 
	\[
	\Log_{\mu(t)}\hat{\mu}(t)=\{\expect H_{t}(\mu(t))\}^{-1}\left(\frac{1}{n}\sum_{i=1}^{n}\Log_{\mu(t)}X_{i}(t)\right)+\op(1),
	\] 
	where the $\op(1)$ terms do not depend on $t$. Given this, we can now conclude the proof of Part \ref{enu:thm:mean-curve-3} by applying
	a central limit theorem in Hilbert spaces \citep{Aldous1976} to establish that the process 
	$\frac{1}{\sqrt{n}}\sum_{i=1}^{n}\{\expect H_{t}(\mu(t))\}^{-1}\Log_{\mu(t)}X_{i}(t)$ converges to a Gaussian
	measure on tensor Hilbert space $\mathscr{T}(\mu)$ with \hyphenation{covariance} \hyphenation{operator} $\mathcal{C}(\cdot)=\expect(\vfinnerprod[\mu]{V,\,\cdot\,}V)$
	for a random element $V(t)=\{\expect H_{t}(\mu(t))\}^{-1}\Log_{\mu(t)}X(t)$ in the tensor Hilbert space $\mathscr{T}(\mu)$.
\end{proof}
\begin{proof}[Proof of Theorem \ref{thm:explict-bound-eigenfunction}]
	Note that
	\begin{align*}
	\Phi\hat{\covarop}-\covarop= & n^{-1}\sum(\Gamma\Log_{\hat{\mu}}X_{i})\otimes(\Gamma\Log_{\hat{\mu}}X_{i})-\covarop\\
	= & n^{-1}\sum(\Log_{\mu}X_{i})\otimes(\Log_{\mu}X_{i})-\covarop\\
	& +n^{-1}\sum(\Gamma\Log_{\hat{\mu}}X_{i}-\Log_{\mu}X_{i})\otimes(\Log_{\mu}X_{i})\\
	& +n^{-1}\sum(\Log_{\mu}X_{i})\otimes(\Gamma\Log_{\hat{\mu}}X_{i}-\Log_{\mu}X_{i})\\
	& +n^{-1}\sum(\Gamma\Log_{\hat{\mu}}X_{i}-\Log_{\mu}X_{i})\otimes(\Gamma\Log_{\hat{\mu}}X_{i}-\Log_{\mu}X_{i})\\
	\equiv & A_{1}+A_{2}+A_{3}+A_{4}.
	\end{align*}
	For $A_{2}$, it is seen that 
	\begin{align*}
	\opnorm[HS][2]{A_{2}} & \leq\mathrm{const.}\frac{1}{n^{2}}\sum_{i=1}^{n}\sum_{j=1}^{n}(\vfnorm[][2]{\Log_{\mu}X_{i}}+\vfnorm[][2]{\Log_{\mu}X_{j}})\times \\
	& (\vfnorm[][2]{\Gamma\Log_{\hat{\mu}}X_{i}-\Log_{\mu}X_{i}}+\vfnorm[][2]{\Gamma\Log_{\hat{\mu}}X_{j}-\Log_{\mu}X_{j}}).
	\end{align*}

	With smoothness of $d_{\manifold}^2$, continuity of $\mu$ and compactness of $\tdomain$, one can show that $\sup_{t\in\tdomain}\|H_t(\mu(t))\|<\infty$. By the uniform consistency of $\hat{\mu}$, with the same Taylor series expansion in \eqref{eq:mean-pf-0} and the technique in the proof of Lemma \ref{lem:delta-operator}, it can be established that $n^{-1}\sum_{i=1}^n\vfnorm[][2]{\Gamma\Log_{\hat{\mu}}X_{i}-\Log_{\mu}X_{i}}\vfnorm[][2]{\Log_{\mu}X_{i}}\leq\mathrm{const.}(1+\op(1))\sup_{t\in\tdomain}d_{\manifold}^{2}(\hat{\mu}(t),\mu(t))$. Also note that by LLN, $n^{-1}\sum_{j=1}^n\vfnorm[][2]{\Log_{\mu}X_{j}}=\Op(1)$.
	Then, with Part \ref{enu:thm:mean-curve-4} of Theorem \ref{thm:property-mean-function},
	$$\opnorm[HS][2]{A_{2}}\leq \mathrm{const.}\{4+\op(1)+\Op(1)\}\sup_{t\in\tdomain}d_{\manifold}^{2}(\hat{\mu}(t),\mu(t))=\Op(1/n).$$
	Similar calculation shows that $\opnorm[HS][2]{A_{3}}=\Op(1/n)$
	and $\opnorm[HS][2]{A_{4}}=\Op(1/n^{2})$. Now, by \citet{Dauxois1982},
	$\vfnorm[HS][2]{n^{-1}\sum(\Log_{\mu}X_{i})\otimes(\Log_{\mu}X_{i})-\covarop}=\Op(1/n)$.
	Thus, $\vfnorm[HS][2]{\Phi\hat{\covarop}-\covarop}=\Op(1/n)$. According to Part 1 \& 5 of Proposition 2, $\hat{\lambda}_k$ are also eigenvalues of $\Phi\hat{\covarop}$.
	The results for $\hat{\lambda}_k$ and $(J,\delta_{j})$ follow from \citet{Bosq2000}.
	Those for $(\hat{J},\hat{\delta}_{j})$ are due to $\sup_{k\geq1}|\hat{\lambda}_{k}-\lambda_{k}|\leq\opnorm[HS]{\hat{\covarop}\ominus_{\Phi}\covarop}$.
\end{proof}

\begin{proof}[Proof of Theorem \ref{thm:fpcr-estimator}]
	In this proof, both $\op(\cdot)$ and $\Op(\cdot)$ are understood
	to be uniform for the class $\mathcal{F}$. Let $\check{\beta}=\Exp_{\mu}\sum_{k=1}^{K}\hat{b}_{k}\Gamma\hat{\phi}_{k}$.
	Then
	\begin{align*}
	d_{\manifold}^{2}(\hat{\beta},\beta) & \leq2d_{\manifold}^{2}(\hat{\beta},\check{\beta})+2d_{\manifold}^{2}(\check{\beta},\beta).
	\end{align*}
	The first term is of order $O_{p}(1/n)$ uniform for the class
	$\mathcal{F}$, according to a technique similar to the one in the proof of Lemma \ref{lem:delta-operator}, as well as  Theorem \ref{thm:property-mean-function} (note that the results
	in Theorem \ref{thm:property-mean-function} are uniform for the class
	$\mathcal{F}$). Then the convergence rate is established if one can
	show that $$d_{\manifold}^{2}(\check{\beta},\beta)=\Op\left(n^{-\frac{2\varrho-1}{2\varrho+4\alpha+2}}\right),$$
	which follows from 
	\begin{align}
	\vfnorm[\mu][2]{\sum_{k=1}^{K}\hat{b}_{k}\Gamma\hat{\phi}_{k}-\sum_{k=1}^{\infty}b_{k}\phi_{k}} & =\Op\left(n^{-\frac{2\varrho-1}{2\varrho+4\alpha+2}}\right)\label{eq:pf-fpca-estimator-1}
	\end{align}
	and Proposition \ref{prop:truncation-residual}. It remains to show
	(\ref{eq:pf-fpca-estimator-1}).
	
	We first observe that because $b_{k}\leq Ck^{-\varrho}$,
	\begin{align}
	\vfnorm[\mu][2]{\sum_{k=1}^{K}\hat{b}_{k}\Gamma\hat{\phi}_{k}-\sum_{k=1}^{\infty}b_{k}\phi_{k}} & \leq2\vfnorm[\mu][2]{\sum_{k=1}^{K}\hat{b}_{k}\Gamma\hat{\phi}_{k}-\sum_{k=1}^{K}b_{k}\phi_{k}}+O(K^{-2\varrho+1}).\label{eq:pf-fpca-estimator-2}
	\end{align}
	Define 
	$$A_{1} =\sum_{k=1}^{K}(\hat{b}_{k}-b_{k})\phi_{k},\,\,\,\,\,\,\,\,\,
	A_{2} =\sum_{k=1}^{K}b_{k}(\Gamma\hat{\phi}_{k}-\phi_{k}),\,\,\,\,\,\,\,\,\,
	A_{3} =\sum_{k=1}^{K}(\hat{b}_{k}-b_{k})(\Gamma\hat{\phi}_{k}-\phi_{k}).$$
	Then $$\vfnorm[\mu][2]{\sum_{k=1}^{K}\hat{b}_{k}\Gamma\hat{\phi}_{k}-\sum_{k=1}^{K}b_{k}\phi_{k}}\leq2\vfnorm[][2]{A_{1}}+2\vfnorm[][2]{A_{2}}+2\vfnorm[][2]{A_{3}}.$$
	It is clear that the term $A_{3}$ is asymptotically dominated by
	$A_{1}$ and $A_{2}$. Note that the compactness of $X$ in condition \ref{cond:X-moment-flr}  implies $\expect\vfnorm[][4]{\Log_{\mu}X}<\infty$. Then, by Theorem \ref{thm:explict-bound-eigenfunction}, for $A_{2}$, we have the bound
	\begin{align*}
	\vfnorm[][2]{A_{2}} & \leq2\sum_{k=1}^{K}b_{k}^{2}\vfnorm[][2]{\Gamma\hat{\phi}_{k}-\phi_{k}}=\begin{cases}
	\Op\left(\frac{1}{n}K^{-2\varrho+2\alpha+3}\right) & \text{if }\alpha>\varrho-3/2,\\
	\Op\left(\frac{1}{n}\log K\right) & \text{if }\alpha=\varrho-3/2,\\
	\Op\left(\frac{1}{n}\right) & \text{if \ensuremath{\alpha}<\ensuremath{\varrho}-3/2}.
	\end{cases}
	\end{align*}
	It is easy to verify that, because $K\asympeq n^{1/(4\alpha+2\varrho+2)}$,
	$\vfnorm[][2]{A_{2}}$ is asymptotically dominated by $K^{-2\varrho+1}$.
	Thus, 
	\[
	\vfnorm[][2]{A_{2}}=\Op\left(n^{-\frac{2\varrho-1}{2\varrho+4\alpha+2}}\right).
	\]
	
	Now we focus on the term $A_{1}$. Note that $\hat{a}_{k}=\vfinnerprod[\hat{\mu}]{\hat{\chi},\hat{\phi}_{k}}=\vfinnerprod[\mu]{\Gamma\hat{\chi},\Gamma\hat{\phi}_{k}}$.
	Then
	\begin{align*}
	\vfnorm[][2]{A_{1}} & =\sum_{k=1}^{K}(\hat{b}_{k}-b_{k})^{2}=\sum_{k=1}^{K}\left(\hat{\lambda}_{k}^{-1}\vfinnerprod[\mu]{\Gamma\hat{\chi},\Gamma\hat{\phi}_{k}}-b_{k}\right)^{2}.
	\end{align*}
	With Theorem
	\ref{thm:property-mean-function} and CLT, it can be shown that 
	$\vfinnerprod[\mu]{\Gamma\hat{\chi},\Gamma\hat{\phi}_{k}}=\lambda_{k}b_{k}+\Op\left(k^{\alpha+1}n^{-1/2}\right).$
	Note that $\sup_{k}|\hat{\lambda}_{k}-\lambda_{k}|^{2}=\Op(1/n)$
	which implies that $\hat{\lambda}_{k}>\lambda_{k}-\Op(1/\sqrt n)$
	uniformly for all $k$. In conjunction with the choice of $K\asympeq n^{1/(4\alpha+2\varrho+2)}$,
	one can conclude that
	\begin{align*}
	\vfnorm[][2]{A_{1}} & =\sum_{k=1}^{K}\left(\frac{\lambda_{k}-\hat{\lambda}_{k}}{\hat{\lambda}_{k}}b_{k}\right)^{2}+\Op\left(\frac{1}{n}\right)\sum_{k=1}^{K}k^{2\alpha+2}\hat{\lambda}_{k}^{-2} \\ & \leq\Op\left(\frac{1}{n}\right)\sum_{k=1}^{K}\lambda_{k}^{-2}b_{k}^{2}+\Op\left(\frac{1}{n}\right)\sum_{k=1}^{K}k^{2\alpha+2}\lambda_{k}^{-2}\\
	& =\Op\left(\frac{1}{n}\right)\sum_{k=1}^{K}k^{2\alpha+2}\lambda_{k}^{-2}=\Op\left(\frac{K^{4\alpha+3}}{n}\right) =\Op\left(n^{-\frac{2\varrho-1}{2\varrho+4\alpha+1}}\right).
	\end{align*}
	Thus, 
	\[
	\vfnorm[\mu][2]{\sum_{k=1}^{K}\hat{b}_{k}\Gamma\hat{\phi}_{k}-\sum_{k=1}^{K}b_{k}\phi_{k}}=\Op\left(n^{-\frac{2\varrho-1}{2\varrho+4\alpha+2}}\right).
	\]
	Finally, (\ref{eq:pf-fpca-estimator-1}) follows from the above equation,
	the equation (\ref{eq:pf-fpca-estimator-2}) and $
	K^{-2\varrho+1}=O\left(n^{-\frac{2\varrho-1}{2\varrho+4\alpha+2}}\right)$.
\end{proof}

\begin{proof}[Proof of Theorem \ref{thm:tikhonov-estimator-flr}]
	In this proof, both $\op(\cdot)$ and $\Op(\cdot)$ are understood
	to be uniform for the class $\mathcal{G}$. First, $d_{\manifold}^{2}(\hat{\beta},\beta)\leq2d_{\manifold}^{2}(\hat{\beta},\Exp_{\mu}\Gamma(\hat{\covarop}^{+}\hat{\chi}))+2d_{\manifold}^{2}(\Exp_{\mu}\Gamma(\hat{\covarop}^{+}\hat{\chi}),\beta)$.
	The first term is of order $O_{p}(1/n)$ uniform for the class
	$\mathcal{G}$, according to a technique similar to the one in the proof of Lemma \ref{lem:delta-operator}, as well as Theorem \ref{thm:property-mean-function} (note that the results
	in Theorem \ref{thm:property-mean-function} are uniform for the class
	$\mathcal{G}$). For the second term, using Proposition \ref{prop:property-Gamma-operator},
	one can show that $\Phi\hat{\covarop}^{+}=(\Phi\hat{\covarop}+\rho\Phi I_{\hat{\mu}})^{-1}=(\Phi\hat{\covarop}+\rho I_{\mu})^{-1}$,
	where $I_{\hat{\mu}}$ and $I_{\mu}$ denote the identity operators
	on $\mathscr{T}(\hat{\mu})$ and $\mathscr{T}(\mu)$, respectively.
	It remains to show $d_{\manifold}^{2}(\Exp_{\mu}\Gamma(\hat{\covarop}^{+}\hat{\chi}),\beta)=\Op(n^{-(2\varrho-\alpha)/(2\varrho+\alpha)})$.
	
	Define $\covarop_{\rho}^{+}=(\covarop+\rho I_{\mu})^{-1}$, $\chi_{n}=n^{-1}\sum_{i=1}^{n}(Y_{i}-\bar{Y})\Log_{\mu}X_{i}$,
	and 
	\begin{align*}
	A_{n1} & =\covarop_{\rho}^{+}\chi_{n}+\covarop_{\rho}^{+}(\hat{\chi}-\chi_{n})\equiv A_{n11}+A_{n12},\\
	A_{n2} & =(\Phi\hat{\covarop}^{+}-\covarop_{\rho}^{+})\chi_{n}+(\Phi\hat{\covarop}^{+}-\covarop_{\rho}^{+})(\hat{\chi}_{n}-\chi_{n})\equiv A_{n21}+A_{n22}.
	\end{align*}
	It is easy to see that the dominant term is $A_{n11}$ and $A_{n21}$
	for $A_{n1}$ and $A_{n2}$, respectively. With $\rho=n^{-\alpha/(2\varrho+\alpha)}$,
	it has been shown in \citet{Hall2007c} that $\expect\vfnorm[\mu][2]{A_{n1}-\Log_{\mu}\beta}=O(n^{-(2\varrho-1)/(2\varrho+\alpha)})$. 
	Denote $\Delta=\Phi\hat{\covarop}-\covarop_{\rho}^{+}$. Then 
	\begin{align*}
	A_{n21} & =(\Phi\hat{\covarop}^{+}-\covarop_{\rho}^{+})\chi_{n} =-(I_{\mu}+\covarop_{\rho}^{+}\Delta)^{-1}\covarop_{\rho}^{+}\Delta\covarop_{\rho}^{+}\chi_{n}\\
	& =-(I_{\mu}+\covarop_{\rho}^{+}\Delta)^{-1}\covarop_{\rho}^{+}\Delta\Log_{\mu}\beta-(I_{\mu}+\covarop_{n}^{+}\Delta)^{-1}\covarop_{\rho}^{+}\Delta(\covarop_{\rho}^{+}\chi_{n}-\Log_{\mu}\beta)\\
	& \equiv A_{n211}+A_{n212}.
	\end{align*}
	By Theorem \ref{thm:explict-bound-eigenfunction}, $\opnorm[\mu]{\Delta}=\Op(1/n)$.
	Also, one can see that $\opnorm[\mu]{(I_{\mu}+\covarop_{\rho}^{+}\Delta)^{-1}}=\Op(1)$,
	with the assumption that {$\rho^{-1}/n=o(1)$}.
	Also, $\opnorm[op]{(I_{\mu}+\covarop_{\rho}^{+}\Delta)^{-1}\covarop_{\rho}^{+}\Delta}=\Op(\rho^{-2}/n)$.
	Using the similar technique in \citet{Hall2005}, we can show that
	$\vfnorm[\mu][2]{\covarop_{\rho}^{+}\chi_{n}-\Log_{\mu}\beta}=\Op(n^{-(2\varrho-1)/(2\varrho+\alpha)})$,
	and hence conclude that $\vfnorm[\mu][2]{A_{n212}}=\Op(n^{-(2\varrho-1)/(2\varrho+\alpha)})$.
	For $A_{n211}$, 
	\begin{align*}
	\vfnorm[\mu][2]{A_{n211}} & =\vfnorm[\mu][2]{(I_{\mu}+\covarop_{n}^{+}\Delta)^{-1}\covarop_{n}^{+}\Delta\Log_{\mu}\beta}\\ & \leq\opnorm[op][2]{(I_{\mu}+\covarop_{n}^{+}\Delta)^{-1}}\opnorm[op][2]{\covarop_{n}^{+}\Delta}\vfnorm[\mu][2]{\Log_{\mu}\beta} =\Op(n^{-(2\varrho-\alpha)/(2\varrho+\alpha)}).
	\end{align*}
	Combining all results above, we deduce that $\vfnorm[\mu][2]{\Gamma(\hat{\covarop}^{+}\hat{\chi})-\Log_{\mu}\beta}=\Op(n^{-(2\varrho-\alpha)/(2\varrho+\alpha)})$
	and thus $$d_{\manifold}^{2}(\Exp_{\mu}\Gamma(\hat{\covarop}^{+}\hat{\chi}),\beta)=\Op(n^{-(2\varrho-\alpha)/(2\varrho+\alpha)}),$$
	according to condition \ref{cond:X-moment-flr} and Proposition \ref{prop:truncation-residual}.
\end{proof}

\section{Ancillary Lemmas}

\begin{lem}\label{lem:delta-operator}$\sup_{t\in\tdomain}n^{-1}\|\Delta_{t}(\hat{\mu}(t))\|=\op(1)$, where $\Delta_t$ is as in \eqref{eq:mean-pf-0}.
\end{lem}
\begin{proof}
	With the continuity of $\mu$ and compactness of $\tdomain$, the
	existence of local smooth orthnormal frames (e.g., Proposition 11.17
	of \citet{Lee2002}) suggests that we can find a finite open cover
	$\tdomain_{1},\ldots,\tdomain_{m}$ for $\tdomain$ such that there
	exists a smooth orthonormal frame $b_{j,1},\ldots,b_{j,d}$ for the
	$j$th piece $\{\mu(t):t\in\mathrm{cl}(\tdomain_{j})\}$ of $\mu$,
	where $\mathrm{cl}(A)$ denotes topological closure of a set $A$. For fixed $t\in\tdomain_{j}$, by mean value theorem, it can be shown that
	\begin{equation}
	\Delta_{t}(\hat{\mu}(t))U=\sum_{r=1}^{d}\sum_{i=1}^{n}\left(\mathcal{P}_{\gamma_{t,\hat{\mu}(t)}(\theta_{t}^{r,j}),\mu(t)}\nabla_{U}W_{t,i}^{r,j}(\gamma_{t,\hat{\mu}(t)}(\theta_{t}^{r,j}))-\nabla_{U}W_{t,i}^{r,j}(\mu(t))\right)\label{eq:mean-pf-1}
	\end{equation}
	for $\theta_t^{r,j}\in [0,1]$ and $W_{t,i}^{r,j}=\metric{V_{t,i}}{e_{t}^{r,j}}e_{t}^{r,j}$, where $e_{t}^{1,j},\ldots,e_{t}^{d,j}$
	is the orthonormal frame extended by parallel transport of $b_{j,1}(\mu(t)),\ldots,b_{j,d}(\mu(t))$
	along minimizing geodesic. 
	
	Take $\epsilon=\epsilon_{n}\downarrow0$ as $n\rightarrow\infty$.
	For each $j$, by the same argument of Lemma 3 of \cite{Kendall2011}, together with
	continuity of $\mu$ and the continuity of the frame $b_{j,1},\ldots,b_{j,d}$, we can
	choose a continuous positive $\rho_{t}^{j}$ such that, $\hat{\mu}(t)\in B(\mu(t),\rho_{t}^{j})$ and for $p\in B(\mu(t),\rho_{t}^{j})$
	where $B(q,\rho)$ denotes the ball on $\manifold$ centered at $q$
	with radius $\rho$,
	\begin{alignat*}{1}
	& \|\mathcal{P}_{p,\mu(t)}\nabla W_{t,i}^{r,j}(p)-\nabla W_{t,i}^{r,j}(\mu(t))\| \\
	 & \leq(1+2\epsilon\rho_{t}^{j})\sup_{q\in B(\mu(t),\rho_{t}^{j})}\|\mathcal{P}_{q,\mu(t)}\nabla V_{t,i}(q)-\nabla V_{t,i}(\mu(t))\| \\ & +2\epsilon(\|V_{t,i}(\mu(t))\|+\rho_{t}^{j}\|\nabla V_{t,i}(\mu(t))\|).
	\end{alignat*}
	In the above, $p$ plays a role of $\gamma_{t,\hat{\mu}(t)}(\theta_{t}^{r,j})$
	in (\ref{eq:mean-pf-1}). Let $\rho^{j}=\max\{\rho_{t}:t\in\mathrm{cl}(\tdomain_{j})\}$
	and $\rho_{\max}=\max_{j}\rho^{j}$. 
	We then have
	\begin{alignat}{1}
	\sup_{t\in\tdomain}\|\Delta_{t}(\hat{\mu}(t))\| & \leq\max_{j}\sup_{t\in\tdomain_{j}}\|\Delta_{t}(\hat{\mu}(t))\| \nonumber\\
	& =\sum_{r=1}^d\sum_{i=1}^{n}\max_{j}\sup_{t\in\tdomain_{j}}\|\mathcal{P}_{\gamma_{t,\hat{\mu}(t)}(\theta_t^{r,j}),\mu(t)}\nabla_{U}W_{r,i}^{t,j}(\gamma_{t,\hat{\mu}(t)}(\theta_t^{r,j}))-\nabla_{U}W_{r,i}^{t,j}(\mu(t))\|\nonumber\\
	& \leq d(1+2\epsilon\rho_{\max})m\sum_{i=1}^{n}\sup_{t\in\tdomain}\sup_{q\in B(\mu(t),\rho_{\max})}\|\mathcal{P}_{q,\mu(t)}\nabla V_{t,i}(q)-\nabla V_{t,i}(\mu(t))\|\label{eq:summand-1}\\
	& +2d\epsilon\sum_{i=1}^{n}\sup_{t\in\tdomain}\|V_{t,i}(\mu(t))\|\label{eq:summand-2} +2d\epsilon\rho_{\max}\sum_{i=1}^{n}\sup_{t\in\tdomain}\|\nabla V_{t,i}(\mu(t))\|.
	\end{alignat}
	For \eqref{eq:summand-1}, the Lipschitz condition of \ref{cond:Lipschitz-moments}
	and smoothness of $d_{\manifold}$ imply that 
	\begin{alignat*}{1}
	& \lim_{\rho\downarrow0}\expect\sup_{q\in B(\mu(t),\rho)}\|\mathcal{P}_{q,\mu(t)}\nabla V_{t,i}(q)-\nabla V_{t,i}(\mu(t))\| \\
	& =\lim_{\rho\downarrow0}\expect\sup_{q\in B(\mu(t),\rho)}\|\mathcal{P}_{q,\mu(t)}H_{t}(q)-H_{t}(\mu(t))\|=0.
	\end{alignat*}
	As $\sup_{t\in\tdomain}d_{\manifold}(\hat{\mu}(t),\mu(t))=o_{a.s.}(1)$,
	$\rho_{\max}$ could be chosen so that $\rho_{\max}\downarrow0$ as
	$n\rightarrow\infty$. Thus, with probability tending to one, by Markov
	inequality, we deduce that
	\begin{equation}
	(1+2\epsilon\rho_{\max})m\frac{1}{n}\sum_{i=1}^{n}\sup_{t\in\tdomain_{t}}\sup_{q\in B(\mu(t),\rho_{\max})}\|\mathcal{P}_{q,\mu(t)}\nabla V_{t,i}(q)-\nabla V_{t,i}(\mu(t))\|=\op(1).\label{eq:mean-pf-2}
	\end{equation}
	For the first term  in \eqref{eq:summand-2}, LLN shows that 
	\begin{alignat*}{1}
	\frac{1}{n}\sum_{i=1}^{n}\sup_{t\in\tdomain}\|V_{t,i}(\mu(t))\| & \overset{P}{\rightarrow}\expect\sup_{t\in\tdomain}\|V_{t,i}(\mu(t))\|=\expect\sup_{t\in\tdomain}d_{\manifold}(X(t),\mu(t))<\infty,
	\end{alignat*}
	or 
	\begin{equation}
	\frac{1}{n}\sum_{i=1}^{n}\sup_{t\in\tdomain}\|V_{t,i}(\mu(t))\|=\Op(1).\label{eq:mean-pf-3}
	\end{equation}
	For the second term in \eqref{eq:summand-2}, the compactness of $\tdomain,$ the Lipschitz
	condition of \ref{cond:Lipschitz-moments} and smoothness of $d_{\manifold}$
	also imply that $\expect\sup_{t\in\tdomain}\|\nabla V_{t,i}(\mu(t))\|=\expect\sup_{t\in\tdomain}\|H_{t}(\mu(t))\|<\infty$.
	Consequently, by LLN, 
	\begin{equation}
	\frac{1}{n}\sum_{i=1}^{n}\sup_{t\in\tdomain}\|\nabla V_{t,i}(\mu(t))\|=\Op(1).\label{eq:mean-pf-4}
	\end{equation}
	Combining (\ref{eq:mean-pf-2}), (\ref{eq:mean-pf-3}) and (\ref{eq:mean-pf-4}), with $\epsilon=\epsilon_n\downarrow0$, 
	one concludes that 
	$\sup_{t\in\tdomain}n^{-1}\|\Delta_{t}(p)\|=\op(1)$.
\end{proof}

\begin{lem}
	\label{lem:sup-sup-F-n-F}Suppose {conditions \ref{cond:exist-mean-function}
		and \ref{cond:manifold-property}}-\ref{cond:uniform-second-moment}
	hold. For any compact subset $\mathcal{K}\subset\manifold$, one has
	\[
	\sup_{p\in\mathcal{K}}\sup_{t\in\tdomain}|F_{n}(p,t)-F(p,t)|=\oas(1).
	\]
\end{lem}
\begin{proof}
	By applying the uniform SLLN to $n^{-1}\sum_{i=1}^{n}d_{\manifold}(X_{i}(t),p_{0})$,
	{} for a given $p_{0}\in\mathcal{K}$,
	\begin{align*}
	\sup_{p\in\mathcal{K}}\sup_{t\in\tdomain}\frac{1}{n}\sum_{i=1}^{n}d_{\manifold}(X_{i}(t),p) & \leq\sup_{t\in\tdomain}\frac{1}{n}\sum_{i=1}^{n}d_{\manifold}(X_{i}(t),p_{0})+\sup_{p\in\mathcal{K}}d_{\manifold}(p_{0},p)\\
	& \leq\sup_{t\in\tdomain}\expect d_{\manifold}(X(t),p_{0})+\mathrm{diam}(\mathcal{K})+\oas(1).
	\end{align*}
	Therefore, there exists a set $\Omega_{1}\subset\Omega$ such that
	$\prob(\Omega_{1})=1$, $N_{1}(\omega)<\infty$ and for all $n\geq N_{1}(\omega)$,
	\[
	\sup_{p\in\mathcal{K}}\sup_{t\in\tdomain}\frac{1}{n}\sum_{i=1}^{n}d_{\manifold}(X_{i}(t),p)\leq\sup_{t\in\tdomain}\expect d_{\manifold}(X(t),p_{0})+\mathrm{diam}(\mathcal{K})+1\define c_{1}<\infty,
	\]
	since {$\sup_{t\in\tdomain}\expect d_{\manifold}(X(t),p_{0})<\infty$}
	by condition \ref{cond:uniform-second-moment}. Fix $\epsilon>0$.
	By the inequality $|d_{\manifold}^{2}(x,p)-d_{\manifold}^{2}(x,q)|\leq\{d_{\manifold}(x,p)+d_{\manifold}(x,q)\}d_{\manifold}(p,q)$,
	for all $n\geq N_{1}(\omega)$ and $\omega\in\Omega_{1}$,
	\[
	\sup_{p,q\in\mathcal{K}:d_{\manifold}(p,q)<\delta_{1}}\sup_{t\in\tdomain}|F_{n,\omega}(p,t)-F_{n,\omega}(q,t)|\leq2c_{1}\delta_{1}=\epsilon/3
	\]
	with $\delta_{1}\define\epsilon/(6c_{1})$. Now, let $\delta_{2}>0$
	be chosen such that $\sup_{t\in\tdomain}|F(p,t)-F(q,t)|<\epsilon/3$ if $p,q\in\mathcal{K}$
	and $d_{\manifold}(p,q)<\delta_{2}$. %
	{} Suppose $\{p_{1},\ldots,p_{r}\}\subset\mathcal{K}$ is a $\delta$-net
	in $\mathcal{K}$ %
	{} with $\delta\define\min\{\delta_{1},\delta_{2}\}$. Applying uniform
	SLLN again, there exists a set $\Omega_{2}$ such that $\prob(\Omega_{2})=1$,
	$N_{2}(\omega)<\infty$ for all $\omega\in\Omega_{2}$, and 
	\[
	\max_{j=1,\ldots,r}\sup_{t\in\tdomain}|F_{n,\omega}(p_{j},t)-F(p_{j},t)|<\epsilon/3
	\]
	for all $n\geq N_{2}(\omega)$ with $\omega\in\Omega_{2}$. Then,
	for all $\omega\in\Omega_{1}\cap\Omega_{2}$, for all $n\geq\max\{N_{1}(\omega),N_{2}(\omega)\}$,
	we have
	\begin{align*}
	& \sup_{p\in\mathcal{K}}\sup_{t\in\tdomain}|F_{n,\omega}(p,t)-F(p,t)|\\
	& \leq\sup_{p\in\mathcal{K}}\sup_{t\in\tdomain}|F_{n,\omega}(p)-F_{n,\omega}(u_{p})|+\sup_{p\in\mathcal{K}}\sup_{t\in\tdomain}|F_{n,\omega}(u_{p},t)-F(u_{p},t)|+\sup_{p\in\mathcal{K}}\sup_{t\in\tdomain}|F(u_{p},t)-F(p,t)|\\
	& <\epsilon/3+\epsilon/3+\epsilon/3=\epsilon,
	\end{align*}and this concludes the proof.
\end{proof}
\begin{lem}
	\label{lem:neighbor-tube-mu-hat}{Assume conditions
		\ref{cond:exist-mean-function} and \ref{cond:manifold-property}}-\ref{cond:separability}
	hold. Given any $\epsilon>0$, there exists $\Omega^{\prime}\subset\Omega$
	such that $\prob(\Omega^{\prime})=1$ and for all $\omega\in\Omega^{\prime}$,
	$N(\omega)<\infty$ and for all $n\geq N(\omega)$, $\sup_{t\in\tdomain}d_{\manifold}(\hat{\mu}_{\omega}(t),\mu(t))<\epsilon$.
\end{lem}

\begin{proof}
	Let $c(t)=F(\mu(t),t)=\min\{F(p,t):p\in\manifold\}$ and $\mathcal{N}(t)\define\{p:d_{\manifold}(p,\mu(t))\geq\epsilon\}$.
	It is sufficient to show that there exists $\delta>0$ and $N(\omega)<\infty$
	for all $\omega\in\Omega^{\prime}$, such that for all $n\ge N(\omega)$,
	\begin{equation*}
	\sup_{t\in\tdomain}\{F_{n,\omega}(\mu(t),t)-c(t)\}\leq\delta/2\,\,\,\,\,\,\,\text{and}\,\,\,\,\,\,\,\,\inf_{t\in\tdomain}\{\inf_{p\in\mathcal{N}(t)}F_{n,\omega}(p,t)-c(t)\}\geq\delta.
	\end{equation*}
	This is because the above two inequalities suggest that for all $t\in \tdomain$, $\inf\{F_{n,\omega}(p,t):p\in\manifold\}$
	is not attained at $p$ with $d_{\manifold}(p,\mu(t))\geq\epsilon$,
	and hence $\sup_{t\in \tdomain}d_{\manifold}(\hat{\mu}_{\omega}(t),\mu(t))<\epsilon$.
	
	Let $\mathcal{U}=\{\mu(t):t\in\tdomain\}$. We first show that there
	exists a compact set $\mathcal{A}\supset\mathcal{U}$ and $N_{1}(\omega)<\infty$
	for some $\Omega_{1}\subset\Omega$ such that $\prob(\Omega_{1})=1$,
	and both $F(p,t)$ and $F_{n,\omega}(p,t)$ are greater than $c(t)+1$
	for all $p\in\manifold\backslash\mathcal{A}$ , $t\in \tdomain$ and $n\geq N_{1}(\omega)$.
	This is trivially true when $\mathcal{M}$ is compact, by taking $\mathcal{A}=\manifold$. Now assume
	$\mathcal{M}$ is noncompact. By the inequality $d_{\manifold}(x,q)\geq|d_{\manifold}(q,y)-d_{\manifold}(y,x)|$,
	one has
	\[
	\expect d_{\manifold}^{2}(X(t),q)\geq\expect\{d_{\manifold}^{2}(q,\mu(t))+d_{\manifold}^{2}(X(t),\mu(t))-2d_{\manifold}(q,\mu(t))d_{\manifold}(X(t),\mu(t))\},
	\]
	and by Cauchy--Schwarz inequality,
	\[
	F(q,t)\geq d_{\manifold}^{2}(q,\mu(t))+F(\mu(t),t)-2d_{\manifold}(q,\mu(t))\{F(\mu(t),t)\}^{1/2}.
	\]
	Similarly, 
	\[
	F_{n,\omega}(q,t)\geq d_{\manifold}^{2}(q,\mu(t))+F_{n,\omega}(\mu(t),t)-2d_{\manifold}(q,\mu(t))\{F_{n,\omega}(\mu(t),t)\}^{1/2}.
	\]
	Now, we take $q$ at a sufficiently large distance $\Delta$ from $\mathcal{U}$
	such that $F(q,t)>c(t)+1$ on $\manifold\backslash\mathcal{A}$ for
	all $t$, where $\mathcal{A}\define\overline{\{q:d_{\manifold}(q,\mathcal{U})\leq\Delta\}}$
	(Heine--Borel property yields compactness of $\mathcal{A}$, since
	it is bounded and closed). Since $F_{n,\omega}(\mu(t),t)$ converges
	to $F(\mu(t),t)$ uniformly on $\tdomain$ a.s. by Lemma \ref{lem:sup-sup-F-n-F},
	we can find a set $\Omega_{1}\subset\Omega$ such that $\prob(\Omega_{1})=1$
	and $N_{1}(\omega)<\infty$ for $\omega\in\Omega_{1}$, and $F_{n,\omega}(q,t)>c(t)+1$
	on $\manifold\backslash\mathcal{A}$ for all $t$ and $n\geq N_{1}(\omega)$.
	
	Finally, let $\mathcal{A}_{\epsilon}(t)\define\{p\in\mathcal{A}:d_{\manifold}(p,\mu(t))\geq\epsilon\}$
	and $c_{\epsilon}(t)\define\min\{F(p,t):p\in\mathcal{A}_{\epsilon}\}.$
	Then $\mathcal{A}_{\epsilon}(t)$ is compact and by condition \ref{cond:separability},
	$\inf_{t}\{c_{\epsilon}(t)-c(t)\}>2\delta>0$ for some constant $\delta$.
	By Lemma \ref{lem:sup-sup-F-n-F}, one can find a set $\Omega_{2}\subset\Omega$
	with $\prob(\Omega_{2})=1$ and $N_{2}(\omega)<\infty$ for $\omega\in\Omega_{2}$,
	such that for all $n\geq N_{2}(\omega)$, (i) $\sup_{t}\{F_{n,\omega}(\mu(t),t)-c(t)\}\leq\delta/2$
	and (ii) $\inf_{t}\inf_{p\in\mathcal{A}_{\epsilon}(t)}\{F_{n,\omega}(p,t)-c(t)\}>\delta$.
	Since $\sup_{t}\{F_{n,\omega}(p,t)-c(t)\}>1$ on $\manifold\backslash\mathcal{A}$
	for all $n\geq N_{1}(\omega)$ with $\omega\in\Omega_{1}$, we conclude
	that $\inf_{t}\{F_{n,\omega}(p,t)-c(t)\}>\min\{\delta,1\}$ for all
	$p\in\mathcal{A}_{\epsilon}\cup(\manifold\backslash\mathcal{A})$
	if $n\geq\max\{N_{1}(\omega),N_{2}(\omega)\}$ for $\omega\in\Omega_{1}\cap\Omega_{2}$. The proof is completed by noting that $\Omega_{1}\cap\Omega_{2}$ can serve the $\Omega^\prime$ we are looking for.
\end{proof}

\references


\begin{thebibliography}{57}

\bibitem[\protect\citeauthoryear{Afsari}{2011}]{Afsari2011}
\begin{barticle}[author]
\bauthor{\bsnm{Afsari},~\bfnm{Bijan}\binits{B.}}
(\byear{2011}).
\btitle{Riemannian $L^p$ center of mass: existence, uniqueness, and convexity}.
\bjournal{Proceedings of the American Mathematical Society}
\bvolume{139}
\bpages{655--673}.
\end{barticle}
\endbibitem

\bibitem[\protect\citeauthoryear{Aldous}{1976}]{Aldous1976}
\begin{barticle}[author]
\bauthor{\bsnm{Aldous},~\bfnm{David~J.}\binits{D.~J.}}
(\byear{1976}).
\btitle{A Characterisation of Hilbert Space Using the Central Limit Theorem}.
\bjournal{Journal London Mathematical Society}
\bvolume{14}
\bpages{376--380}.
\end{barticle}
\endbibitem

\bibitem[\protect\citeauthoryear{Arsigny et~al.}{2007}]{Arsigny2007}
\begin{barticle}[author]
\bauthor{\bsnm{Arsigny},~\bfnm{Vincent}\binits{V.}},
  \bauthor{\bsnm{Fillard},~\bfnm{Pierre}\binits{P.}},
  \bauthor{\bsnm{Pennec},~\bfnm{Xavier}\binits{X.}} \AND
  \bauthor{\bsnm{Ayache},~\bfnm{Nicholas}\binits{N.}}
(\byear{2007}).
\btitle{Geometric Means in a Novel Vector Space Structure on Symmetric
  Positive-Definite Matrices}.
\bjournal{SIAM Journal of Matrix Analysis and Applications}
\bvolume{29}
\bpages{328--347}.
\end{barticle}
\endbibitem

\bibitem[\protect\citeauthoryear{Balakrishnan}{1960}]{Balakrishn1960}
\begin{barticle}[author]
\bauthor{\bsnm{Balakrishnan},~\bfnm{A.~V.}\binits{A.~V.}}
(\byear{1960}).
\btitle{Estimation and Detection Theory for Multiple Stochastic Processes}.
\bjournal{Journal of Mathematical Analysis and Applications}
\bvolume{1}
\bpages{386--410}.
\end{barticle}
\endbibitem

\bibitem[\protect\citeauthoryear{Bhattacharya and
  Patrangenaru}{2003}]{Bhattacharya2003}
\begin{barticle}[author]
\bauthor{\bsnm{Bhattacharya},~\bfnm{Rabi}\binits{R.}} \AND
  \bauthor{\bsnm{Patrangenaru},~\bfnm{Vic}\binits{V.}}
(\byear{2003}).
\btitle{Large sample theory of intrinsic and extrinsic sample means on
  manifolds. {I}}.
\bjournal{The Annals of Statistics}
\bvolume{31}
\bpages{1--29}.
\bdoi{10.1214/aos/1046294456}
\end{barticle}
\endbibitem

\bibitem[\protect\citeauthoryear{Binder et~al.}{1997}]{Binder1997}
\begin{barticle}[author]
\bauthor{\bsnm{Binder},~\bfnm{Jeffrey~R.}\binits{J.~R.}},
  \bauthor{\bsnm{Frost},~\bfnm{Julie~A.}\binits{J.~A.}},
  \bauthor{\bsnm{Hammeke},~\bfnm{Thomas~A.}\binits{T.~A.}},
  \bauthor{\bsnm{Cox},~\bfnm{Robert~W.}\binits{R.~W.}},
  \bauthor{\bsnm{Rao},~\bfnm{Stephen~M.}\binits{S.~M.}} \AND
  \bauthor{\bsnm{Prieto},~\bfnm{Thomas}\binits{T.}}
(\byear{1997}).
\btitle{Human Brain Language Areas Identified by Functional Magnetic Resonance
  Imaging}.
\bjournal{Journal of Neuroscience}
\bvolume{17}
\bpages{353--362}.
\end{barticle}
\endbibitem

\bibitem[\protect\citeauthoryear{Bosq}{2000}]{Bosq2000}
\begin{bbook}[author]
\bauthor{\bsnm{Bosq},~\bfnm{Denis}\binits{D.}}
(\byear{2000}).
\btitle{Linear Proceses in Function Spaces}.
\bseries{Lecture Notes in Statistics}.
\bpublisher{Springer}.
\end{bbook}
\endbibitem

\bibitem[\protect\citeauthoryear{Cardot, Ferraty and Sarda}{2003}]{Cardot2003}
\begin{barticle}[author]
\bauthor{\bsnm{Cardot},~\bfnm{Herv\'{e}}\binits{H.}},
  \bauthor{\bsnm{Ferraty},~\bfnm{Fr\'{e}d\'{e}ric}\binits{F.}} \AND
  \bauthor{\bsnm{Sarda},~\bfnm{Pascal}\binits{P.}}
(\byear{2003}).
\btitle{Spline estimators for the functional linear model}.
\bjournal{Statistica Sinica}
\bvolume{13}
\bpages{571--591}.
\end{barticle}
\endbibitem

\bibitem[\protect\citeauthoryear{Cardot, Mas and Sarda}{2007}]{Cardot2007}
\begin{barticle}[author]
\bauthor{\bsnm{Cardot},~\bfnm{Herv\'{e}}\binits{H.}},
  \bauthor{\bsnm{Mas},~\bfnm{Andr\'{e}}\binits{A.}} \AND
  \bauthor{\bsnm{Sarda},~\bfnm{Pascal}\binits{P.}}
(\byear{2007}).
\btitle{{CLT} in Functional Linear Regression Models}.
\bjournal{Probability Theory and Related Fields}
\bvolume{138}
\bpages{325-361}.
\end{barticle}
\endbibitem

\bibitem[\protect\citeauthoryear{Chen and M\"{u}ller}{2012}]{Chen2012}
\begin{barticle}[author]
\bauthor{\bsnm{Chen},~\bfnm{D.}\binits{D.}} \AND
  \bauthor{\bsnm{M\"{u}ller},~\bfnm{H.~G.}\binits{H.~G.}}
(\byear{2012}).
\btitle{Nonlinear manifold representations for functional data}.
\bjournal{The Annals of Statistics}
\bvolume{40}
\bpages{1--29}.
\end{barticle}
\endbibitem

\bibitem[\protect\citeauthoryear{Cheng et~al.}{2016}]{Cheng2016}
\begin{binbook}[author]
\bauthor{\bsnm{Cheng},~\bfnm{Guang}\binits{G.}},
  \bauthor{\bsnm{Ho},~\bfnm{Jeffrey}\binits{J.}},
  \bauthor{\bsnm{Salehian},~\bfnm{Hesamoddin}\binits{H.}} \AND
  \bauthor{\bsnm{Vemuri},~\bfnm{Baba~C.}\binits{B.~C.}}
(\byear{2016}).
\btitle{Recursive Computation of the Fr{\'e}chet Mean on Non-positively Curved
  Riemannian Manifolds with Applications}
In \bbooktitle{Riemannian Computing in Computer Vision}
\bchapter{1},
\bpages{21--43}.
\bpublisher{Springer International Publishing}, \baddress{Cham}.
\end{binbook}
\endbibitem

\bibitem[\protect\citeauthoryear{Cornea et~al.}{2017}]{Cornea2017}
\begin{barticle}[author]
\bauthor{\bsnm{Cornea},~\bfnm{Emil}\binits{E.}},
  \bauthor{\bsnm{Zhu},~\bfnm{Hongtu}\binits{H.}},
  \bauthor{\bsnm{Kim},~\bfnm{Peter}\binits{P.}} \AND
  \bauthor{\bsnm{Ibrahim},~\bfnm{Joseph~G}\binits{J.~G.}}
(\byear{2017}).
\btitle{Regression models on {R}iemannian symmetric spaces}.
\bjournal{Journal of the Royal Statistical Society: Series B (Statistical
  Methodology)}
\bvolume{79}
\bpages{463--482}.
\end{barticle}
\endbibitem

\bibitem[\protect\citeauthoryear{Dai and M{\"{u}}ller}{2017}]{Dai2017}
\begin{barticle}[author]
\bauthor{\bsnm{Dai},~\bfnm{Xiongtao}\binits{X.}} \AND
  \bauthor{\bsnm{M{\"{u}}ller},~\bfnm{Hans-Georg}\binits{H.-G.}}
(\byear{2017}).
\btitle{Principal Component Analysis for Functional Data on {R}iemannian
  Manifolds and Spheres}.
\bjournal{arXiv}.
\end{barticle}
\endbibitem

\bibitem[\protect\citeauthoryear{Dauxois, Pousse and
  Romain}{1982}]{Dauxois1982}
\begin{barticle}[author]
\bauthor{\bsnm{Dauxois},~\bfnm{J.}\binits{J.}},
  \bauthor{\bsnm{Pousse},~\bfnm{A.}\binits{A.}} \AND
  \bauthor{\bsnm{Romain},~\bfnm{Y.}\binits{Y.}}
(\byear{1982}).
\btitle{Asymptotic theory for the principal component analysis of a vector
  random function: some applications to statistical inference}.
\bjournal{Journal of Multivariate Analysis}
\bvolume{12}
\bpages{136--154}.
\end{barticle}
\endbibitem

\bibitem[\protect\citeauthoryear{Dayan and Cohen}{2011}]{Dayan2011}
\begin{barticle}[author]
\bauthor{\bsnm{Dayan},~\bfnm{E.}\binits{E.}} \AND
  \bauthor{\bsnm{Cohen},~\bfnm{L.~G.}\binits{L.~G.}}
(\byear{2011}).
\btitle{Neuroplasticity subserving motor skill learning}.
\bjournal{Neuron}
\bvolume{72}
\bpages{443--454}.
\end{barticle}
\endbibitem

\bibitem[\protect\citeauthoryear{Dryden, Koloydenko and
  Zhou}{2009}]{Dryden2009}
\begin{barticle}[author]
\bauthor{\bsnm{Dryden},~\bfnm{Ian~L.}\binits{I.~L.}},
  \bauthor{\bsnm{Koloydenko},~\bfnm{Alexey}\binits{A.}} \AND
  \bauthor{\bsnm{Zhou},~\bfnm{Diwei}\binits{D.}}
(\byear{2009}).
\btitle{Non-{E}uclidean statistics for covariance matrices, with applications
  to diffusion tensor imaging}.
\bjournal{The Annals of Applied Statistics}
\bvolume{3}
\bpages{1102--1123}.
\end{barticle}
\endbibitem

\bibitem[\protect\citeauthoryear{Essen et~al.}{2013}]{Essen2013}
\begin{barticle}[author]
\bauthor{\bsnm{Essen},~\bfnm{David C.~Van}\binits{D.~C.~V.}},
  \bauthor{\bsnm{Smith},~\bfnm{Stephen~M.}\binits{S.~M.}},
  \bauthor{\bsnm{Barch},~\bfnm{Deanna~M.}\binits{D.~M.}},
  \bauthor{\bsnm{Behrens},~\bfnm{Timothy E.~J.}\binits{T.~E.~J.}},
  \bauthor{\bsnm{Yacoub},~\bfnm{Essa}\binits{E.}},
  \bauthor{\bsnm{Ugurbil},~\bfnm{Kamil}\binits{K.}} \AND
  \bauthor{\bsnm{Consortium},~\bfnm{WU-Minn~HCP}\binits{W.-M.~H.}}
(\byear{2013}).
\btitle{The {WU}-{M}inn Human Connectome Project: An overview}.
\bjournal{NeuroImage}
\bvolume{80}
\bpages{62--79}.
\end{barticle}
\endbibitem

\bibitem[\protect\citeauthoryear{Ferraty and Vieu}{2006}]{Ferraty2006}
\begin{bbook}[author]
\bauthor{\bsnm{Ferraty},~\bfnm{F.}\binits{F.}} \AND
  \bauthor{\bsnm{Vieu},~\bfnm{P.}\binits{P.}}
(\byear{2006}).
\btitle{Nonparametric Functional Data Analysis: Theory and Practice}.
\bpublisher{Springer-Verlag}, \baddress{New York}.
\end{bbook}
\endbibitem

\bibitem[\protect\citeauthoryear{Fletcher and Joshib}{2007}]{Fletcher2007}
\begin{barticle}[author]
\bauthor{\bsnm{Fletcher},~\bfnm{P.~Thomas}\binits{P.~T.}} \AND
  \bauthor{\bsnm{Joshib},~\bfnm{Sarang}\binits{S.}}
(\byear{2007}).
\btitle{{R}iemannian geometry for the statistical analysis of diffusion tensor
  data}.
\bjournal{Signal Processing}
\bvolume{87}
\bpages{250--262}.
\end{barticle}
\endbibitem

\bibitem[\protect\citeauthoryear{Friston}{2011}]{Friston2011}
\begin{barticle}[author]
\bauthor{\bsnm{Friston},~\bfnm{Karl~J.}\binits{K.~J.}}
(\byear{2011}).
\btitle{Functional and effective connectivity: a review}.
\bjournal{Brain Connectivity}
\bvolume{1}
\bpages{13--36}.
\end{barticle}
\endbibitem

\bibitem[\protect\citeauthoryear{Hall and Horowitz}{2005}]{Hall2005}
\begin{barticle}[author]
\bauthor{\bsnm{Hall},~\bfnm{Peter}\binits{P.}} \AND
  \bauthor{\bsnm{Horowitz},~\bfnm{Joel~L.}\binits{J.~L.}}
(\byear{2005}).
\btitle{Nonparametric methods for inference in the presence of instrumental
  variables}.
\bjournal{The Annals of Statistics}
\bvolume{33}
\bpages{2904--2929}.
\end{barticle}
\endbibitem

\bibitem[\protect\citeauthoryear{Hall and Horowitz}{2007}]{Hall2007c}
\begin{barticle}[author]
\bauthor{\bsnm{Hall},~\bfnm{Peter}\binits{P.}} \AND
  \bauthor{\bsnm{Horowitz},~\bfnm{Joel~L.}\binits{J.~L.}}
(\byear{2007}).
\btitle{Methodology and convergence rates for functional linear regression}.
\bjournal{The Annals of Statistics}
\bvolume{35}
\bpages{70-91}.
\end{barticle}
\endbibitem

\bibitem[\protect\citeauthoryear{Hall and Hosseini-Nasab}{2006}]{Hall2006}
\begin{barticle}[author]
\bauthor{\bsnm{Hall},~\bfnm{P.}\binits{P.}} \AND
  \bauthor{\bsnm{Hosseini-Nasab},~\bfnm{M.}\binits{M.}}
(\byear{2006}).
\btitle{On properties of functional principal components analysis}.
\bjournal{Journal of the Royal Statistical Society: Series B (Statistical
  Methodology)}
\bvolume{68}
\bpages{109--126}.
\end{barticle}
\endbibitem

\bibitem[\protect\citeauthoryear{Happ and Greven}{2018}]{Happ2018}
\begin{barticle}[author]
\bauthor{\bsnm{Happ},~\bfnm{Clara}\binits{C.}} \AND
  \bauthor{\bsnm{Greven},~\bfnm{Sonja}\binits{S.}}
(\byear{2018}).
\btitle{Multivariate functional principal component analysis for data observed
  on different (dimensional) domains}.
\bjournal{Journal of the American Statistical Association}
\bvolume{113}
\bpages{649--659}.
\end{barticle}
\endbibitem

\bibitem[\protect\citeauthoryear{Hastie, Tibshirani and
  Friedman}{2009}]{Hastie2009}
\begin{bbook}[author]
\bauthor{\bsnm{Hastie},~\bfnm{Trevor}\binits{T.}},
  \bauthor{\bsnm{Tibshirani},~\bfnm{Robert}\binits{R.}} \AND
  \bauthor{\bsnm{Friedman},~\bfnm{Jerome}\binits{J.}}
(\byear{2009}).
\btitle{The Elements of Statistical Learning: Data Mining, Inference and
  Prediction},
\bedition{second edition} ed.
\bseries{Springer Series in Statistics}.
\bpublisher{Springer}.
\end{bbook}
\endbibitem

\bibitem[\protect\citeauthoryear{Hsing and Eubank}{2015}]{Hsing2015}
\begin{bbook}[author]
\bauthor{\bsnm{Hsing},~\bfnm{Tailen}\binits{T.}} \AND
  \bauthor{\bsnm{Eubank},~\bfnm{Randall}\binits{R.}}
(\byear{2015}).
\btitle{Theoretical Foundations of Functional Data Analysis, with an
  Introduction to Linear Operators}.
\bpublisher{Wiley}.
\end{bbook}
\endbibitem

\bibitem[\protect\citeauthoryear{Kelly and Root}{1960}]{Kelly1960}
\begin{barticle}[author]
\bauthor{\bsnm{Kelly},~\bfnm{Edward~J.}\binits{E.~J.}} \AND
  \bauthor{\bsnm{Root},~\bfnm{William~L.}\binits{W.~L.}}
(\byear{1960}).
\btitle{A Representation of Vector-Valued Random Processes}.
\bjournal{Studies in Applied Mathematics}
\bvolume{39}
\bpages{211--216}.
\end{barticle}
\endbibitem

\bibitem[\protect\citeauthoryear{Kendall and Le}{2011}]{Kendall2011}
\begin{barticle}[author]
\bauthor{\bsnm{Kendall},~\bfnm{Wilfrid~S.}\binits{W.~S.}} \AND
  \bauthor{\bsnm{Le},~\bfnm{Huiling}\binits{H.}}
(\byear{2011}).
\btitle{Limit theorems for empirical {Fr\'{e}chet} means of independent and
  non-identically distributed manifold-valued random variables}.
\bjournal{Brazilian Journal of Probability and Statistics}
\bvolume{25}
\bpages{323--352}.
\end{barticle}
\endbibitem

\bibitem[\protect\citeauthoryear{Kleffe}{1973}]{Kleffe1973}
\begin{barticle}[author]
\bauthor{\bsnm{Kleffe},~\bfnm{J\"{u}rgen}\binits{J.}}
(\byear{1973}).
\btitle{Principal components of random variables with values in a separable
  {H}ilbert space}.
\bjournal{Statistics: A Journal of Theoretical and Applied Statistics}
\bvolume{4}
\bpages{391--406}.
\end{barticle}
\endbibitem

\bibitem[\protect\citeauthoryear{Kokoszka and Reimherr}{2017}]{Kokoszka2017}
\begin{bbook}[author]
\bauthor{\bsnm{Kokoszka},~\bfnm{Piotr}\binits{P.}} \AND
  \bauthor{\bsnm{Reimherr},~\bfnm{Matthew}\binits{M.}}
(\byear{2017}).
\btitle{Introduction to Functional Data Analysis}.
\bpublisher{Chapman and Hall/CRC}.
\end{bbook}
\endbibitem

\bibitem[\protect\citeauthoryear{Kong et~al.}{2016}]{Kong2016}
\begin{barticle}[author]
\bauthor{\bsnm{Kong},~\bfnm{Dehan}\binits{D.}},
  \bauthor{\bsnm{Xue},~\bfnm{Kaijie}\binits{K.}},
  \bauthor{\bsnm{Yao},~\bfnm{Fang}\binits{F.}} \AND
  \bauthor{\bsnm{Zhang},~\bfnm{Hao~H.}\binits{H.~H.}}
(\byear{2016}).
\btitle{Partially functional linear regression in high dimensions}.
\bjournal{Biometrika}
\bvolume{103}
\bpages{147--159}.
\end{barticle}
\endbibitem

\bibitem[\protect\citeauthoryear{Lang}{1995}]{Lang1995}
\begin{bbook}[author]
\bauthor{\bsnm{Lang},~\bfnm{Serge}\binits{S.}}
(\byear{1995}).
\btitle{Differential and Riemannian Manifolds}.
\bpublisher{Springer}, \baddress{New York}.
\end{bbook}
\endbibitem

\bibitem[\protect\citeauthoryear{Lang}{1999}]{Lang1999}
\begin{bbook}[author]
\bauthor{\bsnm{Lang},~\bfnm{Serge}\binits{S.}}
(\byear{1999}).
\btitle{Fundamentals of Differential Geometry}.
\bpublisher{Springer}, \baddress{New York}.
\end{bbook}
\endbibitem

\bibitem[\protect\citeauthoryear{Lee}{1997}]{Lee1997}
\begin{bbook}[author]
\bauthor{\bsnm{Lee},~\bfnm{J.~M.}\binits{J.~M.}}
(\byear{1997}).
\btitle{Riemannian Manifolds: An Introduction to Curvature}.
\bpublisher{Springer-Verlag}, \baddress{New York}.
\end{bbook}
\endbibitem

\bibitem[\protect\citeauthoryear{Lee}{2002}]{Lee2002}
\begin{bbook}[author]
\bauthor{\bsnm{Lee},~\bfnm{John~M.}\binits{J.~M.}}
(\byear{2002}).
\btitle{Introduction to Smooth Manifolds}.
\bpublisher{Springer}, \baddress{New York}.
\end{bbook}
\endbibitem

\bibitem[\protect\citeauthoryear{Lila and Aston}{2017}]{Lila2017}
\begin{barticle}[author]
\bauthor{\bsnm{Lila},~\bfnm{E.}\binits{E.}} \AND
  \bauthor{\bsnm{Aston},~\bfnm{JAD}\binits{J.}}
(\byear{2017}).
\btitle{Smooth Principal Component Analysis over two-dimensional manifolds with
  an application to Neuroimaging}.
\bjournal{The Annals of Applied Statistics}
\bvolume{10}
\bpages{1854--1879}.
\end{barticle}
\endbibitem

\bibitem[\protect\citeauthoryear{Lin and Yao}{2017}]{Lin2017a}
\begin{barticle}[author]
\bauthor{\bsnm{Lin},~\bfnm{Zhenhua}\binits{Z.}} \AND
  \bauthor{\bsnm{Yao},~\bfnm{Fang}\binits{F.}}
(\byear{2017}).
\btitle{Functional Regression with Unknown Manifold Structures}.
\bjournal{arXiv}.
\end{barticle}
\endbibitem

\bibitem[\protect\citeauthoryear{Moakher}{2005}]{Moakher2005}
\begin{barticle}[author]
\bauthor{\bsnm{Moakher},~\bfnm{Maher}\binits{M.}}
(\byear{2005}).
\btitle{A differential geometry approach to the geometric mean of symmetric
  positive-definite matrices}.
\bjournal{SIAM Journal on Matrix Analysis and Applications}
\bvolume{26}
\bpages{735--747}.
\end{barticle}
\endbibitem

\bibitem[\protect\citeauthoryear{Park et~al.}{2017}]{Park2017}
\begin{barticle}[author]
\bauthor{\bsnm{Park},~\bfnm{J.~E.}\binits{J.~E.}},
  \bauthor{\bsnm{Jung},~\bfnm{S.~C.}\binits{S.~C.}},
  \bauthor{\bsnm{Ryu},~\bfnm{K.~H.}\binits{K.~H.}},
  \bauthor{\bsnm{Oh},~\bfnm{J.~Y.}\binits{J.~Y.}},
  \bauthor{\bsnm{Kim},~\bfnm{H.~S.}\binits{H.~S.}},
  \bauthor{\bsnm{Choi},~\bfnm{C.~G.}\binits{C.~G.}},
  \bauthor{\bsnm{Kim},~\bfnm{S.~J.}\binits{S.~J.}} \AND
  \bauthor{\bsnm{Shim},~\bfnm{W.~H.}\binits{W.~H.}}
(\byear{2017}).
\btitle{Differences in dynamic and static functional connectivity between young
  and elderly healthy adults}.
\bjournal{Neuroradiology}
\bvolume{59}
\bpages{781--789}.
\end{barticle}
\endbibitem

\bibitem[\protect\citeauthoryear{Petersen and M\"{u}ller}{2017}]{Petersen2017}
\begin{barticle}[author]
\bauthor{\bsnm{Petersen},~\bfnm{Alexander}\binits{A.}} \AND
  \bauthor{\bsnm{M\"{u}ller},~\bfnm{Hans-Georg}\binits{H.-G.}}
(\byear{2017}).
\btitle{Fr\'{e}chet Regression for Random Objects with {E}uclidean Predictors}.
\bjournal{arXiv}.
\end{barticle}
\endbibitem

\bibitem[\protect\citeauthoryear{Phana et~al.}{2002}]{Phana2002}
\begin{barticle}[author]
\bauthor{\bsnm{Phana},~\bfnm{K.~Luan}\binits{K.~L.}},
  \bauthor{\bsnm{Wager},~\bfnm{Tor}\binits{T.}},
  \bauthor{\bsnm{Taylor},~\bfnm{Stephan~F.}\binits{S.~F.}} \AND
  \bauthor{\bsnm{Liberzon},~\bfnm{Israel}\binits{I.}}
(\byear{2002}).
\btitle{Functional Neuroanatomy of Emotion: A Meta-Analysis of Emotion
  Activation Studies in PET and fMRI}.
\bjournal{NeuroImage}
\bvolume{16}
\bpages{331--348}.
\end{barticle}
\endbibitem

\bibitem[\protect\citeauthoryear{Raichlen et~al.}{2016}]{Raichlen2016}
\begin{barticle}[author]
\bauthor{\bsnm{Raichlen},~\bfnm{David~A.}\binits{D.~A.}},
  \bauthor{\bsnm{Bharadwaj},~\bfnm{Pradyumna~K.}\binits{P.~K.}},
  \bauthor{\bsnm{Fitzhugh},~\bfnm{Megan~C.}\binits{M.~C.}},
  \bauthor{\bsnm{Haws},~\bfnm{Kari~A.}\binits{K.~A.}},
  \bauthor{\bsnm{Torre},~\bfnm{Gabrielle-Ann}\binits{G.-A.}},
  \bauthor{\bsnm{Trouard},~\bfnm{Theodore~P.}\binits{T.~P.}} \AND
  \bauthor{\bsnm{Alexander},~\bfnm{Gene~E.}\binits{G.~E.}}
(\byear{2016}).
\btitle{Differences in Resting State Functional Connectivity between Young
  Adult Endurance Athletes and Healthy Controls}.
\bjournal{Frontiers in Human Neuroscience}
\bvolume{10}.
\end{barticle}
\endbibitem

\bibitem[\protect\citeauthoryear{Ramsay and Silverman}{2005}]{Ramsay2005}
\begin{bbook}[author]
\bauthor{\bsnm{Ramsay},~\bfnm{J.~O.}\binits{J.~O.}} \AND
  \bauthor{\bsnm{Silverman},~\bfnm{B.~W.}\binits{B.~W.}}
(\byear{2005}).
\btitle{Functional Data Analysis},
\bedition{2nd} ed.
\bseries{Springer Series in Statistics}.
\bpublisher{Springer}, \baddress{New York}.
\end{bbook}
\endbibitem

\bibitem[\protect\citeauthoryear{Rao}{1958}]{Rao1958}
\begin{barticle}[author]
\bauthor{\bsnm{Rao},~\bfnm{C.~R.}\binits{C.~R.}}
(\byear{1958}).
\btitle{Some Statistical Methods for Comparison of Growth Curves}.
\bjournal{Biometrics}
\bvolume{14}
\bpages{1--17}.
\end{barticle}
\endbibitem

\bibitem[\protect\citeauthoryear{Salehian et~al.}{2015}]{Salehian2015}
\begin{binproceedings}[author]
\bauthor{\bsnm{Salehian},~\bfnm{Hesamoddin}\binits{H.}},
  \bauthor{\bsnm{Chakraborty},~\bfnm{Rudrasis}\binits{R.}},
  \bauthor{\bsnm{Ofori},~\bfnm{Edward}\binits{E.}},
  \bauthor{\bsnm{Vaillancourt},~\bfnm{David}\binits{D.}} \AND
  \bauthor{\bsnm{Vemuri},~\bfnm{Baba~C.}\binits{B.~C.}}
(\byear{2015}).
\btitle{An efficient recursive estimator of the {F}r\'{e}chet mean on
  hypersphere with applications to Medical Image Analysis}
In \bbooktitle{5th MICCAI workshop on Mathematical Foundations of Computational
  Anatomy (MFCA)}.
\end{binproceedings}
\endbibitem

\bibitem[\protect\citeauthoryear{Sasaki}{1958}]{Sasaki1958}
\begin{barticle}[author]
\bauthor{\bsnm{Sasaki},~\bfnm{Shigeo}\binits{S.}}
(\byear{1958}).
\btitle{On the differential geometry of tangent bundles of Riemannian
  manifolds}.
\bjournal{Tohoku Mathematical Journal}
\bvolume{10}
\bpages{338--354}.
\end{barticle}
\endbibitem

\bibitem[\protect\citeauthoryear{Shi et~al.}{2009}]{Shi2009}
\begin{binproceedings}[author]
\bauthor{\bsnm{Shi},~\bfnm{Xiaoyan}\binits{X.}},
  \bauthor{\bsnm{Styner},~\bfnm{Martin}\binits{M.}},
  \bauthor{\bsnm{Lieberman},~\bfnm{Jeffrey}\binits{J.}},
  \bauthor{\bsnm{Ibrahim},~\bfnm{Joseph~G.}\binits{J.~G.}},
  \bauthor{\bsnm{Lin},~\bfnm{Weili}\binits{W.}} \AND
  \bauthor{\bsnm{Zhu},~\bfnm{Hongtu}\binits{H.}}
(\byear{2009}).
\btitle{Intrinsic Regression Models for Manifold-Valued Data}.
In \bbooktitle{Medical Image Computing and Computer-Assisted Intervention -
  MICCAI}
\bvolume{12}
\bpages{192--199}.
\end{binproceedings}
\endbibitem

\bibitem[\protect\citeauthoryear{Silverman}{1996}]{Silverman1996}
\begin{barticle}[author]
\bauthor{\bsnm{Silverman},~\bfnm{Bernard~W.}\binits{B.~W.}}
(\byear{1996}).
\btitle{Smoothed Functional Principal Components Analysis by Choice of Norm}.
\bjournal{The Annals of Statistics}
\bvolume{24}
\bpages{1--24}.
\end{barticle}
\endbibitem

\bibitem[\protect\citeauthoryear{Steinke, Hein and
  Sch\"{o}lkopf}{2010}]{Steinke2010}
\begin{barticle}[author]
\bauthor{\bsnm{Steinke},~\bfnm{Florian}\binits{F.}},
  \bauthor{\bsnm{Hein},~\bfnm{Matthias}\binits{M.}} \AND
  \bauthor{\bsnm{Sch\"{o}lkopf},~\bfnm{Bernhard}\binits{B.}}
(\byear{2010}).
\btitle{Nonparametric Regression between General {R}iemannian Manifolds}.
\bjournal{SIAM Journal on Imaging Sciences}
\bvolume{3}
\bpages{527--563}.
\end{barticle}
\endbibitem

\bibitem[\protect\citeauthoryear{Wang}{2008}]{Wang2008}
\begin{bphdthesis}[author]
\bauthor{\bsnm{Wang},~\bfnm{Limin}\binits{L.}}
(\byear{2008}).
\btitle{{K}arhunen-{L}oeve expansions and their applications}
\btype{PhD thesis},
\bpublisher{The London School of Economics and Political Science}.
\end{bphdthesis}
\endbibitem

\bibitem[\protect\citeauthoryear{Wang, Chiou and M\"{u}ller}{2016}]{Wang2016}
\begin{barticle}[author]
\bauthor{\bsnm{Wang},~\bfnm{Jane-Ling}\binits{J.-L.}},
  \bauthor{\bsnm{Chiou},~\bfnm{Jeng-Min}\binits{J.-M.}} \AND
  \bauthor{\bsnm{M\"{u}ller},~\bfnm{Hans-Georg}\binits{H.-G.}}
(\byear{2016}).
\btitle{Review of functional data analysis}.
\bjournal{Annual Review of Statistics and Its Application}
\bvolume{3}
\bpages{257--295}.
\end{barticle}
\endbibitem

\bibitem[\protect\citeauthoryear{Wang and Marron}{2007}]{Wang2007}
\begin{barticle}[author]
\bauthor{\bsnm{Wang},~\bfnm{Haonan}\binits{H.}} \AND
  \bauthor{\bsnm{Marron},~\bfnm{J.~S.}\binits{J.~S.}}
(\byear{2007}).
\btitle{Object Oriented Data Analysis: Sets of Trees}.
\bjournal{The Annals of Statistics}
\bvolume{35}
\bpages{1849--1873}.
\end{barticle}
\endbibitem

\bibitem[\protect\citeauthoryear{Yao, M\"uller and Wang}{2005a}]{Yao2005a}
\begin{barticle}[author]
\bauthor{\bsnm{Yao},~\bfnm{Fang}\binits{F.}},
  \bauthor{\bsnm{M\"uller},~\bfnm{Hans-Georg}\binits{H.-G.}} \AND
  \bauthor{\bsnm{Wang},~\bfnm{Jane-Ling}\binits{J.-L.}}
(\byear{2005}a).
\btitle{Functional Data Analysis for Sparse Longitudinal Data}.
\bjournal{Journal of the American Statistical Association}
\bvolume{100}
\bpages{577--590}.
\end{barticle}
\endbibitem

\bibitem[\protect\citeauthoryear{Yao, M\"{u}ller and Wang}{2005b}]{Yao2005b}
\begin{barticle}[author]
\bauthor{\bsnm{Yao},~\bfnm{Fang}\binits{F.}},
  \bauthor{\bsnm{M\"{u}ller},~\bfnm{Hans-Georg}\binits{H.-G.}} \AND
  \bauthor{\bsnm{Wang},~\bfnm{Jane-Ling}\binits{J.-L.}}
(\byear{2005}b).
\btitle{Functional linear regression analysis for longitudinal data}.
\bjournal{The Annals of Statistics}
\bvolume{33}
\bpages{2873--2903}.
\end{barticle}
\endbibitem

\bibitem[\protect\citeauthoryear{Yuan and Cai}{2010}]{Yuan2010}
\begin{barticle}[author]
\bauthor{\bsnm{Yuan},~\bfnm{Ming}\binits{M.}} \AND
  \bauthor{\bsnm{Cai},~\bfnm{T.~Tony}\binits{T.~T.}}
(\byear{2010}).
\btitle{A reproducing kernel {Hilbert} space approach to functional linear
  regression}.
\bjournal{The Annals of Statistics}
\bvolume{38}
\bpages{3412--3444}.
\end{barticle}
\endbibitem

\bibitem[\protect\citeauthoryear{Yuan et~al.}{2012}]{Yuan2012}
\begin{barticle}[author]
\bauthor{\bsnm{Yuan},~\bfnm{Ying}\binits{Y.}},
  \bauthor{\bsnm{Zhu},~\bfnm{Hongtu}\binits{H.}},
  \bauthor{\bsnm{Lin},~\bfnm{Weili}\binits{W.}} \AND
  \bauthor{\bsnm{Marron},~\bfnm{J.~S.}\binits{J.~S.}}
(\byear{2012}).
\btitle{Local Polynomial Regression for Symmetric Positive Definite Matrices}.
\bjournal{Journal of Royal Statistical Society: Series B (Statistical
  Methodology)}
\bvolume{74}
\bpages{697--719}.
\end{barticle}
\endbibitem

\bibitem[\protect\citeauthoryear{Zhang and Wang}{2016}]{Zhang2016}
\begin{barticle}[author]
\bauthor{\bsnm{Zhang},~\bfnm{X.}\binits{X.}} \AND
  \bauthor{\bsnm{Wang},~\bfnm{J.~L.}\binits{J.~L.}}
(\byear{2016}).
\btitle{From sparse to dense functional data and beyond}.
\bjournal{The Annals of Statistics}
\bvolume{44}
\bpages{2281--2321}.
\end{barticle}
\endbibitem

\end{thebibliography}
\end{document}